\documentclass[12pt]{amsart}

\usepackage{tgpagella}
\usepackage{euler}
\usepackage[T1]{fontenc}
\usepackage{amsmath, amssymb, mathabx}
\usepackage[hidelinks]{hyperref}
\usepackage[english]{babel}
\usepackage{mathrsfs}
\usepackage{eucal}
\usepackage[all]{xy}
\usepackage{tikz}

\newtheorem{thm}{Theorem}[section]
\newtheorem*{thm*}{Theorem}
\newtheorem{lem}[thm]{Lemma}

\newtheorem{prop}[thm]{Proposition}
\newtheorem*{prop*}{Proposition}

\newtheorem{cor}[thm]{Corollary}
\newtheorem*{cor*}{Corollary}

\theoremstyle{definition}
\newtheorem{defn}[thm]{Definition}
\newtheorem*{defn*}{Definition}

\newtheorem*{question*}{Question}
\newtheorem*{Pquestion*}{Popa's question}

\newtheorem*{conv*}{Convention}

\newcommand{\dminus}{ 
\buildrel\textstyle\ .\over{\hbox{ 
\vrule height3pt depth0pt width0pt}{\smash-} 
}}

\def\bb{\mathbb}

\def\bb{\mathbb}

\def\cal{\mathcal}

\def\u{\mathsf 1}

\newcommand{\cstar}{$\mathrm{C}^*$}

\makeatletter

\def\dotminussym#1#2{%
  \setbox0=\hbox{$\m@th#1-$}%
  \kern.5\wd0%
  \hbox to 0pt{\hss\hbox{$\m@th#1-$}\hss}%
  \raise.6\ht0\hbox to 0pt{\hss$\m@th#1.$\hss}%
  \kern.5\wd0}
\newcommand{\dotminus}{\mathbin{\mathpalette\dotminussym{}}}

\DeclareMathOperator{\tr}{tr}

\def \Th{\operatorname{Th}}
\def \R{\mathcal R}
\def \u{\mathcal U}

\def \val{\operatorname{val}}

\newcommand{\mip}{\operatorname{MIP}}

\newcommand{\cqa}{C_{qa}}
\newcommand{\cqc}{C_{qc}}

\def\P{\operatorname{P}}
\def\NP{\operatorname{NP}}
\def\IP{\operatorname{IP}}
\def\MIP{\operatorname{MIP}}
\def\EXP{\operatorname{EXP}}
\def\NEXP{\operatorname{NEXP}}
\def\PSPACE{\operatorname{PSPACE}}
\def\NEEXP{\operatorname{NEEXP}}

\def\BPP{\operatorname{BPP}}
\def\r{\bb R}
\def\n{\bb N}
\def\sval{\operatorname{sval}}
\def\A{\cal A}

\allowdisplaybreaks
\makeatletter
\def\l@subsection{\@tocline{2}{0pt}{2.5pc}{5pc}{}}
\def\l@subsubsection{\@tocline{2}{0pt}{5pc}{7.5pc}{}}
\makeatother

\begin{document}


\title{The Connes Embedding Problem:  A guided tour}

\author{Isaac Goldbring}
\address{Department of Mathematics\\University of California, Irvine, 340 Rowland Hall (Bldg.\# 400),
Irvine, CA 92697-3875}
\email{isaac@math.uci.edu}
\urladdr{http://www.math.uci.edu/~isaac}
\thanks{Goldbring was partially supported by NSF grant DMS-2054477.}

\maketitle

\begin{abstract}
The Connes Embedding Problem (CEP) is a problem in the theory of tracial von Neumann algebras and asks whether or not every tracial von Neumann algebra embeds into an ultrapower of the hyperfinite II$_1$ factor.  The CEP has had interactions with a wide variety of areas of mathematics, including \cstar-algebra theory, geometric group theory, free probability, and noncommutative real algebraic geometry (to name a few).  After remaining open for over 40 years, a negative solution was recently obtained as a corollary of a landmark result in quantum complexity theory known as $\mip^*=\operatorname{RE}$.  In these notes, we introduce all of the background material necessary to understand the proof of the negative solution of the CEP from $\mip^*=\operatorname{RE}$.  In fact, we outline two such proofs, one following the ``traditional'' route that goes via Kirchberg's QWEP problem in \cstar-algebra theory and Tsirelson's problem in quantum information theory and a second that uses basic ideas from logic. 
\end{abstract}

\tableofcontents
\section{Introduction}

\subsection{What is this all about?}

The story told in this tour is (in this author's humble opinion) absolutely fascinating!  It can also be completely confusing and terrifying to an outsider.  It contains a seemingly infinite number of acronyms (CEP, WEP, QWEP, LLP, MIP*, RE,...), all sorts of tensor products $(\bar \otimes, \otimes_{\max}, \otimes_{\min})$, entangled particles, and even good friends Einstein and G\"odel both make an appearance (the latter twice).

At one end of the story is the \emph{Connes embedding problem} (CEP), a problem in the field of \emph{von Neumann algebras} first posed by Alain Connes in his famous 1976 paper ``Classification of Injective Factors'' \cite{Connes} (the paper mainly responsible for his being awarded the Fields Medal in 1982).  Roughly speaking, a von Neumann algebra is a collection of bounded operators on a Hilbert space containing the identity operator, closed under addition, composition, scalar multiplication, and adjoint, and which is closed in a certain topology known as the weak operator topology.  The von Neumann algebras Connes was considering came equipped with a trace functional that shares many of the nice properties enjoyed by the (normalized) trace functional on matrices.  

Here is the passage from \cite{Connes} which led to the establishment of the CEP:  

``We now construct an approximate imbedding of $N$ in $\cal R$.  Apparently such an imbedding ought to exist for all II$_1$ factors because it does for the regular representation of free groups.  However, the construction below relies on condition 6.''

What is this quote trying to convey?  $\cal R$ is the \emph{hyperfinite II$_1$ factor}, arguably the most important tracial von Neumann algebra.  For now, one should just think of $\cal R$ as an appropriate limit of matrix algebras $M_n(\bb C)$ of increasing sizes.  We will have much to say about this algebra throughout this paper.  A II$_1$ factor is just a particular kind of tracial von Neumann algebra and the $N$ appearing in the passage is a particular II$_1$ factor satisfying a certain list of properties.  By an approximate imbedding of $N$ in $\cal R$, Connes means that any finite amount of ``information'' about elements of $N$ (that is, the trace of finitely many $*$-polynomials with elements from $N$ plugged in) can be ``simulated'' by appropriate elements of $\cal R$.  Connes later shows that such approximate imbeddings correspond to actual embeddings of $N$ into a so-called \emph{ultrapower} of $\cal R$, denoted $\cal R^\u$.  He comments that such an embedding ``ought'' to always exist since it does for a particular von Neumann algebra, namely the group von Neumann algebra associated to the free group, denoted $L(\bb F_2)$ (see Subsection \ref{sec3:operator} below).  Why that ``ought to be'' is not quite clear.  Nevertheless, Connes is only able to show that the $N$ under consideration can be embedded in $\cal R^\u$ using one of the conditions (namely the sixth one) he has assumed about this particular algebra.

Thus, the Connes Embedding Problem (CEP) states:  every tracial von Neumann algebra embeds (in a trace-preserving way) in an ultrapower $\cal R^\u$ of $\cal R$.  We will say this slightly more precisely in Subsection \ref{sec3:tracial} below.  Many prefer to call this a ``Problem'' rather than a ``Conjecture'' since ``ought to'' is not a very strong sentiment.

The robustness of the CEP lies in its many reformulations and from the many areas of mathematics it has touched upon; see Section \ref{sec2} for some examples.

At the other end of this story (and seemingly a world far, far away), is a landmark theorem in \emph{quantum complexity theory} known as $\MIP^*=\operatorname{RE}$ \cite{MIP*}.  Like most theorems in complexity theory, it compares two \emph{complexity classes}.  Roughly speaking, a complexity class consists of a collection of ``problems'' that all share some common level of ``difficulty'' with which one can solve or verify these problems.  The class $\operatorname{RE}$ denotes those problems for which there is a computer program so that, if you left the program running long enough, would list all instances for which the problem has a positive answer (but you would never known about instances with a negative answer).  Usually complexity theorists are more interested in levels of efficiency and the class $\operatorname{RE}$ is hardly ever discussed.  The other complexity class in the above equation is $\MIP^*$, which denotes those problems for which a ``verifier'' interacting with multiple cooperating (but noncommunicating) ``provers'' who share a source of quantum entanglement can reliably verify a positive instance of a problem.  The result $\MIP^*=\operatorname{RE}$ states that these two classes coincide!  This is a monumental result for it shows the power of quantum ideas in computational complexity.  One particular instance of this result is that the (in)famous \emph{halting problem}, which asks if a particular computer program will halt on, say, the empty input, which is known to be an undecidable problem, can actually be efficiently and reliably verified by two provers sharing some quantum entanglement; here efficiently means in polynomial time and reliably means that if the machine halts, then the verifier will accept the provers' proof of that fact with probability $1$, while if it does not halt, then only half the time will they accept a proof of halting when it in fact should not.  (An execution of the protocol has a probabilistic outcome, whence here the condition is that there is acceptance with probability at most $\frac{1}{2}$ over the verifier's and provers' random choices in the case of a 
Turing machine that does not halt.) This is an astounding result!

Even more amazing than the sheer statement of the result is that the equality $\mip^*=\operatorname{RE}$ actually yields a negative solution to CEP!  

\subsection{Connecting the dots}

But how could these seemingly unrelated topics be so tightly connected?  The answer lies through a series of previous known connections.  First, in a fundamental paper of Kirchberg \cite{K}, it was shown that CEP is equivalent to an important problem (Kirchberg even used the word conjecture) in the theory of \cstar-algebras stemming from the complexity of \cstar-tensor products known now as Kirchberg's QWEP problem (see Subsection \ref{sec3:problem} below).  Later, Fritz \cite{Fr} and independently Junge et. al. \cite{Junge} demonstrated that a positive answer to the QWEP problem would yield a positive answer to a problem in quantum information theory known as \emph{Tsirelson's problem} which, roughly speaking, ask whether the usual quantum mechanical framework and that coming from quantum field theory yield the same set of quantum correlations corresponding to Bell experiments.  While the jump from the QWEP problem to Tsirelson's problem might seem like quite a leap, once one unravels the definitions, this is actually a fairly straightforward argument and will be given in Subsection \ref{sec6:tsirelson} below.  Both sets of authors almost proved that the Kirchberg and Tsirelson problems were actually equivalent; Ozawa succeded in connecting the last dots in \cite{Oz}.  

Now we are at least in the same arena:  quantum information theory and quantum complexity theory (both area at least have ``quantum'' in their names).  The last step in the puzzle is to use a result of Fritz, Netzer, and Thom \cite{FNT} about the computability of the operator norm for universal group \cstar-algebras and the analysis leading to the equivalence of QWEP and Tsirelson to show that if Tsirelson's problem has a positive answer, then every language in $\mip^*$ would actually be decidable, contradicting $\mip^*=\operatorname{RE}$.  (See Subsection \ref{sec6:tsirelson} below for the complete argument.)

Okay, so that was a mouthful!

\subsection{Why another treatment of CEP?}

Numerous accounts of the CEP and its many equivalents can be found in the literature.  In fact, Pisier \cite{pisier} recently wrote a fascinating account (coming in just shy of 500 pages) on the CEP and its equivalences with QWEP and Tsirelson (and so, so much more).  Much trimmer accounts were given by Ozawa \cite{Oz,OzQ} and Capraro and Lupini \cite{CapLup}.

So if there are so many accounts of the CEP, why write another?  We have several good reasons:

First, all of the above accounts were written pre-$\mip^*=\operatorname{RE}$, so none of them actually explain how the story resolves itself.  

Second, all of the above accounts go into an extreme amount of detail and assume a fair amount of background knowledge in operator algebras.  We envision the reader in, say quantum physics or complexity theory, wanting to understand the main thread of the story and being overburdened by the overhead needed to enter the fray.  In this survey, we try very hard to at least state all of the necessary definitions.  On the other hand, we offer very little details or proofs in the interest of space and refer the reader to the above references if they are interested in the gritty details.  Also, since we are focusing on the one-way implications (as opposed to the equivalences the other accounts present), we save ourselves some complications.

Third, the operator algebra community may know very little quantum theory or complexity theory, so we offer brief introductions to these areas to at least paint the picture for them.

Finally, and most certainly gratuitously, we offer an ``alternative'' and, in this author's biased opinion, ``simpler'' path from $\mip^*=\operatorname{RE}$ to the failure of CEP than that outlined above using basic methods from mathematical logic.  This path also offers some extra bells and whistles to the failure of CEP, including a G\"odelian refutation of the CEP and a proof of the existence of ``many'' counterexamples to the CEP.  While we have our logician hats on, we take advantage of the fact that we have the readers' attention to describe a model-theoretic weakening of the CEP that is still open and quite fascinating (at least to us!).

\subsection{A quick guide to this guide}  In Section 2, we briefly describe some of the known equivalents of the CEP.  The reader may benefit from coming back to this section after having read some of the definitions, but this is supposed to whet the reader's appetite and convince them that the rest of the paper is worth reading.

Section 3 is a crash course in operator algebras, assuming some basic functional analysis that someone in quantum physics should probably be familiar with.  We cover both the \cstar and von Neumann algebra background needed as well as topics such as states and traces, the ultrapower construction, operator algebras arising from groups, and finally, what is so darn complicated about \cstar-algebra tensor products, culiminating in a discussion of why a positive solution to the CEP implies a positive solution to the QWEP problem.

Section 4 is a similar crash course, but this time in complexity theory.  We start from the definition of Turing machines, defining some of the basic complexity classes, and then work our way up to the class $\mip$ of languages verifiable by a verifier interacting with multiple cooperating provers.  

In Section 5, we make a quantum detour for those unfamiliar with the basic tenets of quantum mechanics and even take a digression on superdense coding just for fun (and to indicate the power of entanglement).  This section culminates with the definition of the complexity class $\mip^*$, the analog of $\mip$ where the provers are allowed to share quantum entanglement as a resource, and the precise statement of the result $\mip^*=\operatorname{RE}$.

Section 6 contains the details of the proof of the failure of the QWEP problem from $\mip^*=\operatorname{RE}$ by first showing how the latter yields a negative solution to Tsirelson's problem and then by establishing how a negative solution to Tsirelson's problem yields a negative solution to the QWEP problem.  Combined with our derivation of a positive solution of the QWEP problem from a positive solution to the CEP, this completes the proof of the negative solution to the CEP from $\mip^*=\operatorname{RE}$.

Section 7 offers the alternative proof alluded to above using basic ideas from logic.  We present the appropriate logic for studying tracial von Neumann algebras and discuss the main contribution from logic, namely G\"odel's Completeness Theorem.  We also describe the extra information about the CEP gleamed from the logical perspective mentioned above, including a completely operator-algebraic reformulation of our main model-theoretic contribution in terms of the undecidability of a certain ``moment approximation problem.''  We also offer an alternative proof of the failure of Tsirelson from $\mip^*=\operatorname{RE}$ using the Completeness Theorem.  Most of the material in this section represents joint work with Bradd Hart \cite{GH,GH2}.

Finally, in Section 8, we discuss the open problem around the existence of the so-called \emph{enforceable factor}, which is the model-theoretic weakening of the CEP referred to above.

\subsection{Acknowledgements}

We would like to thank the following people for giving us helpful comments and/or corrections regarding earlier versions of this manuscript:  Alec Fox, Jeffrey Barrett, Michael Cranston, Vern Paulsen, Jennifer Pi, Thomas Vidick, and Henry Yuen.

\section{Equivalent reformulations of CEP}\label{sec2}

One of the aspects of the CEP that makes it such an interesting problem is its numerous equivalences spanning many seemingly different areas of mathematics.  In the main text, the equivalences with Kirchberg's QWEP conjecture in \cstar-algebra theory and Tsirelson's problem in quantum information theory will be expounded on in more detail due to their relevance to the current story.  In this section, we briefly mention some of the other well-known equivalences:

\subsection{Free probability theory}\label{sec2:free}  In free probability theory, one considers ``noncommutative'' probability spaces, such as tracial von Neumann algebras $(M,\tau)$, where the elements of $M$ act as noncommutative random variables and the trace $\tau$ is the analog of the integral.  Voiculescu demonstrated the robustness of this theory, establishing free analogues of many familiar facts from ordinary probability theory and giving applications to operator algebras and random matrices (to name a few).  A nice introduction to free probability is Speicher's lecture notes \cite{speicher}.

In classical probability theory, the entropy of a random variable is an important numerical value measuring the amount of information obtained when measuring the random variable.  One method of calculating the entropy of a discrete random variable with probability distriubtion $\{p_1,\ldots,p_n\}$ is to approximate the distribution using ``microstates,'' which are functions $f:\{1,\ldots,N\}\to \{1,\ldots,n\}$ for which the fraction of $j\in \{1,\ldots,N\}$ for which $f(j)=k$ is within $\epsilon$ of $p_k$ for all $k=1,\ldots,n$.  By taking the logarithm of the number of such functions divided by $N$ for a given pair $(N,\epsilon)$ of parameters and then letting $N\to \infty$ and $\epsilon\to 0$, we obtain the entropy $H(p_1,\ldots,p_n)$ of the distribution.  A more general version of this works for a wider class of random variables.

When faced with the task of defining the free entropy of a tuple $(a_1,\ldots,a_n)$ of self-adjoint elements in a tracial von Neumann algebra $(M,\tau)$, Voiculescu proceeds analogously by considering those tuples $(A_1,\ldots,A_n)$ of self-adjoint matrices in some matrix algebra $M_k(\bb C)$ for which a certain finite number of ``moments'' approximate the corresponding moments in the tracial von Neumann algebra, that is, $\tau(p(a_1,\ldots,a_n))$ and $\tr(p(A_1,\ldots,A_n))$ differ by at most $\epsilon$ for finitely many noncommutative $*$-polynomials $p(X_1,\ldots,X_n)$ in $n$-variables.  Now one has to calculate the volume of the set of those matrices and let the various parameters involved tend to infinity or $0$.  With this definition of free entropy, one can prove a number of results which are the ``free'' analog of the corresponding result in the classical theory.  For example, it is known that a tuple of classical random variables has maximal entropy if and only if they are independent and have Gaussian distribution.  In the free theory, the free entropy of a tuple is maximal if and only if the elements of the tuple are freely independent and have ``semicircular distributions'' (which are known to be the free analog of the Gaussian distribution).  The paper \cite{voic} is a survey of free entropy by Voiculescu himself.

This definition of free entropy leads to an interesting feature:  if there are no such tuples $(A_1,\ldots,A_n)$ that ``simulate'' $(a_1,\ldots,a_n)$, then the free entropy of $(a_1,\ldots,a_n)$ equals $-\infty$.  It is well-known (see Subsection \ref{sec3:tracial} below) that, for a given tracial von Neumann algebra $(M,\tau)$, the set of such moments is nonempty for all such tuples $(a_1,\ldots,a_n)$ from $M$ if and only if $M$ embeds into $\cal R^\u$ (in a trace-preserving way).  Thus, CEP is equivalent to all tuples of self-adjoint elements in tracial von Neumann algebras having nonnegative free entropy.

\

\subsection{Hyperlinear groups}\label{sec2:hyperlinear}  Okay, so this one really is not an equivalence, but rather an equivalence with a \emph{special case} of the CEP.  An important notion in group theory is that of a \emph{sofic group}.  Roughly speaking, a countable discrete group $G$ is sofic if, for every finite subset $F$ of $G$, there is a symmetric group $S_n$ and a function $\phi:F\to S_n$ that is an ``approximately injective approximate homomorphism''.  For example, if $g,h,gh\in F$, then one would like to say that $\phi(gh)$ is close to $\phi(g)\phi(h)$, where closeness is measured with respect to the normalized Hamming distance between permutations (which calculates what fraction of elements the permutations disagree on).  The importance of this class of groups is that many important conjectures in group theory are known to hold when restricted to the class of sofic groups.  Surprisingly, there is no known example of a non-sofic group!  One can make a similar definition, replacing symmetric groups $S_n$ with unitary groups $U_n$, equipped with their normalized Hilbert-Schmidt metric; the resulting class of groups is called the class of \emph{hyperlinear groups}.  Every sofic group is hyperlinear and since we do not know if every group is sofic, we do not know if this inclusion is proper.  Moreover, there is no known example of a non-hyperlinear group.  We refer the reader to \cite{CapLup} for more information on sofic and hyperlinear groups.

The connection with CEP comes via an observation of Radulescu \cite{radulescu}, who showed that $G$ is hyperlinear if and only if the group von Neumann algebra $L(G)$ of $G$ (see Subsection \ref{sec3:operator} below) embeds into $\R^\u$.  In other words, if CEP is true just for group von Neumann algebras, then every group is hyperlinear!

Interestingly enough, even though we now know that CEP is false, we still do not know if its special case for group von Neumann algebras holds, that is, we still do not know if every group is hyperlinear.

\

\subsection{Embeddability of general von Neumann algebras}\label{sec2:embeddability}  The CEP is about tracial von Neumann algebras.  But there is a much wider class of von Neumann algebras out there.  Is there a reformulation of the CEP that addresses them?  The answer is yes and was established by Ando, Haagerup, and Winslow in \cite{AHW}.  There is a so-called type III (in the sense of Subsection \ref{sec3:more} below) version of $\cal R$, called the \emph{Araki Woods factor} $\cal R_\infty$, which is the unique hyperfinite type III$_1$ factor.  Moreover, there is a generalization of the tracial ultraproduct construction, known as the \emph{Ocneanu ultraproduct}, that covers the much larger class of $\sigma$-finite von Neumann algebras, of which $\cal R_\infty$ is one of them.  The main result of \cite{AHW} states that CEP is equivalent to the assertion that \emph{every} separably acting von Neumann algebra embeds \emph{with expectation} into the Ocneanu ultrapower $\cal R_\infty^\u$.  The notion of an embedding with expectation is defined in Subsection \ref{sec3:kirchberg's} below.  In the case of tracial von Neumann algebras, the embedding is automatically with expectation, but in the general case, it is a necessary nontriviality condition.

\

\subsection{Existentially closed factors}\label{sec2:existentially}  The model-theoretic notion of an \emph{existentially closed (e.c.)} structure is the generalization of the notion of algebraically closed field to an arbitrary structure (see Subsection \ref{sec8:existentially} below for a precise definition).  In particular, it makes sense to study e.c. groups, e.c. graphs, and, yes, even e.c. tracial von Neumann algebras.  One can prove many general facts about the class of e.c.\ tracial von Neumann algebras, such as they must be II$_1$ factors with McDuff's property and with only approximate inner automorphisms.  There are a plethora of e.c.\ tracial von Neumann algebras; in particular, every tracial von Neumann algebra embeds in an e.c.\ one.  However, can one actually name a concrete e.c.\ tracial von Neumann algebra?  It turns out that a positive solution to CEP is equivalent to the statement that $\cal R$ is an e.c.\ tracial von Neumann algebra; a proof of this fact will be given in Subsection \ref{sec8:existentially} below.

\

\subsection{Noncommutative real algebraic geometry}\label{sec2:noncommutative}  A Positivstellenzats is a theorem that declares that certain elements that are positive in some way are so for some ``good reason.'' Perhaps the best-known such result is the positive solution to Hilbert's 17th Problem, due to Artin \cite{artin} (although this author is unabashedly fond of Abraham Robinson's model-theoretic solution \cite{robinson}):  if $f(X_1,\ldots,X_n)\in \bb R(X_1,\ldots,X_n)$ is a positive semidefinite rational function, that is, a rational function such that $f(x_1,\ldots,x_n)\geq 0$ for all $x_1,\ldots,x_n\in \bb R$, then $f$ is a sum of squares of rational functions, providing a ``good reason'' that $f$ is positive semidefinite.  

One can ponder noncommutative versions of Artin's theorem.  First, we set $\bb R\langle X_1,\ldots,X_n\rangle$ to be the set of polynomials in $n$ noncommuting variables.  Consider the ``positivity'' statement that $f(A_1,\ldots,A_n)\geq 0$ for all self-adjoint matrices $A_1,\ldots,A_n\in M_m(\bb R)$ of operator norm at most $1$, for all $m\in \bb N$.  Then a theorem of Helton and McCullough \cite{HM} tells us that there is a good reason for this kind of positivity, namely that, for all $\epsilon\in \bb R^{>0}$, $f+\epsilon$ belongs to the \emph{quadratic module} generated by $1-X_i^2$, $i=1,\ldots,n$.  Here, a quadratic module is a subset $M$ of $\bb R\langle X_1,\ldots,X_n\rangle$ containing $1$, closed under addition, and closed under the function $a\mapsto g^*ag$, where $a\in M$ and $g\in \bb R\langle X_1,\ldots,X_n\rangle$ (and where $g^*$ is the result of reversing the orders of the variables in each monomial of $g$).  Note indeed that all functions in the quadratic module generated by the $1-X_i^2$'s must be positive in the above sense and the Helton-McCullough result says that this is (approximately) the good reason that any such noncommutative polynomial might be positive.  

Now suppose instead that we assume that $f$ is merely ``trace positive,'' that is, $\tr(f(A_1,\ldots,A_n))\geq 0$ for all such $A_1,\ldots,A_n$ as in the previous paragraph.  Clearly the operators in the Helton and McCullough result are trace positive.  But now you can also add finite sums of commutators $[A,B]:=AB-BA$ since the trace of a commutator vanishes.  One can ask if this new class of noncommutative polynomials gives a necessary and sufficient condition to be trace positive, that is, if $f$ is tracially positive, must it be the case that, for every $\epsilon>0$, $f+\epsilon$ differs from an element of the quadratic module generated by the $1-X_i^2$'s by a sum of commutators?  It turns out that this tracial version of the Positivstellenzats from the previous paragraph is actually equivalent to the CEP, a result proven by Klep and Schweighofer in in \cite{KS}.  

\section{A crash course in operator algebras}\label{sec3}

In this long section, we explain all of the background material in operator algebras one needs to know to understand the statements of both the CEP and the QWEP problem as well as to understand how a positive solution to the former implies a positive solution to the latter.  Nearly everything discussed here can be found in Pisier's book \cite{pisier}.  Brown and Ozawa's book \cite{BO} is another nice reference.

\subsection{Introducing \cstar-algebras}\label{sec3:introducingC}

A \emph{$*$-algebra} is an algebra $\A$ over $\bb C$ satisfying, for all $x,y\in \A$ and $\lambda\in \bb C$:
\begin{itemize}
    \item $(x+y)^*=x^*+y^*$
    \item $(xy)^*=y^*x^*$
      \item $(x^*)^*=x$
    \item $(\lambda x)^*=\bar \lambda x^*$.
\end{itemize}
If $\A$ is actually a unital algebra over $\bb C$ with unit $1$ for which $1^*=1$, we say that $\A$ is a \emph{unital $*$-algebra}.  There are obvious notions of $*$-subalgebra of a $*$-algebra and unital $*$-subalgebra of a unital $*$-algebra.

A \emph{$*$-homomorphism} between $*$-algebras is an algebra homomorphism that also preserves the $*$-operation.  If $\phi:\cal A\to \cal B$ is a $*$-homomorphism between unital $*$-algebras, then we implicitly assume that $\phi$ maps the unit of $\A$ to the unit of $\cal B$.

In this paper, the most relevant (unital) $*$-algebras are $\cal B(\cal H)$ and its (unital) $*$-subalgebras.  Recall that for a Hilbert space $\cal H$, a linear operator $T:\cal H\to \cal H$ is \emph{bounded} if its \emph{operator norm} $\|T\|:=\sup\{\|T\xi\| \ : \ \|\xi\|\leq 1\}$ is finite.  $\cal B(\cal H)$ is a $*$-algebra with the algebra operations being addition, composition, and scalar multiplication and with the $*$-operation being given by the \emph{adjoint}, where, for $T\in \cal B(\cal H)$, we have that $T^*\in \cal B(\cal H)$ is the unique operator for which $\langle T\xi,\eta\rangle=\langle \xi,T^*\eta\rangle$ for all $\xi,\eta\in \cal H$.  (In connection with this formula, we follow the convention that inner products are linear in the first argument and conjugate-linear in the second argument; this is the opposite of the convention used in the physics literature.)  $\cal B(\cal H)$ is a unital $*$-algebra with identity operator $I_{\cal H}$ acting as the unit.

We now define the first kind of operator algebra, namely the class of \cstar-algebras.  For both classes of operator algebras, there are two approaches to their definition, namely the concrete and the abstract.  A \emph{concrete \cstar-algebra} is a $*$-subalgebra $\A$ of $\cal B(\cal H)$ that is closed in the operator norm topology.  If, moreover, $\A$ contains the identity $I_{\cal H}$, then we say that $\A$ is a \emph{unital concrete \cstar-algebra}.  

We now present the abstract approach to \cstar-algebras.  Suppose that $\A$ is a $*$-algebra.  A \emph{\cstar-norm} on $\A$ is a norm on $\A$ satisfying the following identities for all $x,y\in \A$:
\begin{itemize}
    \item $\|xy\|\leq \|x\|\|y\|$
    \item $\|x^*\|=\|x\|$
    \item $\|x^*x\|=\|x\|^2$.
\end{itemize}
The first two identities are the usual axioms for defining a \emph{normed} $*$-algebra; the last axiom, called the \emph{\cstar-identity}, is what makes a \cstar-norm a \cstar-norm.  An \emph{abstract \cstar-algebra} is a $*$-algebra equipped with a complete \cstar-norm.  If $\A$ is a $*$-algebra equipped with a \cstar-norm, then the $*$-algebra operations extend naturally to the completion of the $*$-algebra, which is then an abstract \cstar-algebra.  An \emph{abstract unital \cstar-algebra} is an abstract \cstar-algebra that is a unital *-algebra; in this case, we have $\|1\|=1$.

It is an important fact that a $*$-homomorphism between abstract \cstar-algebras is necessarily contractive; it is an isometric embedding if and only if it is injective.  In particular, given any $*$-algebra $\A$, there is at most one norm on $\A$ which makes $\A$ into a \cstar-algebra.

It is an easy exercise to see that the operator norm on $\cal B(\cal H)$ is a \cstar-norm, whence every concrete \cstar-algebra is an abstract \cstar-algebra.  On the other hand, the \emph{Gelfand-Naimark theorem} states that every abstract \cstar-algebra is isomorphic (as abstract \cstar-algebras) to a concrete \cstar-algebra.  This result can be reformulated in terms of representations of \cstar-algebras.  Given a \cstar-algebra $\A$, a \emph{representation} of $\A$ is a $*$-homomorphism $\pi:\A\to \cal B(\cal H)$ for some Hilbert space $\cal H$.  Usually a nondegeneracy condition is assumed on a representation, namely that $\{\pi(a)(\xi)\ : \ a\in \A, \xi\in \cal H\}$ is dense in $\cal H$; if $\A$ is unital, this is equivalent to assuming that $\pi(1)=I_{\cal H}$.  The representation $\pi$ is \emph{faithful} if it is moreover injective (equivalently isometric).  Thus, the Gelfand-Naimark theorem states that every abstract \cstar-algebra admits a faithful representation.

From here on out, we no longer make a distinction between concrete and abstract \cstar-algebras and simply take either perspective whenever it is convenient.

\textbf{Unless stated otherwise, in the rest of this paper, we restrict attention to unital \cstar-algebras; we might often repeat this convention for emphasis}.

A \cstar-algebra is \emph{commutative} (or \emph{abelian}) if its multiplication is commutative.  Given a compact Hausdorff space $X$, the set $C(X)$ of continuous, complex-valued functions is a unital commutative \cstar-algebra under the pointwise operations of addition, multiplication, and scalar multiplication, with the $*$-operation being given by $f^*:=\bar f$ (complex conjugate), and with norm given by $\|f\|:=\sup_{x\in X}|f(x)|$.  In fact, all unital commutative \cstar-algebras are of this form and there is a dual equivalence of categories (known as \emph{Gelfand duality}) between compact Hausdorff spaces with continuous maps and unital commutative \cstar-algebras with $*$-homomorphisms.  For this reason, \cstar-algebra theory is often dubbed ``noncommutative topology.''

There are special kinds of elements in \cstar-algebras that will be important throughout this paper.  If $\A$ is a \cstar-algebra, $x\in \A$ is called:
\begin{itemize}
    \item \emph{self-adjoint} if $x^*=x$
    \item \emph{positive} if $x=y^*y$ for some $y\in \A$
    \item a \emph{projection} if $x$ is self-adjoint and $x^2=x$
    \item \emph{unitary} if $x^*x=xx^*=1$.
\end{itemize}

In the case that $\A=\cal B(\cal H)$, the self-adjoint (resp. positive) elements are those with spectrum contained in the reals (resp. the positive reals) while the projections correspond to orthogonal projections onto closed subspaces.  In the case that $\A=C(X)$ for $X$ a compact space, the self-adjoint (resp. positive) elements correspond to the real-valued (resp. positive real-valued) functions while the projections correspond to those functions which take only the values $0$ or $1$.

\subsection{Let's be positive}\label{sec3:let's}

An important role in this story is played by maps between \cstar-algebras which are not necessarily $*$-homomorphisms but still preserve some remnants of the \cstar-algebra structure. It is hard to truly appreciate the importance of these maps without getting into the details of the results to follow, but we introduce the terminology in order to be able to follow the definitions and theorems.

First, we say that a linear map $\phi:\cal A\to \cal B$ between \cstar-algebras is \emph{positive} if it maps positive elements to positive elements.  Note that a $*$-homomorphism is positive:  $\phi(a^*a)=\phi(a)^*\phi(a)$.

For many purposes, a stronger version of positivity is needed.  First, for any \cstar-algebra $\A$ and any $n\geq 1$, we let $M_n(\A)$ denote the set of $n\times n$ matrices with entries from $\A$.  We can view $M_n(\A)$ as a $*$-subalgebra of $\cal B(\bigoplus_{i=1}^n \cal H)$ and it is readily verified that, under this identification, $M_n(\A)$ is closed in the operator norm topology, that is, $M_n(\A)$ is a \cstar-algebra once again.  Note also that a linear map $\phi:\A\to \cal B$ induces a linear map $\phi_n:M_n(\A)\to M_n(\cal B)$ given by $\phi_n(a_{ij})=(\phi(a_{ij}))$.  We say that $\phi$ is \emph{completely positive} if each $\phi_n$ is a positive map.  If, in addition, $\phi(1)=1$, we say that $\phi$ is \emph{unital, completely positive}, or \emph{ucp} for short.  A $*$-homomorphism $\phi:\A \to \cal B$ is easily seen to induce $*$-homomorphisms $\phi_n:M_n(\A)\to M_n(\cal B)$, whence $*$-homomorphisms are ucp.  It can be shown that if $\cal A$ is commutative, then any positive map $\phi:\cal A\to \cal B$ is automatically completely positive.

In a similar vein, one says that $\phi$ as above is \emph{completely bounded} (resp. \emph{completely contractive}) if each $\phi_n$ is bounded (resp. contractive).

A fundamental theorem of Stinespring says that ucp maps are not too far away from being $*$-homomorphisms.  More precisely, consider the following situation:  suppose that $\cal H$ and $\cal K$ are Hilbert spaces and $V:\cal H\to \cal K$ is an isometry, that is, a linear map for which $V^*V=I_{\cal H}$ (or, in other words, $\langle V\xi,V\xi\rangle=\langle \xi,\xi\rangle$ for all $\xi\in \cal H$).  Then for any representation $\pi:\A\to \cal B(\cal K)$ of $\A$, we have a map $\phi:\A \to \cal B(\cal H)$ given by $\phi(a)(\xi):=V^*(\pi(a)(V\xi))$ which is readily verified to be ucp.  The Stinespring dilation theorem says that all ucp maps $\phi:\A\to \cal B(\cal H)$ arise in this way.  A particular consequence of this theorem is that ucp maps are completely contractive.  The relevance of Stinespring's theorem to our story is that certain results that hold somewhat immediately for $*$-homomorphisms will also hold for ucp maps (see Subsection \ref{sec3:problem} below).

We mention in passing that cp maps play an important role in quantum information theory.  Indeed, one perspective on a quantum state (say on a finite-dimensional state space) is that of a positive matrix of trace $1$, corresponding to the density matrix of some mixed state.  A \emph{quantum channel} is a linear map that is to represent some allowable physical transformation on quantum states.  In particular, if it is to map density matrices to density matrices, then it should be trace-preserving and positive.  However, often one needs to add ``ancilla bits'' to a given state and then apply the correspondiing quantum channel.  The desire to have the resulting matrix be a density matrix again is equivalent to the requirement that the quantum channel be completely positive instead of merely positive.  The reader can find more details in, for example, Vern Paulsen's lecture notes \cite{paulsen}.

\subsection{Introducing von Neumann algebras}\label{sec3:introducingV}  We now turn our attention to the other kind of operator algebra, the von Neumann algebra.  Once again, we have a choice between a concrete definition and an abstract definition.  A \emph{concrete von Neumann algebra} is a unital $*$-subalgebra of $\cal B(\cal H)$ closed in the \emph{weak operator topology} (WOT), where the WOT on $\cal B(\cal H)$ has as a subbasic open neighborhood of $T\in \cal B(\cal H)$ those sets of the form $\{S\in \cal B(\cal H) \ : \ |\langle T\xi,\eta\rangle-\langle S\xi,\eta\rangle|<\epsilon\}$, where $\xi$ and $\eta$ range over $\cal H$ and $\epsilon$ ranges over positive real numbers.  It is easy to check that the WOT is a finer topology on $\cal B(\cal H)$ than the operator norm topology, whence every concrete von Neumann algebra is a (unital) concrete \cstar-algebra.  

A theorem of Sakai allows for an abstract reformulation:  a (unital) abstract \cstar-algebra $\cal M$ is isomorphic (as an abstract \cstar-algebra) to a concrete von Neumann algebra if and only if $\cal M$ is isometrically isomorphic to a dual Banach space, that is, if and only if there is a closed subspace $X\subseteq \cal M^{**}$ such that $\cal M=X^*$ isometrically.  In this case, $X$ is unique  and is called the \emph{predual} of $\cal M$, denoted $\cal M_*$.

Von Neumann's \emph{bicommutant theorem} allows for a purely algebraic reformulation of being a von Neumann algebra that is incredibly important to the theory.  First, given a subset $\cal S\subseteq \cal B(\cal H)$, set $\cal S':=\{T\in \cal B(\cal H) \ : \ TS=ST \text{ for all }S\in \cal S\}$, the so-called \emph{commutant of $\cal S$}.  Note that for any set $\cal S$, we have that $\cal S'$ is a von Neumann subalgebra of $\cal B(\cal H)$ and that $\cal S\subseteq \cal S'':=(\cal S')'$.  von Neumann's bicommutant theorem states that for any unital $*$-subalgebra $\cal A$ of $\cal B(H)$, $\cal A''$ coincides with the WOT-closure of $\cal A$ in $\cal B(\cal H)$; this common algebra is the von Neumann algebra generated by $\cal A$.  In fact, the bicommutant theorem shows that both of these coincide with the closure of $\cal A$ in the \emph{strong operator topology} (SOT), where the SOT is the topology on $\cal H$ where a subbasic open neighborhood of $T\in \cal B(\cal H)$ is the set $\{S\in \cal B(\cal H) \ : \ \|(T-S)\xi\|<\epsilon\}$, where $\xi$ ranges over $\cal H$ and $\epsilon$ ranges over the positive real numbers.  A consequence of the bicommutant theorem is that a unital $*$-subalgebra $\cal M$ of $\cal B(\cal H)$ is a von Neumann algebra if and only if $\cal M=\cal M''$.

A von Neumann algebra $\cal M$ is called \emph{separable} if its bidual is a separable Banach space.  This is equivalent to having a concrete representation of $\cal M$ on a separable Hilbert space, whence one sometimes calls such a von Neumann algebra \emph{separably acting}.

Just as in the case of \cstar-algebras, we can completely characterize the commutative von Neumann algebras.  Given a $\sigma$-finite measure space $(X,\mu)$, we can view $L^\infty(X,\mu)\subseteq \cal B(L^2(X,\mu))$ by identifying $f\in L^\infty(X,\mu)$ with $M_f\in \cal B(L^2(X,\mu))$ given by $M_f(g):=fg$.  It is an exercise to check that $L^\infty(X,\mu)=L^\infty(X,\mu)'$, whence $L^\infty(X,\mu)$ is a commutative von Neumann algebra.  Moreover, all commutative von Neumann algebras have this form, whence von Neumann algebra theory is often dubbed ``noncommutative measure theory.''

While $*$-homommorphisms between von Neumann algebras are automatically contractive with respect to the operator norm (as they are \cstar-algebras), since the relevant topology for defining von Neumann algebras is the WOT, the appropriate continuity condition relates to this latter topology.  More precisely, a positive linear map $\Phi:\cal M\to \cal N$ between von Neumann algebras is \emph{normal} if the restriction of $\Phi$ to the operator norm unit ball of $\cal M$ is continuous when $\cal M$ and $\cal N$ are equipped with their WOT topologies.  (The restriction to the operator norm unit ball may seem slightly unsightly; this is equivalent to saying that $\Phi$ is continuous when both $\cal M$ and $\cal N$ are equipped with their weak*-topologies when viewed as dual Banach spaces.)  One can reformulate normality in an intrinsic way that does not refer to the particular realization of $\cal M$ and $\cal N$:  $\Phi:\cal M\to \cal N$ is normal if and only if it is positive, linear, and $\Phi(\sup_{i\in I}x_i)=\sup_{i\in I}\Phi(x_i)$ for every bounded increasing net $(x_i)_{i\in I}$ of positive elements in $\cal M$.

\subsection{States and traces}\label{sec3:states}

Fix a compact Hausdorff space $X$.  Given a complex Borel measure $\mu$ on $X$, we can consider the associated integration functional $I_\mu\in C(X)^*$ given by $I_\mu(f):=\int_X fd\mu$ which satisfies $\|I_\mu\|=\|\mu\|$, where $\|\mu\|$ denotes the total variation norm of $\mu$.   Setting $M(X)$ to be the Banach space of complex Borel measures on $X$,  the Riesz representation theorem implies that this association yields an isomorphism $C(X)^*\cong M(X)$.  Moreover, the probability measures $\mu$ on $X$ correspond to those $I\in C(X)^*$ for which $I$ is a positive map satisfying $I(1)=1$.

More generally, given a unital \cstar-algebra $\A$, a \emph{state} on $\A$ is a positive linear functional $\phi$ on $\A$ with $\phi(1)=1$.  Thus, the states on a unital abelian \cstar-algebra $C(X)$ correspond to the integration functionals associated to probability measures on $X$ and one thinks of states on arbitrary \cstar-algebras as the abstract analog of such an integral.

The states on $\A$ form a convex, closed subset $\frak S(\A)$ of $\A^*$.  The extreme points of $\frak S(\A)$ are referred to as \emph{pure states}.  By the Krein-Milman theorem, finite convex combinations of pure states are dense in the space of all states.  The pure states on $\cal B(\cal H)$ are the \emph{vector states}, that is, the states of the form $T\mapsto \langle T \xi,\xi\rangle$ for some $\xi\in \cal H$.  

A consequence of the Hahn-Banach theorem is the fact that, for any \cstar-algebra $\A$ and any self-adjoint element $a\in \A$, we have $\|a\|=\sup_{\phi\in \frak S(\A)}|\phi(a)|$.  Another consequence of the Hahn-Banach theorem is that whenever $\cal A$ is a subalgebra of $\cal B$, any state on $\A$ can be extended to a state on $\cal B$.

When $\cal H$ is finite-dimensional, every state on $\cal B(\cal H)$ is of the form $T\mapsto \operatorname{Tr}(T\rho)$ for a unique positive operator $\rho$ of trace $1$, where $\operatorname{Tr}$ denotes the trace of an operator.  The operator $\rho$ is often called the \emph{density matrix} for the state.  When $\cal H$ is not necessarily finite-dimensional, the same result holds true for states on $\cal B(\cal H)$ that are continuous with respect to the weak*-topology on $\cal B(\cal H)$, except that $\rho$ is now stipulated to be a trace-class operator (see, for example, \cite[Theorem 19.9]{Hall}). 

The term state comes from quantum mechanics.  A first introduction to quantum mechanics will introduce a state of a physical system as simply a unit vector $\xi$ in the Hilbert space $\cal H$ associated to the physical system.  This usage of the word state corresponds to the vector states described above.  Later, one then encounters the notion of \emph{mixed state} (to accomodate for the fact that results of quantum measurements are merely probabilistic ``ensembles'' of pure states), which is often defined in terms of the density matrix as defined above.  The role of a state in quantum mechanics is simply to assign expected values of observables (see \cite[Section 19]{Hall}).    

An important construction associated to a state is the \emph{GNS construction}, which associates a representation of the \cstar-algebra to the state.  Suppose that $\phi$ is a state on $\A$.  We define a sesquilinear form $\langle \cdot,\cdot\rangle_\phi$ on $\A$ by $\langle x,y\rangle_\phi:=\phi(y^*x)$.  It is straightforward to check that this is a so-called pre-inner product on $\A$ in that it satisfies all of the properties of being an inner product except that $\langle x,x\rangle_\phi=0$ need not necessarily imply that $x=0$; when $\langle,\rangle_\phi$ is actually an inner product, we say that $\phi$ is \emph{faithful}.  Associated to $\langle \cdot,\cdot\rangle_\phi$ is the semi-norm $\|\cdot\|_\phi$ on $\A$ given by $\|x\|_\phi:=\sqrt{\langle x,x\rangle_\phi}$.  We obtain a Hilbert space $L^2(\A,\phi)$ by quotienting out by the closed subspace of vectors with $\|\cdot\|_\phi=0$ and then taking the completion.  Given $a\in \A$, we let $\hat a$ denote its equivalence class in $L^2(\A,\phi)$.  It follows that there is a representation $\pi_\phi:\A\to L^2(\A,\phi)$ uniquely determined by the condition $\pi_\phi(a)(\hat b):=\hat{ab}$ for all $a,b\in \A$.  The representation $\pi_\phi$ is \emph{cyclic}, meaning that there is a vector $\xi\in L^2(\A,\phi)$ for which $\overline{\{\pi_\phi(a)\xi \ : \ a\in \A\}}=L^2(\A,\phi)$, namely $\xi=\hat 1$.  Moreover, the vector state $\langle \cdot \hat 1,\hat 1\rangle_\phi$ on $L^2(\A,\phi)$ restricts to $\phi$ on the image of $\A$.  Note that $\pi_\phi$ is a faithful representation precisely when $\phi$ is faithful.

There is a converse to the above construction:  if $\pi:\A\to \cal B(\cal H)$ is a cyclic representation of $\A$ with cyclic vector $\xi\in \cal H$, then one obtains a state $\phi_\pi$ on $\A$ by $\phi_\pi(a):=\langle \pi(a)\xi,\xi\rangle$ and the GNS representation associated to $\phi_\pi$ is unitarily equivalent to $\pi$.

Define $\cal H_u:=\bigoplus_{\phi\in \frak S(\cal A)}L^2(\A,\phi)$ and set $\pi_u=\bigoplus_{\phi\in \frak S(\cal A)}\pi_\phi:\A\to \cal H_u$, which we call the \emph{universal representation} of $\A$.  Since $\|a\|=\sup_{\phi\in \frak S(\A)}|\phi(a)|$ for any self-adjoint $a\in \A$, it follows that $\pi_u$ is a faithful representation of $\A$.  Since any representation of $\A$ is a direct sum of cyclic representations and since every cyclic representation of $\A$ is, up to unitary equivalence, of the form $\pi_\phi$ for some $\phi\in \frak S(\A)$, we see that every representation of $\A$ is unitarily equivalent to a subrepresentation of $\pi_u$, whence the name!  

An important ingredient in this story is the von Neumann algebra $\pi_u(\A)''$ generated by the image of $\A$ in $\cal B(\cal H_u)$.  It can be shown that this von Neumann algebra is isometrically isomorphic to the Banach space $\A^{**}$, whence it is this notation that is usually used.  

A state $\phi$ on $\A$ is called a \emph{tracial state} if $\phi(ab)=\phi(ba)$ for all $a,b\in \A$.  For example, the normalized trace $\tr$ on $M_n(\bb C)$ given by $\tr(a):=\frac{1}{n}\operatorname{Tr}(a)$ is a tracial state on $M_n(\bb C)$.  

Since $\bb C$ is a von Neumann algebra, it makes sense to speak of \emph{normal} states on von Neumann algebras.  A faithful normal tracial state on a von Neumann algebra is simply referred to as a \emph{trace}.  A von Neumann algebra is called \emph{finite} if it admits a trace.  (This terminology makes much more sense if you introduce Murray-von Neumann equivalence of projections.)  A \emph{tracial von Neumann algebra} is a pair $(\cal M,\tau)$, where $\cal M$ is a von Neumann algebra and $\tau$ is a trace on $\cal M$.  An embedding of tracial von Neumann algebras is a normal, injective $*$-homomorphism that preserves the trace. The normalized trace $\tr$ on $M_n(\bb C)$ is a trace (in the von Neumann algebra sense) on $M_n(\bb C)$.  However, when $\cal H$ is infinite-dimensional, there is no trace on $\cal B(\cal H)$.  

Suppose that $\tau$ is a trace on $\cal M$.  The corresponding representation $\pi_\tau:\cal M\to \cal B(L^2(\cal M,\tau))$ is normal and the image of $\pi_\tau(M)$ is WOT-closed in $\cal B(L^2(\cal M,\tau))$.  Also, $\cal M$ is separable if and only if it is separable with respect to the metric stemming from the norm $\|\cdot\|_\tau$.  When restricted to the unit ball of $\cal M$, the topology induced by $\|\cdot\|_\tau$ coincides with the SOT on $\cal M$ it inherits from $L^2(\cal M,\tau)$.

We end this section by showing how traces can be used to define the hyperfinite II$_1$ factor, the star of this paper!

Given $n\geq 1$, there is a natural embedding of tracial von Neumann algebras $M_{2^n}(\bb C)\hookrightarrow M_{2^{n+1}}(\bb C)$ given by $A\mapsto \left(\begin{matrix} A & 0 \\ 0 & A\end{matrix}\right)$.  In this way, we obtain a directed system of tracial von Neumann algebras whose union is a $*$-algebra we denote by $\cal M:=\bigcup_{n}M_{2^n}(\bb C)$.   The fact that the embeddings preserve the normalized traces on the individual matrix algebras implies that $\cal M$ has a tracial state $\tau$ on it.  We apply the GNS construction to $\tau$ (which still works even though the original algebra is not necessarily complete) and take the von Neumann algebra generated by $\pi_\tau(M)$ inside of $\cal B(L^2(\cal M,\tau))$.  This von Neumann algebra is called the \emph{hyperfinite II$_1$ factor} $\cal R$.  We will see the reason for the ``II$_1$ factor'' in the name in the next section but the terminology \emph{hyperfinite} can be explained now.  A separable von Neumann algebra is called \emph{hyperfinite} if it contains an increasing union of finite-dimensional subalgebras whose union is WOT-dense.  Murray and von Neumann showed that there is a unique separable hyperfinite II$_1$ factor.  Consequently, if we started the above construction with any $M_n(\bb C)$ instead of $M_2(\bb C)$, we would have arrived at the same II$_1$ factor, namely $\R$.

\subsection{More on von Neumann algebras}\label{sec3:more}

The \emph{center of a von Neumann algebra} $\cal M$ is $Z(\cal M):=\cal M\cap \cal M'=\{x\in \cal M \ : \ xy=yx \text{ for all }y\in \cal M\}$.  A von Neumann algebra $M$ is called a \emph{factor} when its center is trivial, that is when $Z(\cal M)=\bb C\cdot 1$.  It is quite easy to see that $\cal B(\cal H)$ is a factor; in particular, each $M_n(\bb C)$ is a factor.  This makes it plausible that $\R$ is also a factor, given that it is the completion of an increasing sequence of factors; one just needs to check that no elements snuck into the center at the completion stage.

The interest in factors comes from the fact that they are the ``building blocks'' of all von Neumann algebras in the sense that every von Neumann algebra can be written as a direct integral (a generalization of direct sum) of factors and thus the study of arbitrary von Neumann algebras can usually be reduced to studying factors.

Murry and von Neumann divided the collection of factors into three types, (creatively) called types I, II, and III.  They further split the first two types into subtypes as follows.  First, for each $n\in \bb N$, there is a unique factor of type I$_n$, namely $M_n(\bb C)$.  The unique factor of type I$_\infty$ is $\cal B(\cal H)$ for $\cal H$ infinite-dimensional.  Next, a II$_1$ factor is an infinite-dimensional finite factor, that is, an infinite-dimensional factor that admits a trace.  Thus, the hyperfinite II$_1$ factor is indeed a II$_1$ factor.  A II$_\infty$ factor is one that can be written as a proper increasing union of type II$_1$ factors.  Equivalently, a II$_\infty$ factor can be written in the form $\cal M\bar \otimes \cal B(\cal H)$ for some II$_1$ factor $\cal M$ (see Subsection \ref{sec3:problem} below for the definition of tensor products of von Neumann algebras).  There is also division of type III factors into subtypes III$_\lambda$ for $\lambda\in [0,1]$, but we will not need to get into that here.

It is important to note that, while an arbitrary finite von Neumann algebra may have many traces, the trace on a finite factor is unique.  We also note the crucial fact that $\R$ embeds into any II$_1$ factor.

As alluded to above when defining finite von Neumann algebras, the above type classification makes more sense in the context of Murray-von Neumann equivalence.  However, we can still see understand this division using traces.  Given a von Neumann algebra $\cal M$, let $\cal P(\cal M)$ denote the set of projections in $\cal M$.  If $\tau$ is a trace on $\cal M$, set $\tau(\cal P(\cal M)):=\{\tau(p) \ : \ p\in \cal P(\cal M)\}$.  Note that, for the unique type I$_n$ factor $M_n(\bb C)$, we have $\tau(\cal P(M_n(\bb C)))=\{0,\frac{1}{n},\ldots,\frac{n-1}{n},1\}$, one value for every dimension being projected onto.  On the other hand, for a II$_1$ factor $\cal M$, one can show that $\tau(\cal P(\cal M))=[0,1]$.  Thus, we think of II$_1$ factors being like matrix factors in that they admit traces, but now we have a ``continuous'' dimension for projections.

One can explain the cases I$_\infty$ and II$_\infty$ using traces if one considers the unnormalized trace $\operatorname{Tr}$ on $\cal B(\cal H)$ given by $\operatorname{Tr}(T):=\sum_{i\in I}\langle T\xi_i,\xi_i\rangle$, where $(\xi_i)_{i\in I} $ is any orthonormal basis for $\cal H$.  Note that $\operatorname{Tr}$ can take the value $\infty$.  We then have that the possible traces of projections for the I$_\infty$ factor are $\{0,1,2,\ldots,\}\cup\{\infty\}$ while in the type II$_\infty$ case they are $[0,\infty]$.

\subsection{The tracial ultrapower construction and the official statement of the CEP}\label{sec3:tracial}

In general, an ultraproduct of a collection of ``similar'' structures is a structure of the same type that represents some sort of ``limit'' of these structures.  There are ways of making this precise using model theory and we will discuss this later in Subsection \ref{sec7:model}.  In this section, we show how to carry this construction out in the case of tracial von Neumann algebras and see how this allows us to precisely state the CEP.

First, one needs to introduce the notion of an \emph{ultrafilter}.  Given an index set $I$, an ultrafilter $\u$ on $I$ is simply a $\{0,1\}$-valued finitely additive probability measure on $I$.  One often identifies $\u$ with its set of measure $1$ sets and writes $X\in \u$ instead of $\u(X)=1$.  Following typical measure-theoretic terminology, given a property $P$ that may or may not hold of elements of $I$, we may write ``for $\u$-almost all $i\in I$, $P(i)$ holds'' when $\{i\in I \ : \ P(i) \text{ holds}\}$ belongs to $\u$.

Given a bounded sequence $(z_i)_{i\in I}$ of complex numbers, it is straightforward to show that there is a unique complex number $z$ such that, for every $\epsilon>0$, we have $|z-z_i|<\epsilon$ for $\u$-almost all $i\in I$.  This unique complex number $z$ is called the \emph{$\u$-ultralimit} of the sequence $(z_i)_{i\in I}$, denoted $\lim_{i,\u}z_i$ or simply $\lim_{\u} z_i$.

Given $j\in I$, the unique ultrafilter $\u$ on $I$ for which $\u(\{j\})=1$ is called the \emph{principal ultrafilter generated by $j$}, denoted $\u_j$.  An ultrafilter $\u$ on $I$ is called \emph{nonprincipal} if it is not principal.  Equivalently, $\u$ is nonprincipal if $\u(X)=0$ for all finite $X\subseteq I$.  It is easy to check that $\lim_{\u_j}z_i=z_j$, whence ultralimits along principal ultrafilters do not really capture a genuine notion of limit.  It is a basic fact that, for any infinite set $I$, there is a nonprincipal ultrafilter $\u$ on $I$.

We now come to the tracial ultraproduct construction.  Fix a family $(\cal M_i,\tau_i)_{i\in I}$ of tracial von Neumann algebras and an ultrafilter $\u$ on $I$.  We first set $$\ell^\infty(\cal M_i):=\left\{x\in \prod_{i\in I}\cal M_i \ : \ \sup_{i\in I}\|x(i)\|<\infty\right\},$$ that is, $\ell^\infty(\cal M_i)$ collects all those sequences from the Cartesian product $\prod_{i\in I} \cal M_i$ for which the operator norms of the coordinates are uniformly bounded.  It is readily verified that $\ell^\infty(\cal M_i)$ is a \cstar-algebra under the supremum norm.  It is tempting to try to define a tracial state $\tau$ on $\ell^\infty(\cal M_i)$ by declaring $\tau(x):=\lim_{\u}\tau_i(x(i))$, which makes sense given that the sequence $(\tau_i(x(i)))_{i\in I}$ is a uniformly bounded sequence of complex numbers (a consequence of the uniform bound on the operator norms of the coordinates of $x$).  Unfortunately, if each $x(i)$ is a positive element of $\cal M_i$, whence $x$ is a positive element of $\ell^\infty(\cal M_i)$, with the property that $\lim_\u \tau_i(x(i))=0$, then we have that $\tau(x)=0$ even though $x$ may not be zero.  In other words, this definition leads to a tracial state on $\ell^\infty(\cal M_i)$ that is not faithful.

We fix the above problem by defining $c_\u:=\{x\in \ell^\infty(\cal M_i) \ : \ \lim_\u \tau_i(x(i))=0\}$.  While there are a lot of things to check, we have that:
\begin{itemize}
    \item $c_\u$ is a two-sided ideal in $\ell^\infty(\cal M_i)$,
    \item $\ell^\infty(\cal M_i)/c_\u$ is a von Neumann algebra, and
    \item the induced tracial state $\tau$ on $\ell^\infty(\cal M_i)/c_\u$ given by $\tau([x]_\u):=\lim_\u \tau_i(x(i))$ is a trace (that is, a faithful, normal tracial state) on $\ell^\infty(\cal M_i)/c_\u$, where $[x]_{\u}$ denotes the coset of $x$ modulo $c_\u$.
\end{itemize}

The resulting tracial von Neumann algebra is denoted $(\prod_{\u}\cal M_i,\lim_\u \tau_i)$ and is called the \emph{tracial ultraproduct} of the family $(\cal M_i,\tau_i)$ with respect to $\u$.  When each $\cal M_i$ is a finite factor, then the trace on each factor is unique and we simplify the notation to $\prod_\u \cal M_i$.  When each $(\cal M_i,\tau_i)=(\cal M,\tau)$ equals a common tracial von Neumann algebra, we simply write $(\cal M,\tau)^\u$ and speak of the \emph{ultrapower} of $(\cal M,\tau)$ with respect to $\u$.  Similarly, the ultrapower of a finite factor is denoted $\cal M^\u$.  We view any tracial von Neumann algebra $(\cal M,\tau)$ as a subalgebra of $(\cal M,\tau)^\u$ via the \emph{diagonal embedding} which maps an element $a\in \cal M$ to the coset of the diagonal sequence $(i\mapsto a)$ modulo $c_\u$.

When $\u=\u_j$ is principal, one can verify that $\prod_\u (\cal M_i,\tau_i)\cong (\cal M_j,\tau_j)$, whence this is not a terribly interesting procedure.  The true power of the ultraproduct construction comes when one uses a nonprincipal ultrafilter, for then the ultraproduct is sort of an ``average'' or ``limit'' of the constitutent tracial von Neumann algebras.

If $\lim_\u \dim(\cal M_j)<\infty$, then $\prod_\u (\cal M_i,\tau_i)$ is also finite-dimensional.  Otherwise, $\prod_\u (\cal M_i,\tau_i)$ is quite large, in fact, non-separable, even if each $\cal M_i$ is separable.  It is quite common to hear expressions such as ``every II$_1$ factor in this paper (or talk) is separable unless it isn't.''  This tautology refers to the fact that often researchers are only interested in separable tracial von Neumann algebras and the only nonseparable II$_1$ factors that one might encounter are those obtained from a nonprincipal ultraproduct of a family of separable II$_1$ factors.  (We are being a bit sloppy:  nonprincipality only guarantees non-separability when the index set is countable; otherwise, one needs the mild assumption of \emph{countable incompleteness}.)

We can now officially state the:

\noindent \textbf{Connes Embedding Problem}: Given any nonprincipal ultrafilter $\u$ on $\bb N$, every separable tracial von Neumann algebra embeds into $\cal R^\u$, that is, admits a trace-preserving injective $*$-homomorphism into $\R^\u$.

Let us make several remarks on variations of the statement of the CEP:
\begin{itemize}
    \item It can be shown using some basic model theory that the CEP is equivalent to the statement that every separable tracial von Neumann algebra embeds into $\cal R^{\cal U}$ for \emph{some} nonprincipal ultrafilter $\cal U$ on $\bb N$.  This fact can also be witnessed by using a simple ultrafilter-free equivalent reformulation of the CEP known as the \emph{Microstate Conjecture}, discussed below. 
    \item The restriction to separable tracial von Neumann algebras is not necessary.  It can be shown that ultrapowers of $\cal R$ with respect to certain kinds of ultrafilters on larger index sets known as \emph{good ultrafilters} lead to larger ultrapowers that can embed tracial von Neumann algebras of larger density character.  In other words, we can reformulate CEP by saying every tracial von Neumann algebra embeds into some ultrapower of $\R$.  
    \item The validity of the CEP does not change if we restrict to embedding II$_1$ factors into an ultrapower of $\R$.  The reason for this is due to the fact that every tracial von Neumann algebra $(\cal M,\tau)$ embeds into a II$_1$ factor, say $(\cal M*L(\bb Z),\tau*\tau_{L(\bb Z)})$; here $L(\bb Z)$ is the group von Neumann algebra of the group of integers (see Subsection \ref{sec3:operator} below) and $*$ denotes the \emph{free product} of tracial von Neumann algebras.
    \item One can replace $\cal R^\u$ with a nonprincipal ultraproduct of matrix algebras without changing the validity of the CEP.  More precisely, CEP is equivalent to the statement that, for any nonprincipal ultrafilter $\u$ on $\bb N$, every separable tracial von Neumann algebra embeds into $\prod_\u M_n(\bb C)$.  This follows from the fact that each $M_n(\bb C)$ embeds in $\cal R$, whence $\prod_\u M_n(\bb C)$ embeds in $\cal R^\u$, while there are conditional expectations $\Phi_n:\cal R\to M_n(\bb C)$ and the ultralimit of these expectations yields an embedding $\lim_\u \Phi_n:\cal R^\u\hookrightarrow \prod_\u M_n(\bb C)$.  (See Subsection \ref{sec3:kirchberg's} below for the definition of conditional expectation.)
\end{itemize} 

  The last alternate reformulation makes the equivalence with the so-called \emph{Microstates Conjecture} more apparent.  The Microstate Conjecture states that:  for any tracial von Neumann algebra $(\cal M,\tau)$, any finite collection $p_1(x),\ldots,p_m(x)$ of $*$-polynomials in the noncommuting variables $x=(x_1,\ldots,x_n)$, any $a_1,\ldots,a_n\in \cal M$ in the operator norm unit ball of $\cal M$, and any $\epsilon>0$, there is $k\in \bb N$ and $b_1,\ldots,b_n\in M_k(\bb C)$ in the operator norm unit ball such that
$$\max_{1\leq i\leq n}|\tau(p_i(a))-\tr(p_i(b))|<\epsilon.$$  In other words, any ``finite configuration'' that can be obtained in some tracial von Neumann algebra can be approximately obtained in some matrix algebra.  It is this formulation of CEP that appeared in connection with free entropy as discussed in Subsection \ref{sec2:free} above.

\subsection{Operator algebras coming from groups}\label{sec3:operator}

A large source of operator algebras arise from groups and these algebras play an important role in our story.

First, a \emph{unitary representation} of a (discrete) group $G$ is a group homomorphism $\pi:G\to \cal U(\A)$, where $\cal A$ is a \cstar-algebra and $\cal U(\cal A)$ denotes the group of unitary elements of $\A$.

Suppose that $G$ is a group.  Let $\ell^2(G)$ be the Hilbert space formally generated by an orthonormal basis $\zeta_h$ for all $h \in G$.  For any $g \in G$, define $u_g$ to be the linear operator on $\ell^2(G)$ determined by $u_g(\zeta_h) = \zeta_{gh}$ for all $h \in G$.  Notice that $u_g$ is unitary for all $g \in G$ (since $u_g^* = u_g^{-1} = u_{g^{-1}})$ and so $\lambda:G\to \cal U(\ell^2(G))$ given by $\lambda(g):=u_g$ is a unitary representation of $G$, called the \emph{left regular representation} of $G$.

Recall that the group algebra $\bb C[G]$ consists of formal linear combinations $\sum_{g\in G}c_gg$ with only finitely many nonzero coefficients.  There is a natural $*$-algebra structure on $\bb C[G]$, the addition and multiplication being the obvious ones and the $*$-operation being given by $(\sum_{g\in G}c_gg)^*=\sum_{g\in G}\overline{c_g}g^{-1}$.  $\bb C[G]$ is in fact a unital $*$-algebra with unit $e$, where $e$ denotes the identity of the group.

The left regular representation $\lambda$ of $G$ extends by linearity to a unital $*$-algebra homomorphism $\pi:\bb C[G]\to \cal B(\ell^2(G))$.

The \emph{reduced group \cstar-algebra of $G$}, denoted $C_r^*(G)$, is the closure of $\pi(\bb C[G])$ in the operator norm topology on $\cal B(\ell^2(G))$.  The \emph{group von Neumann algebra of $G$}, denoted $L(G)$, is the closure of $\pi(\bb C[G])$ in the WOT on $\cal B(\ell^2(G))$.  Moreover, the vector state on $\cal B(\ell^2(G))$ corresponding to $\xi_e$ yields a tracial state on $C^*_r(G)$ and a trace on $L(G)$.

When $G$ is finite, $C^*_r(G)=L(G)=\bb C[G]$ and is generally considered uninteresting (to operator algebraists).  When $G$ is infinite, $L(G)$ is a II$_1$ factor precisely when $G$ is an \emph{ICC group}, that is, when all nontrivial conjugacy classes of $G$ are infinite.

The procedure of taking the group von Neumann algebra of a group can ``forget'' a lot of the algebraic structure of the group.  For example, it follows from Connes' fundamental work \cite{Connes} that all ICC amenable groups have group von Neumann algebra isomorphic to $\R$.

In the sequel, the reduced group \cstar-algebra of a group is not quite as important as a second \cstar-algebra associated to a group, the so-called \emph{universal (or maximal) group \cstar-algebra}.  To define this, we define a norm on $\bb C[G]$ by defining
$$\left\|\sum_{g\in G}c_gu_g\right\|=\sup\left\{\left\|\sum_{g\in G}c_g\pi(g)\right\| \ : \ \pi:G\to U(\A) \text{ a unitary representation of }G\right\}.$$  It is readily verified that this is a well-defined (that is, finite) \cstar-norm on $\bb C[G]$.  The completion of $\bb C[G]$ with respect to this norm is thus a \cstar-algebra, called the universal \cstar-algebra associated to $G$, denoted $C^*(G)$.  Since the above norm is easily seen to be the maximal \cstar-norm on $\bb C[G]$, it is sometimes called the maximal norm on $\bb C[G]$ and the completion the maximal group \cstar-algebra of $G$.  It follows immediately from the definition that any unitary representation $\pi:G\to \cal U(\A)$ of $G$ extends uniquely to a $*$-homomorphism $\pi:C^*(G)\to \A$.  

A particular corollary of this universal property of the universal group \cstar-algebra is that $C^*(\bb F_\infty)$ is \emph{surjectively universal}, where $\bb F_\infty$ is the free group on a countably infinite set of generators.  More precisely, given any separable \cstar-algebra $\A$, there is a surjective $*$-homomorphism $C^*(\bb F_\infty)\to \A$.  To see that this is the case, just note that there is a countable set $\{u_i \ : \ i\in \bb N\}$ of unitaries that generates $\A$ (as a \cstar-algebra); now apply the universal property to the surjective unitary representation $\bb F_\infty \to \cal U(\A)$ obtained by mapping the $i^{\text{th}}$ basis element of $\bb F_\infty$ onto $u_i$.

Another consequence of the universal property is that if $f:G\to H$ is a group homomorphism, then we get an induced $*$-algebra homomorphism $C^*(f):C^*(G)\to C^*(H)$.  A less obvious fact is that if $H$ is a subgroup of $G$, then $C^*(H)$ is naturally a \cstar-subalgebra of $C^*(G)$.  This follows immediately from the definitions once one knows that any unitary representation of $H$ can be extended to a unitary representation of $G$ (via a technique known as \emph{induction}; see \cite[Chapter 6]{folland}).

By considering the left-regular representation of $G$, we immediately see that there is a canonical surjective $*$-homomorphism $C^*(G)\to C^*_r(G)$.  In general, this map is not an isomorphism, that is, it often has nontrivial kernel.  In fact, the canonical map $C^*(G)\to C^*_r(G)$ is an isomorphism precisely when $G$ is amenable.

One final fact will prove useful later:  for any two groups $G$ and $H$, we have $C^*(G*H)\cong C^*(G)*C^*(H)$, where $G*H$ denotes the free product of groups and $C^*(G)*C^*(H)$ denotes the \emph{unital free product of \cstar-algebras}, which is slightly annoying to define but whose properties can be guessed from the terminology.

\subsection{The problem with \cstar-algebra tensor products}\label{sec3:problem}

Before discussing the issues associated with defining tensor products of \cstar-algebras, we first recall the tensor product construction for vector spaces.  Let $V$ and $W$ be vector spaces over the same field $\bb K$.  We let $\bb F(V\times W)$ be the free $\bb K$-vector space on the set $V\times W$, that is, all formal linear combinations $\sum_{(v,w)\in V\times W}c_{(v,w)}(v,w)$ with only finitely many nonzero coefficients.  $\bb F(V\times W)$ carries an obvious $\bb K$-vector space structure.  The tensor product of $V$ and $W$, denoted $V\odot W$, is the quotient of $\bb F(V\times W)$ by the subspace generated by elements of the following form, for $v,v'\in V$, $w,w'\in W$, and $\alpha\in \bb K$:
\begin{itemize}
    \item $(v+v',w)-(v,w)-(v',w)$
    \item $(v,w+w')-(v,w)-(v,w')$
    \item $(\alpha v,w)-\alpha(v,w)$
    \item $(v,\alpha w)-\alpha(v,w)$.
\end{itemize}

While it is more common to write $V\otimes W$ instead of $V\odot W$, we will reserve $\otimes$ for ``analytic'' tensor products (to be defined shortly) and will use $\odot$ for the above ``algebraic'' tensor product.

The equivalence class of $(v,w)$ in $V\odot W$ is denoted $v\otimes w$.  Thus, an arbitrary element of $V\odot W$ may be written as a formal linear combination $\sum_{i=1}^n \alpha_i v_i\otimes w_i$, but not necessarily uniquely.

If $V$ and $W$ are both finite-dimensional, then so is $V\odot W$ with $\dim(V\odot W)=\dim(V)\cdot \dim(W)$; if $\{v_1,\ldots,v_m\}$ is a basis for $V$ and $\{w_1,\ldots,w_n\}$ is a basis for $W$, then $\{v_i\otimes w_j \ : \ 1\leq i\leq m, 1\leq j\leq n\}$ is a basis for $V\odot W$.

It is clear from the construction that if $S:V_1\to V_2$ and $T:W_1\to W_2$ are $\bb K$-linear maps, then there is a $\bb K$-linear map $S\odot T:V_1\odot W_1\to V_2\odot W_2$ uniquely determined by $(S\odot T)(v\otimes w)=S(v)\otimes T(w)$.

If $\cal H$ and $\cal K$ are Hilbert spaces, then the algebraic tensor product $\cal H\odot \cal K$ comes naturally equipped with an inner product uniquely determined by $$\langle \xi_1\otimes \eta_1,\xi_2\otimes \eta_2\rangle=\langle \xi_1,\xi_2\rangle \cdot \langle \eta_1,\eta_2\rangle.$$

The completion of $\cal H\odot \cal K$ with respect to this inner product is then a Hilbert space, denoted $\cal H\otimes \cal K$ and called the \emph{Hilbert space tensor product} of $\cal H$ and $\cal K$.  If $\{e_i\ : \ i\in I\}$ and $\{f_j\ : \ j\in J\}$ are orthonormal bases for $\cal H$ and $\cal K$ respectively, then $\{e_i\otimes f_j \ : \ i\in I, \ j\in J\}$ is an orthonormal basis for $\cal H\otimes \cal K$.  Moreover, if $S:\cal H_1\to \cal H_2$ and $T:\cal K_1\to \cal K_2$ are bounded linear maps, then the algebraic tensor product map $S\odot T$ extends uniquely to a bounded linear map $S\otimes T:\cal H_1\otimes \cal K_1\to \cal H_2\otimes \cal K_2$.

We now come to the task of defining tensor products of operator algebras.  We first note that if $\A$ and $\cal B$ are two $*$-algebras, there is a natural $*$-algebra operation on their algebraic tensor product $\A\odot \cal B$ determined by 
\begin{itemize}
    \item $(x_1\otimes y_1)\cdot (x_2\otimes y_2)=(x_1x_2)\otimes (y_1y_2)$
    \item $(x\otimes y)^*=x^*\otimes y^*$.
\end{itemize}

If $\cal A$ and $\cal B$ are both unital, then so is $\cal A\odot \cal B$ with unit $1\otimes 1$.

The tensor product of von Neumann algebras is fairly uncontroversial.  Consider concretely represented von Neumann algebras $\cal M\subseteq \cal B(\cal H)$ and $\cal N\subseteq \cal B(\cal K)$.  It is straightforward to check that the algebraic tensor product $\cal M\odot \cal N$ is naturally a subset of $\cal B(\cal H\otimes \cal K)$ (using the tensor product of linear transformation construction above) and that the $*$-algebra structure induced by this identification agrees with the one placed on it in the previous paragraph.  The \emph{von Neumann algebra tensor product} $\cal M\bar\otimes \cal N$ of $\cal M$ and $\cal N$ is then the WOT closure of $\cal M\odot \cal N$ in $ \cal B(\cal H\otimes \cal K)$.  One can verify that this construction is indeed independent of the choice of representations of $\cal M$ and $\cal N$.

The story for \cstar-algebras, on the other hand, is far more complicated in general.  Fix \cstar-algebras $\A$ and $\cal B$.  We seek \cstar-norms on $\A\odot \cal B$, for then the completion of $\A\odot \cal B$ with respect to such a \cstar-norm will be \emph{a} \cstar-algebra tensor product of $\A$ and $\cal B$.  

One natural choice is to proceed as in the case of von Neumann algebras, that is, fix concrete representations $\A\subseteq \cal B(\cal H)$ and $\cal B\subseteq \cal B(\cal K)$ and to consider the operator norm on $\A\odot \cal B\subseteq \cal B(\cal H\otimes \cal K)$.  One can verify that this norm on $\A\odot \cal B$ is a \cstar-norm and is independent of the choice of representations.  This norm is called the \emph{minimal tensor norm}, denoted $\|\cdot\|_{\min}$.  The justification for the name comes from a theorem of Takesaki showing that $\|\cdot\|_{\min}$ is indeed the minimal \cstar-norm on $\A\odot \cal B$.  The completion of $\A\odot \cal B$ with respect to $\|\cdot\|_{\min}$ is denoted $\A\otimes_{\min}\cal B$ and is called the \emph{minimal tensor product} of $\A$ and $\cal B$.  One should be aware of the fact that some authors simply write $\A\otimes \cal B$ instead of $\A\otimes_{\min}\cal B$.  

A useful property of the minimal tensor product is the following result, which follows from the independence of the choice of representation:
If $\pi_\A:\A_1\to \A_2$ and $\pi_{\cal B}:\cal B_1\to \cal B_2$ are $*$-homomorphisms, then the linear map $\pi_\A\odot \pi_{\cal B}:\A_1\odot \cal B_1\to \A_2\odot \cal B_2$ extends uniquely to a $*$-homomorphism $\pi_\A\otimes \pi_{\cal B}:\A_1\otimes_{\min}\cal B_1\to \A_2\otimes_{\min}\cal B_2$.
Using the Stinespring Dilation theorem, one can generalize the conclusion of the previous sentence to ucp maps as follows: if $\Phi_\A:\A\to \cal B(\cal H_{\cal A})$ and $\Phi_{\cal B}:\cal B\to \cal B(\cal H_{\cal B})$ are ucp maps, then there is a unique ucp map $\Phi_{\cal A}\otimes \Phi_{\cal B}:\cal A\otimes_{\min} \cal B\to \cal B(\cal H_{\cal A}\otimes \cal H_{\cal B})$ determined by $(\Phi_{\cal A}\otimes \Phi_{\cal B})(a\otimes b)=\Phi_{\cal A}(a)\otimes \Phi_{\cal B}(b)$.

Another natural \cstar-norm to consider on $\A\odot \cal B$ is the so-called \emph{maximal norm} defined by 
$$\|x\|_{\max}:=\sup\{\|\pi(x)\| \ : \ \pi:\A\odot \cal B\to \cal B(\cal H) \text{ a }*\text{-homomorphism}\}.$$  In connection with this formula, it is useful to observe that a $*$-homomorphism $\pi:\A\odot \cal B\to \cal B(\cal H)$ restricts to $*$-homomorphisms $\pi_\A:\A\to \cal B(\cal H)$ and $\pi_{\cal B}:\cal B\to \cal B(\cal H)$ with commuting ranges and, conversely, any two $*$-homomorphisms $\pi_\A:\A\to \cal B(\cal H)$ and $\pi_{\cal B}:\cal B\to \cal B(\cal H)$ with commuting ranges yield a $*$-homomorphism $\pi_{\A}\odot \pi_{\cal B}:\A\odot \cal B\to \cal B(\cal H)$ uniquely determined by $(\pi_{\cal A}\odot \pi_{\cal B})(x\otimes y):=\pi_{\cal A}(x)\pi_{\cal B}(y)$.  (The commutativity of the ranges of $\pi_{\cal A}$ and $\pi_{\cal B}$ ensure that this map is in fact a $*$-homomorphism.)  It is clear that $\|\cdot\|_{\max}$ is a \cstar-norm on $\A\odot \cal B$; the completion of $\cal A\odot \cal B$ with respect to $\|\cdot\|_{\max}$ is called the \emph{maximal tensor product} of $\A$ and $\cal B$, denoted $\A\otimes_{\max}\cal B$.  Any pair of $*$-homomorphisms $\pi_\A:\A\to \cal B(\cal H)$ and $\pi_{\cal B}:\cal B\to \cal B(\cal H)$ with commuting ranges yields a $*$-homomorphism $\pi_\A\otimes \pi_{\cal B}:\A\otimes_{\max}\cal B\to \cal B(\cal H)$ that extends $\pi_{\cal A}\odot \pi_{\cal B}$.    Consequently, $\|\cdot\|_{\max}$ really is the largest \cstar-norm on $\cal A\odot \cal B$.  Using a more complicated Stinespring argument than the one mentioned above, one can show that any pair of ucp maps $\Phi_\A:\A\to \cal B(\cal H)$ and $\Phi_{\cal B}:\cal B\to \cal B(\cal H)$ with commuting ranges yields a ucp map $\Phi_\A\otimes \Phi_{\cal B}:\A\otimes_{\max}\cal B\to \cal B(\cal H)$ uniquely determined by $(\Phi_\A\otimes \Phi_{\cal B})(a\otimes b)=\Phi(a)\Phi(b)$. 

Before moving forward, we notice the following two facts, which are readily verified from the definitions:  for any pair of groups $G$ and $H$, we have:
\begin{itemize}
    \item $C^*_r(G\times H)\cong C^*_r(G)\otimes_{\min}C^*_r(H)$
    \item $C^*(G\times H)\cong C^*(G)\otimes_{\max}C^*(H)$.
\end{itemize}

Returning to the general discussion, we have defined two ``extreme'' \cstar-norms on $\A\odot \cal B$.  In general, they can be different.  For example, it can be shown that the maximal and minimal norms on $C^*_r(\bb F_2)\odot C^*_r(\bb F_2)$ are distinct.  The corresponding question for $C^*(\bb F_2)\odot C^*(\bb F_2)$ turns out to be equivalent to CEP, as we will soon see!  Another somewhat surprising result is that the maximal and minimal norms on $\cal B(\cal H)\odot \cal B(\cal H)$ (for $\cal H$ infinite-dimensional) are also distinct, a result due to Junge and Pisier \cite{JP}.  In fact, Ozawa and Pisier \cite{OP} showed that there exist at least continuum many different \cstar-norms on $\cal B(\cal H)\odot \cal B(\cal H)$ when $\cal H$ is infinite-dimensional.

We say that $(\A,\cal B)$ form a \emph{nuclear pair} if there is a unique \cstar-norm on $\A\odot \cal B$, that is, if the minimal and maximal norms on $\A\odot \cal B$ coincide.  We also say that $\A$ is \emph{nuclear} if $(\A,\cal B)$ is a nuclear pair for every \cstar-algebra $\cal B$.  There are many interesting examples of nuclear \cstar-algebras.  For example, $M_n(\bb C)$ is nuclear for all $n$, the reason being that $M_n\odot \cal B\cong M_n(B)$, which is already a \cstar-algebra with a unique \cstar-norm.  A more interesting example coming from groups is that $C^*(G)$ is nuclear if and only if $C^*_r(G)$ is nuclear if and only if $G$ is amenable (in which case $C^*(G)=C^*_r(G)$).

The following theorem of Kirchberg \cite{K} will be central moving forward:

\begin{thm}
$(C^*(\bb F_\infty),\cal B(\cal H))$ is a nuclear pair.
\end{thm}

Note that neither of these algebras are nuclear.  The importance of Kirchberg's theorem stems from the fact that $C^*(\bb F_\infty)$ is surjectively universal while $\cal B(\cal H)$ is injectively universal.

We also note that if $H$ is a subgroup of $G$ and $\A$ is any \cstar-algebra for which $(C^*(G),\A)$ is a nuclear pair, then so is $(C^*(H),\A)$.  In particular, whether or not $(C^*(\bb F_k),C^*(\bb F_k))$ is a nuclear pair is independent of the choice of $k\in \{2,3,\ldots\}\cup\{\infty\}$, a fact that will come up in our discussion of Kirchberg's QWEP problem.  We will also need the fact that if $p:G\to H$ is a surjective group morphism for which the canonical surjection $C^*(p):C^*(G)\to C^*(H)$ has a ucp lift (meaning a ucp map $\Phi:C^*(H)\to C^*(G)$ for which $C^*(p)\Phi$ is the identity on $C^*(H)$), then $(C^*(G),C^*(G))$ being a nuclear pair implies $(C^*(H),C^*(H))$ is a nuclear pair.  

\subsection{Kirchberg's QWEP problem}\label{sec3:kirchberg's}

It follows from the definition of the minimal tensor product that for any inclusion $\A\subseteq \cal B$ of \cstar-algebras, one has that $\A\otimes_{\min} \cal C\subseteq \cal B\otimes_{\min} \cal C$ (isometrically) for any other \cstar-algebra $\cal C$.  On the other hand, with the same setup, while there will always be a $*$-homorphism $\A\otimes_{\max}\cal C\to \cal B\otimes_{\max}\cal C$, this homomorphism need not be injective, that is, isometric.  One case, however, where this does hold is the inclusion $\A\subseteq \A^{**}$.  That is, it follows from the definitions, that $\A\otimes_{\max} \cal C\subseteq \A^{**}\otimes_{\max} \cal C$ isometrically for all \cstar-algebras $\A$ and $\cal C$.  

Suppose again that $\cal A\subseteq \cal B$ and further suppose, for the sake of argument, that there is a ucp map $\Phi:\cal B\to \A^{**}$ with $\Phi(a)=a$ for all $a\in \A$.  By a fact pointed out in the the previous subsection, we obtain a ucp (and thus contractive) map $\Phi\otimes I_{\cal C}:\cal B\otimes_{\max} \cal C\to \A^{**}\otimes_{\max} \cal C$.  By the observation made in the previous paragraph, it follows that the canonical map $\A\otimes_{\max} \cal C\to \cal B\otimes_{\max} \cal C$ is an isometric inclusion.

If $\A$ is a \cstar-subalgebra of $\cal B$ and $\Phi:\cal B\to \A$ is a linear map for which $\Phi(a)=a$ for all $a\in \A$, then a theorem of Tomiyama says that the following are equivalent:
\begin{itemize}
    \item $\Phi$ is cp
    \item $\Phi$ is contractive
    \item $\Phi$ is a \emph{conditional expectation}, that is, $\Phi(axb)=a\Phi(x)b$ for all $a,b\in A$ and $x\in B$. 
\end{itemize}  When such a map exists, we say that $\A$ is \emph{cp-complemented} in $\cal B$.  The nomenclature comes from Banach space theory, for a Banach space $X$ is complemented in a superspace $Y$ if and only if there is a contractive linear map $\Phi:Y\to X$ that is the identity on $X$.  In the previous paragraph, we merely had to posit the existence of a ucp map $\Phi:\cal B\to \A^{**}$, whence we call $\Phi$ a \emph{weak conditional expecation} and say that $\A$ is \emph{weakly cp-complemented in $\cal B$}.  Consequently, we  proved that if $\cal A$ is weakly complemented in $\cal B$, then $\A\otimes_{\max} \cal C\subseteq \cal B\otimes_{\max} \cal C$ isometrically for any other \cstar-algebra $\cal C$.   With more work, one can actually show that the converse of this observation holds as well.  In fact, by the surjective universality of $C^*(\bb F_\infty)$, we see that $\A$ is weakly cp-complemented in $\cal B$ if and only if $\A\otimes_{\max}C^*(\bb F_\infty)\subseteq \cal B\otimes_{\max}C^*(\bb F_\infty)$ isometrically.

There are two notable examples of cp-complemented inclusions worth pointing out now:
\begin{itemize}
    \item If $\cal M$ is a finite von Neumann algebra, then any von Neumann subalgebra $\cal N$ of $\cal M$ is cp-complemented.  To see this, fix a trace $\tau$ on $\cal M$ and note that $L^2(N,\tau)$ is a closed subspace of $L^2(M,\tau)$.  One shows that the orthogonal projection $L^2(\cal M,\tau)\to L^2(\cal N,\tau)$ actually restricts to a conditional expectation $M\to N$.
    \item If $H$ is a subgroup of $G$, then $C^*(H)$ is cp-complemented in $C^*(G)$.  To see this, one shows that the map $\Phi:G\to \bb C[H]$ defined by setting $\Phi(g)=g$ for all $g\in H$ while $\Phi(g)=0$ for all $g\in G\setminus H$ extends to a conditional expectation $C^*(G)\to C^*(H)$.
\end{itemize}

Returning to the general situation, if $\A$ is weakly cp-complemented in every superalgebra (or equiv. in $\cal B(\cal H)$), then we say that $\A$ has the \emph{weak expectation property (or WEP)}.  An insight of Kirchberg \cite{K} was to use his theorem proving that $(C^*(\bb F_\infty),\cal B(\cal H))$ is a nuclear pair to provide an alternate ``test'' for having WEP:

\begin{thm}
For a \cstar-algebra $\A\subseteq \cal B(\cal H)$, the following are equivalent:
\begin{enumerate}
    \item $\A$ has the WEP.
    \item For every \cstar-algebra $C$, $\A\otimes_{\max} \cal C\subseteq \cal B(\cal H)\otimes_{\max}\cal C$ isometrically.
    \item $(\A,C^*(\bb F_\infty))$ is a nuclear pair.
\end{enumerate}
\end{thm}

\begin{proof}
We already observed the equivalence of (1) and (2).  Now suppose that $\cal A$ has WEP.  We then have $$\A\otimes_{\max}C^*(\bb F_\infty)\subseteq \cal B(\cal H)\otimes_{\max}C^*(\bb F_\infty)=\cal B(\cal H)\otimes_{\min}C^*(\bb F_\infty)\supseteq \A\otimes_{\min}C^*(\bb F_\infty),$$ that is, $(\A,C^*(\bb F_\infty))$ is a nuclear pair.  Conversely, suppose that $(\A,C^*(\bb F_\infty))$ is a nuclear pair.  It suffices to show that $\A\otimes_{\max}C^*(\bb F_\infty)\subseteq \cal B(\cal H)\otimes_{\max} C^*(\bb F_\infty)$.  However, this follows immediately from the assumption, Kirchberg's theorem, and the fact that $\otimes_{\min}$ preserves inclusions.
\end{proof}
 
 What are the ramifications of assuming that $C^*(\bb F_\infty)$ itself has the WEP, that is, $(C^*(\bb F_\infty),C^*(\bb F_\infty))$ is a nuclear pair?  First, as observed above, this is equivalent to $(C^*(\bb F_k),C^*(\bb F_k))$ being a nuclear pair for any fixed $k\geq 2$.  Next, if we define the \emph{QWEP} to be the property that a \cstar-algebra is a quotient of a \cstar-algebra with WEP, then $C^*(\bb F_\infty)$ having the WEP implies that all separable \cstar-algebras have the QWEP.  We note that it is common to see both phrases ``$\cal A$ has the QWEP'' and ``$\cal A$ is QWEP'' (although the latter is of course grammatically incorrect). 
 
 Conversely, suppose that all separable \cstar-algebras have the QWEP.  Then certainly $C^*(\bb F_\infty)$ has the QWEP.  However, $C^*(\bb F_\infty)$ has another property, the so-called \emph{lifting property (or LP)} which, when combined with QWEP, actually implies the WEP.  A \cstar-algebra $\A$ has the lifting property if, for any ucp map $\Phi:\A\to \cal B/\cal J$, where $\cal B$ is a \cstar-algebra and $\cal J$ is a closed, two-sided ideal in $\cal B$, there is a ucp map $\Psi:\A\to \cal B$ for which $\pi\circ \Psi=\Phi$, where $\pi:\cal B\to \cal B/\cal J$ is the canonical quotient map.  Said more casually, $\A$ has the LP if every ucp map into a quotient \cstar-algebra has a ucp lift.  Now suppose that $\A$ has the LP and the QWEP as witnessed by the quotient map $q:\cal B\to \cal A$ with $\cal B$ having the WEP.  Let $\Psi:\cal A\to \cal B$ be a ucp lift of the identity map $\cal A\to \cal A$, that is, $q\circ \Psi=I_{\A}$.  Then by applying the ucp lift $\Psi\otimes I_{C^*(\bb F_\infty)}$ of $q\otimes I_{C^*(\bb F_\infty)}$, we see that $\cal A$ also has the WEP.
 
 We have thus arrived at the following:
 
 \begin{thm}
 The following assertions are equivalent:
 \begin{enumerate}
     \item $C^*(\bb F_\infty)$ has the WEP.
     \item For some (equiv. any) $k\in \{2,3,\ldots\}\cup\{\infty\}$, $(C^*(\bb F_k),C^*(\bb F_k))$ is a nuclear pair.
     \item Every separable \cstar-algebra has the QWEP.
 \end{enumerate}
 \end{thm}
 
 Any of the above equivalent statements is known as \emph{Kirchberg's QWEP problem}.  (We might be tempted to follow the CEP's lead and call this the QWEPP, but that looks a bit silly.)  As mentioned above, QWEP combined with LP implies WEP.  It turns out that it suffices to consider a ``local'' version of LP, aptly called the \emph{local lifting property (or LLP)} and the same argument works, that is QWEP together with LLP implies WEP.  Thus, another equivalent formulation of the QWEP problem is the statement that LLP implies WEP.
 
 We mention one other equivalent formulation of the QWEP problem that is not relevant for our particular story but is fascinating nonetheless:  the QWEP problem is equivalent to the statement that $C^*(\bb F_\infty\times \bb F_\infty)$ has a faithful tracial state.  What makes this interesting is that this is true for the reduced group \cstar-algebra $C^*_r(\bb F_\infty\times \bb F_\infty)$ (simply because it is true for any reduced group \cstar-algebra) \emph{and} it is true for $C^*(\bb F_\infty)$ (a result due to Choi).
 

The classes of WEP and QWEP algebras enjoy a number of closure properties relevant to the proofs that follow.  Rather than enumerate them all now, we will simply quote them when we need them later in the paper.

\subsection{From CEP to QWEP}\label{sec3:from}

In this section, we show how a positive solution to the CEP implies a positive solution to the QWEP problem.  While these statements are indeed equivalent, we focus on the direction that we need in order to give a negative solution to CEP.

So how does CEP get involved in a story about \cstar-algebras?  The first clue is that, for von Neumann algebras, the WEP had already been well-studied and is referred to as \emph{injectivity}.  A not so trivial result is that any hyperfinite von Neumann algebra is injective, whence $\R$ is injective.  The extremely deep work of Connes in \cite{Connes}, where the CEP originally comes from, proved the converse, namely any separable injective II$_1$ factor must be hyperfinite, and thus isomorphic to $\R$.  Thankfully we do not need this result in our story, although the proof that $\R$ is injective (and thus has WEP) is difficult enough.  

Now that we know that $\R$ is injective, so is $\ell^\infty (\R)$ as WEP is closed under the formation of direct sums.  Since $\R^\u$ is a \cstar-algebra quotient of $\ell^\infty(\R)$, we see that $\R^\u$ is QWEP!  Okay, it smells like we are getting closer.

Now suppose that $\cal M$ is a tracial von Neumann algebra that embeds in $\R^\u$ in a trace-preserving manner.  Without loss of generality, let us assume that $\cal M$ is simply a subalgebra of $\R^\u$.  By a fact pointed out above, this means that $\cal M$ is cp-complemented in $\R^\u$.  Since QWEP is preserved by (weakly) cp-complemented inclusions, we conclude that $\cal M$ is also QWEP.  

We have thus arrived at the statement:  a positive solution to CEP implies all finite von Neumann algebras are QWEP!

But we are still talking about von Neumann algebras.  How do we bridge the gap into talking about \cstar-algebras?  Well, recall that every \cstar-algebra $\A$ has a canonically associated von Neumann algebra $\A^{**}$.  Since $\A$ is tautologically weakly cp-complemented in $A^{**}$, in order to show that $\A$ has QWEP, it suffices to show that $\cal A^{**}$ has QWEP (again using the closure of QWEP under weakly cp-complemented subalgebras).  

While $\cal A^{**}$ is a separable von Neumann algebra, it may not be finite.  How do we get CEP to help us with non-finite von Neumann algebras?

Given any von Neumann algebra $\cal M$, there is an important one-parameter group $(\sigma_t^\varphi)$ of automorphisms of $\cal M$, known as the \emph{modular group}.  When $\cal M$ is finite, the modular automorphism group is trivial and thus plays no role.  But in the general theory, it is an indispensible tool.  (For all of the fancy type III material discussed in this paragraph, Takesaki's book \cite{takesaki} is the canonical reference.)  Akin to the semidirect product construction in group theory, there is a \emph{crossed product construction} that associates to any group acting on a von Neumann algebra a larger von Neumann algebra where this action is implemented by unitaries.  Thus, we are entitled to consider the crossed product algebra $\cal M\rtimes_{\sigma_t^\varphi}\bb R$ corresponding to the action of $\bb R$ on $\cal M$ via the modular automorphism group.  A serious theorem of Takesaki states that $\cal M\rtimes_{\sigma_t^\varphi} \bb R$ is semifinite.  We came across semifinite factors above.  For a general von Neumann algebra, we can take semifinite to mean that the algebra contains an increasing union of finite subalgebras whose union generates the von Neumann algebra.  Since QWEP is preserved under unions and a von Neumann algebra is QWEP if it contains a WOT-dense $*$-subalgebra with QWEP, we see that $\cal M\rtimes_{\sigma^t_\varphi}\bb R$ has QWEP.  An alternative approach is to use the fact that a von Neumann algebra is QWEP if and only if all of the factors involved in its direct integral decomposition are QWEP.  Thus, to show that a semifinite von Neumann algebra is QWEP, it suffices to consider the case of factors.  But then a semifinite factor is of the form $\cal M\bar\otimes \cal B(\cal H)$ for a finite factor $\cal M$, and one can use the fact that the von Neumann algebra tensor product of QWEP von Neumann algebras is again QWEP.  Either way, we now know that $\cal M\rtimes_{\sigma_t^\varphi} \bb R$ is QWEP.

Finally, it is a general fact that any von Neumanna algebra $\cal M$ is always cp-complemented in any crossed product $\cal M\rtimes_\alpha G$; since $\cal M\rtimes_{\sigma_t^\varphi} \bb R$ is QWEP for any von Neumann algebra $\cal M$, it follows that $\cal M$ itself is also QWEP.  Applying this fact to $\cal M=\A^{**}$, we see that $\A^{**}$, and thus $\A$, are also QWEP for any \cstar-algebra $\A$.  This finishes the proof that a positive solution to CEP implies a positive solution to the QWEP problem.

As mentioned before, a positive solution to the QWEP problem implies a positive solution to the CEP.  The proof involves the theory of \emph{amenable traces}, which we will not go into now, but which will be important in our alternate derivation of the failure of CEP from $\mip^*=\operatorname{RE}$ given in Subsection \ref{sec7:synchronous} below.

\section{A crash course in complexity theory}

In this section, we introduce the basic notions from (classical) complexity theory needed to understand the statement of the result $\mip^*=\operatorname{RE}$.  Essentially all of this material (apart from the business about nonlocal games) was taken from the book \cite{complex}.
\subsection{Turing machines}\label{sec4:Turing}

A \emph{Turing machine} is one of the more popular mathematical formulations of an idealized computing device.  Formally, a Turing machine is a pair $\mathbf{M}=(Q,\delta)$, where $Q$ is a finite set of \emph{states} of the machine and $\delta:Q\times \{0,1,\Box,\triangle\}^3\to Q\times \{0,1,\Box,\triangle\}^2\times \{L,S,R\}^3$ is the \emph{transition function}; here $\Box$ and $\triangle$ are two special symbols whose significance will be seen shortly.  We always assume that $Q$ contains two special states, namely the \emph{start state} $q_{start}$ and the \emph{halting state} $q_{halt}$.

Throughout, for any $n\in \bb N$, $\{0,1\}^n$ denotes the set of binary strings of length $n$ while $\{0,1\}^*:=\bigcup_{n\in \bb N}\{0,1\}^n$ denotes the set of all finite binary strings.  Given $z\in \{0,1\}^*$, $|z|$ denotes the length of the string $z$.

Here is how one should envision the computation performed by the Turing machine $\bf M$ upon some input $z\in \{0,1\}^*$.  The machine contains three \emph{tapes}, which are one-way infinite strips containing boxes on which, at any given moment in the computation, contain exactly one symbol from $\{0,1,\Box,\triangle\}$.  The first tape is the \emph{input tape}, the second tape is the \emph{work tape}, and the last tape is the \emph{output tape}.  At the beginning of the computation, the input tape has the \emph{start symbol} $\triangle$ in the first box, then the input string $z$ in the next $|z|$ boxes, and then the remainder of the boxes contain the \emph{blank symbol} $\Box$.  Both the work tape and the output tape contain the start symbol $\triangle$ in the first box and then blank symbols $\Box$ in the remaining boxes.  One envisions each tape having a ``tape head'' which is placed over exactly one box in the tape at any given moment during the computation; the tape head for the input tape can read the symbol in that box while the tape head for the other two tapes can both read the symbol in that box and potentially change it to a new symbol.

The Turing machine begins the computation in the start state $q_{start}$ with the tape head above the leftmost box (which contains the start symbol $\triangle$) for each tape.  In general, at any given moment during the computation, the Turing machine is in some state $q\in Q$ with tape heads reading boxes $k_1,k_2,k_3\in \bb N$ (representing how far they are from the beginning of their respective tape) and with symbols $s_1,s_2,s_3\in \{0,1,\Box,\triangle\}$ inside of each of the boxes being read.  The Turing machine then computes $\delta(q,s_1,s_2,s_3)$, obtaining the tuple $(q',s_2',s_3',I_1,I_2,I_3)$, which should be interpreted as follows:
\begin{itemize}
    \item The box in the work tape (resp. output tape) that the tape head is reading should have its contents replaced by $s_2'$ (resp. $s_3'$).
    \item The tape head for the input tape should move to the left if $I_1=L$, to the right if $I_1=R$, and should stay in the same place if $I_1=S$.  Similar actions should be taken corresponding to $I_2$ for the work tape and $I_3$ for the output tape.  If any tape head is at the leftmost box and the instruction is $L$, then the tape head should also stay in the same place.
    \item After executing these acts, the Turing machine should now enter state $q'$.
\end{itemize}

The machine continues ``running'' in this fashion.  If the machine ever enters the state $q_{halt}$, then the machine ``stops running'', that is, no further modification of the three tapes will take place.  In this case, the \emph{output} of the computation upon input $z$ is the longest initial string on the output tape not containing any blank symbols.  (If all the symbols are blank, then the output is considered the empty string).

Every Turing machine $\bf M$ \emph{computes} a partial function $f^{\bf M}:\{0,1\}^*\rightharpoonup \{0,1\}^*$ whose domain consists of those strings $z\in \{0,1\}^*$ for which $\bf M$ halts upon input $z$; in this case, we define $f^{\bf M}(z)$ to be the corresponding output.  We sometimes abuse notation and identify $f^{\bf M}$ with $\bf M$ itself, that is, we may write $\mathbf{M}(z)$ instead of $f^{\bf M}(z)$.  We say a partial function $f:\{0,1\}^*\rightharpoonup \{0,1\}^*$ is \emph{computable} if $f=f^{\bf M}$ for some Turing machine $\bf M$.

Given a function $T:\bb N\to \bb N$, we say that the Turing machine $\bf M$ \emph{runs in $T(n)$-time} if, for any input $z\in \{0,1\}^*$, upon input $z$, $\bf M$ halts in at most $T(|z|)$ steps.  Note that if $\bf M$ runs in $T(n)$-time for some function $T$, then $f^{\bf M}$ is a total function.  We say that $\bf M$ is a \emph{polynomial time} (resp. \emph{exponential time}) Turing machine if $\bf M$ runs in $Cn^c$- (resp. $C2^{n^c}$-) time for some constants $C>0$ and $c\geq 1$.

A \emph{language} is simply a subset $\mathbf{L} \subseteq \{0,1\}^*$.  We identify a language $\bf L$ with its characteristic function $\chi_{\bf L}:\{0,1\}^*\to \{0,1\}$.  Consequently, it makes sense to speak of $\bf L$ being computable by a Turing machine.  Usually a language is described in terms of some mathematical problem under consideration, e.g. the set of finite graphs that can be $3$-colored.  The implict assumption is that there is some natural (and effective) way of coding the set of such graphs as a set of finite binary strings.  In the sequel, for all languages introduced in this manner, we assume that the reader can figure out how such a coding might be performed.

Turing machines are one of several mathematically precise models for computation; other alternatives include register machines and the class of recursive functions.  However, all known models of computation lead to precisely the same class of computable functions.  This is evidence for the \emph{Church-Turing thesis}, which states that this common class of functions coincides with our heuristic notion of what a computable function should be.  (See \cite[Chapter 3]{enderton} for more on this.)  One can even formulate a stronger version of the thesis, which states that even when taking into account effective matters, that is, how ``fast'' one can compute a function, the choice of model is still irrelevant.  (It is plausible that quantum computers could pose a serious threat to the strong Church-Turing thesis.)  The import of the strong Church-Turing thesis for us in these notes is that, in the sequel, when claiming that a certain problem can be solved in a certain efficient manner, we never need to actually write down the Turing machine that witnesses this fact.  Instead, one can write down an argument using ``pseudo-code'' and the reader can (if they choose to) convert the pseudo-code into an actual Turing machine program.

\subsection{Some basic complexity classes}\label{sec4:some}

A \emph{complexity class} is simply a collection of languages.  The most interesting complexity classes are those defined by some sort of condition saying that the languages in the class represent efficiently computable (or verifiable, as we shall shortly see) problems.

The complexity class $\P$ is defined to be the class of languages $\bf L$ such that membership in $\bf L$ can be decided by a Turing machine in polynomial time, that is, $\chi_{\bf L}$ can be computed by a polynomial time Turing machine.  For example, the set of connected graphs is a language that belongs to $\P$ (as witnessed by, say, the breadth first search algorithm). 

The complexity class $\EXP$ is defined in the same manner as $\P$, replacing polynomial time by exponential time.  The \emph{time hierarchy theorem} implies that $\P\subsetneq \EXP$.

Sometimes it is too difficult to come up with an algorithm that efficiently decides membership in a particular language while it is the case that if someone were to ``hand you'' a proof that a certain string belonged to the language, then you could efficiently verify that the proof was indeed correct.  The complexity class $\NP$ captures this idea.  More precisely, the complexity class $\NP$ consists of those languages $\bf L$ for which there is a polynomial time Turing machine $\bf M$ and a polynomial $p(n)$ such that:
\begin{itemize}
    \item for all $z\in \bf L$, there is $w\in \{0,1\}^{p(|z|)}$ for which $\mathbf {M}(z,w)=1$.
    \item for all $x\notin \bf L$ and for all $w\in \{0,1\}^{p(|z|)}$, $\mathbf{ M}(z,w)=0$.
\end{itemize}

In the above definition, one thinks of $w$ as the ``proof'' that $z\in \bf L$; other commonly used terms for $w$ are ``witness'' and ``certificate.''  One often envisions this situation using two fictious players, a verifier and a prover.  If $z\in \bf L$, the prover hands the verifier a proof $w$ that $z$ indeed belongs to $\bf L$; the prover has unlimited computation power in this regards.  In order for the verifier to be able to efficiently check that the proof indeed works, the proof cannot be too long (or else the verifier will not even be able to read the entire proof), hence the polynomial length requirement.  Moreover, if $z\notin \bf L$, then there should be no proof that $z$ belongs to $\bf L$, whence the second condition.

It is clear that $\P\subseteq \NP$.  While intuitively it seems clear that this inclusion should be proper (there ``ought'' to be problems that are impossible to efficiently decide but yet there are always proofs that are efficiently verifiable), this fact has yet to be established and remains one of the more famous open problems in mathematics.

We also note that $\NP\subseteq \EXP$ as one can check all of the exponentially many possible certificates for a given string in exponential time.

An example of a language in $\NP$ is the set of codes for pairs $(G,k)$, where $G$ is a finite graph that contains an independent set of size $k$; a certificate for a given pair is simply an independent set of size $k$.  This language is unlikely to be in $\P$.  Indeed, this is an example of a so-called \emph{$\NP$-complete problem}, meaning that it is as difficult as any other problem in $\NP$ (in a precise sense), whence if it belongs to $\P$, then so do all languages in $\NP$ and $\P=\NP$.  Another example of a language in $\NP$ is the set of codes for pairs $(G_1,G_2)$ of finite graphs that are isomorphic; the certificate here is the isomorphism between the graphs.  This problem, however, is unlikely to be $\NP$-complete (see \cite[Section 8.4]{complex}).

An alternative way of defining the class $\NP$ is to use \emph{nondeterministic Turing machines}, which is actually the original definition and explains the terminology ($\NP$ stands for nondeterministic polynomial-time).  A nondeterministic Turing machine is defined exactly like a deterministic one except that it has two transition functions rather than one.  Consequently, rather than there being a single (determinstic) sequence of steps during a computation upon a given input, there is an entire binary tree of such computations, for at every step during a computation, one can apply either of the two transition functions.  One additional difference is that instead of a single halting state $q_{halt}$, we now have two halting states called $q_{accept}$ and $q_{reject}$.  We say that the nondeterministic Turing machine $\bf M$ outputs $1$ on input $z$ if there is \emph{some} sequence of steps which causes the machine to reach $q_{accept}$.  If every sequence of steps causes the machine to reach $q_{reject}$, then we say that $\bf M$ outputs $0$.  If every sequence of computations results in either $q_{accept}$ or $q_{reject}$ in time $T(|z|)$, then we say that $\bf M$ runs in time $T(n)$.  It thus makes sense to speak of polynomial (resp. exponential time) nondeterminstic Turing machines.

It can easily be verified that a language $\bf L$ belongs to $\NP$ if and only if there is a polynomial time nondeterministic Turing machine $\bf M$ such that $f^{\mathbf{M}}=\chi_{\bf L}$.  Moreover, using exponential time nondeterministic Turing machines, we can also define the complexity class $\NEXP$.   Of course, using nondeterministic Turing machines that run in doubly exponential time, one can also define $\NEEXP$ (this will come up later).  A nondeterministic version of the time hierarchy theorem guarantees $\NP\subsetneq \NEXP \subsetneq \NEEXP$.

At this point, we have $P\subseteq \NP\subseteq \EXP\subseteq \NEXP$ with $\P\subsetneq \EXP$ and $\NP\subsetneq \NEXP$.  One can also use a ``padding'' argument to show that if $\EXP=\NEXP$, then $\P=\NP$.

So far we have only been concerned with time efficiency.  One can instead consider ``space efficiency.''  We will only consider the class $\PSPACE$, which consists of all languages $\bf L$ for which there is a Turing machine such that, upon input $z$, decides whether or not $z\in \bf L$ using only a polynomial amount of work space.  It is clear that $\P\subseteq \PSPACE$.  It is also fairly easy to see that $\NP\subseteq \PSPACE$, for one can simply check all possible certificates, erasing one's work after each individual check, thus using only a polynomial amount of space.  A slightly less obvious inclusion is $\PSPACE\subseteq \EXP$; the proof uses the notion of a configuration graph for a computation.  So, to update our state of knowledge, we have $\P\subseteq \NP\subseteq \PSPACE\subseteq \EXP\subseteq \NEXP$.  It is not known if the inclusion $\NP\subseteq \PSPACE$ is proper.  In fact, it is not even known if the inclusion $\PSPACE\subseteq \NEXP$ is proper (this is relevant for our later discussion).  Of course, if $\PSPACE=\NEXP$, then $\EXP=\NEXP$, whence $\P=\NP$.

We end this section with the definition of the class $\operatorname{BPP}$.  Although it will not play a direct role in the story to follow, it will make a later pill easier to swallow.  We return to the setting of nondeterministic Turing machines, but this time we count the proportion of computations that output $\mathbf{M}(z)=1$.  We say that the language $\bf L$ belongs to the class $\operatorname{BPP}$ if there is a nondeterministic Turing machine $\bf M$ such that, upon $z\in \{0,1\}^*$, the probability that a random nondeterministic computation agrees with $\chi_{\mathbf{L}}(z)$ is at least $\frac{2}{3}$.  Just as in the case of $\NP$, there is a formulation using deterministic Turing machines:  $\bf L$ belongs to $\BPP$ if and only if there is a Turing machine $
\bf M$ and a polynomial $p(n)$ such that, for every string $z\in \{0,1\}^*$, the probability that a random $r\in \{0,1\}^{p(|z|)}$ is such that $\mathbf{M}(z,r)=\chi_{\mathbf{L}}(z)$ is at least $\frac{2}{3}$.  We remark that the choice of $\frac{2}{3}$ is fairly arbitrary; by repeating the computation several (but still a reasonable number of) times and taking the majority result of the computations, we can replace $\frac{2}{3}$ with a probability as close to $1$ as one desires.  The class $\BPP$ is contained in $\EXP$ as one can check all random bits and compute the probability that a random choice yields $1$ or $0$.

If $\bf L$ is a language in $\BPP$ and one repeats the computation a sufficient number of times to achieve a probability of, say, $0.99$, then one can be fairly certain that the result of the probabilistic computation is the truth, and thus $\BPP$ seems like a fairly good substitute for $\P$.  In fact, there are complexity-theoretic reasons for believing that $\BPP$ might coincide with $\P$ (see \cite[Chapter 16]{complex}).  

\subsection{Interactive proofs}\label{sec4:interactive}

We now imagine the situation where rather than the prover just handing the verifier a proof, the prover and the verifier are allowed to interact.  Given an input, the verifier can ask the prover a question, the prover can answer that question, then based on that answer, the verifier can ask the prover another question to which the prover can reply, and so on, for a certain number of rounds.  Each time, the verifier's question and the prover's answer depend on the entire sequence of questions and answers obtained up to that point (as well as the input).  After this discussion, the verifier can decide whether or not to accept.  Once again, the verifier uses a polynomial-time Turing machine to choose which questions to ask and whether or not to accept at the end of the conversation while the prover has no computational limitations.

It is not too difficult to verify that, with this description of interactive proof, the corresponding complexity class would simply be $\NP$ in disguise.  Indeed, one can just use the conversation, or ``transcript'' as it is usually called, as the certificate.  However, combining this idea with a randomized process as discussed at the end of the last subsection does lead to a class with more computational power (although, interestingly enough, a class we have already seen before).

In order to define this class, we fix $k\in \bb N$ (although one could actually work with a polynomial-time computable $k:\bb N\to \bb N$ instead) and a polynomial $p(n)$.  Assume that we also have a Turing machine $\bf V$ (now that we are really viewing the machine as a verifier, we have replaced $\bf M$ with $\bf V$) such that, for all $z\in \{0,1\}^*$, all $r\in \{0,1\}^{p(|z|)}$, and all strings $a_1,\ldots,a_{2k}\in \{0,1\}^*$, $\bf V$ halts upon input $(z,r,a_1,\ldots,a_{2i})$ in time polynomial in $|z|$ for all $i=0,\ldots,k$.  We then imagine a prover $P:\{0,1\}^*\to \{0,1\}^*$ interacting with $\bf V$ as follows.  First, the verifier randomly selects $r\in \{0,1\}^{p(|z|)}$ and computes $a_1:=\mathbf{V}(z,r)$; this is $\bf V$'s first ``question'' to $P$.  $P$ then responds with the ``answer'' $a_2:=\mathbf{P}(z,a_1)$.  (Note that $P$ does not have access to the random string $r$; one says that $\bf V$ is using ``private coins.''  It turns out that for what we are going to define below, one can also use ``public coins'' that the prover is aware of.)  This constitutes the first ``round'' of their interaction.  The verifier then asks $P$ their second question $a_3:=\mathbf{V}(z,r,a_1,a_2)$ and $P$ responds with $a_4:=P(z,a_1,a_2,a_3)$, completing the second round of interaction.  This is repeated for a total of $k$ rounds.  $\mathbf V$ then returns their decision $\mathbf V(z,r,a_1,\ldots,a_{2k})\in \{0,1\}$, indicating whether or not they accept the prover's answers as constituting evidence that $z$ indeed belongs to $\bf L$.

The complexity class $\IP$ is defined to be the collection of those languages $\bf L$ for which there is a Turing machine $\bf V$ as above such that:
\begin{itemize}
    \item If $z\in \bf L$, then there is a prover $P$ such that the probability that a random bit $r$ causes $\bf V$ to accept is at least $\frac{2}{3}$.
    \item If $z\notin \bf L$, then no prover can cause $\bf V$ to accept more than $\frac{1}{3}$ of the time.
\end{itemize}

The probabilities $\frac{2}{3}$ and $\frac{1}{3}$ above are called the \emph{completeness} and \emph{soundness} parameters respectively.  As in the case of $\operatorname{BPP}$, they are somewhat arbitrary; any completeness parameter strictly larger than $\frac{1}{2}$ will define the same class.  It turns out that one can even replace the completeness parameter by $1$ without changing the class; this property of $\IP$ is called \emph{perfect completeness}.

A nice example of a language in $\IP$ is the collection of pairs $(G_1,G_2)$ of finite graphs that are \emph{not} isomorphic.  Note that this class is not obviously in $\NP$ for there are too many possible functions that could serve as an isomorphism.  There is however a simple interactive proof for this class.  Indeed, the verifier randomly selects $i\in \{1,2\}$ and then randomly selects a permutation $\sigma$ on the number of vertices of $G_i$, obtaining a graph $H$ isomorphic to $G_i$.  The verifier then sends the graph $H$ to the prover as its ``question.''  The prover then responds with a bit $a\in \{1,2\}$, which represents its guess as to which of the two graphs $G_1$ or $G_2$ the verifier selected randomly.  The verifier accepts if and only if the prover guessed the chosen graph correctly.  Note that if $G_1\not\cong G_2$ (that is, if the pair $(G_1,G_2)$ belongs to the class), then the prover can always respond correctly, for the prover can just figure out whether or not $H$ is isomorphic to $G_1$ or to $G_2$.  (Do not forget that the prover is all-powerful!)  However, if $G_1\cong G_2$ (that is, $(G_1,G_2)$ does not belong to the class), then $H$ is isomorphic to both $G_1$ and $G_2$ and thus the prover (regardless of its unlimited power) can do no better than simply guessing which graph was chosen by the verifier, thus only convincing the verifier at most half of the time

Given a verifier $\bf V$ as in the definition of $\IP$, one can compute in $poly(|z|)$-space the optimal prover strategy.  This shows that $\IP\subseteq \PSPACE$.  A landmark theorem in the subject shows that in fact we have equality:

\begin{thm}[Fortnow, Karloff, Lund, Nisan \cite{FKLN}]
$\IP=\PSPACE$.
\end{thm}

Recall that randomization alone likely does not achieve anything new (earlier we remarked that $\P=\operatorname{BPP}$ is likely) and, similarly, interaction alone does not achieve anything too new (as we just recover $\NP$).  However, by combining randomization with interaction bumps us up to $\PSPACE$ (which is likely bigger than $\NP$).

But why stop at one prover?  One can consider interactions as above but allowing for \emph{multiple} provers to interact with the verifier.  It should be emphasized that the provers are not allowed to interact with each other during the interaction, but only with the verifier.  They can, however, have a meeting before the interaction starts and decide upon a strategy that they will use while interacting with the verifier.  In other words, the provers are cooperating but noncommunicating.  If the provers use deterministic strategies as above, we arrive at the complexity class $\MIP$.  By allowing different kinds of strategies (in particular, those that employ quantum methods), we arrive at variations of $\MIP$, such as the famous $\MIP^*$ appearing in the equation $\MIP^*=\operatorname{RE}$.  

It turns out that the complexity class $\MIP$ is unchanged if one restricts to just two provers and one round of interaction; we thus make that default assumption from now on.  By ignoring one of the provers, we clearly have that $\IP\subseteq \MIP$.  As with $\IP$, one can also achieve perfect completeness.    

With two provers, one can now utilize ``police-style'' interrogation tactics.  This makes it possible for the verifier to read polynomially many random portions of an exponentially long proof and come to a conclusion that with high probability agrees with the truth.  A formalization of this idea yields another major theorem in the subject:

\begin{thm}[Babai, Fortnow, Lund \cite{BFL}]
$\MIP=\NEXP$.
\end{thm}

As mentioned before, it is believed that $\PSPACE\not=\NEXP$ (else $\P=\NP$).  Consequently, it appears that the jump from one to more than one prover does indeed lead to a computationally superior verifier.

\subsection{Nonlocal games}\label{sec4:nonlocal}

It will be useful to recast our description of the class $\MIP$ in terms of so-called \emph{nonlocal games}, a certain collection of two-person games.  (The terminology ``nonlocal'' comes from the connection with Bell's theorem on quantum nonlocality, as we discuss later.)

Consider a language $\bf L$ in $\MIP$ as witnessed by the polynomial-time verifier $\bf V$.  Given input $z$ and a sequence of random bits $r$, by computing $\mathbf{V}(z,r)$, we are really computing the two ``questions'' $x$ and $y$ (sequences of bits of length polynomial in $|z|$) that are being sent to the two provers, who we will call Alice and Bob, following typical quantum information nomenclature. Alice and Bob, employing their deterministic strategies $A$ and $B$, then respond with their ``answers'', say $a:=A(x)$ and $b:=B(y)$, and then the prover calculates $\mathbf{V}(z,r,x,y,a,b)$ to decide whether or not to accept their answers.  Whether or not $z$ belongs to $\bf L$ then corresponds to the expected value over a randomly chosen $r$ that the verifier returns $\mathbf{V}(z,r,x,y,a,b)=1$.  Note that the polynomial time requirement on $\bf V$ allows us to assume that the set of possible answers only contains bits that are of size at most some fixed polynomial in $|z|$. 

This reformulation leads us to the following notion:
A \emph{nonlocal game with $k$ questions and $n$ answers} is a pair $\frak G=(\pi,D)$, where $\pi$ is a probability distribution on $[k]\times [k]$ and $D:[k]\times [k]\times [n]\times [n]\to \{0,1\}$ is the decision predicate for the game.  Here, $[k]:=\{1,\ldots,k\}$ and similarly for $[n]$.  A \emph{strategy} for the players consists of a conditional probability $p(a,b|x,y)$ expressing the probability that Alice and Bob respond with answers $a$ and $b$ if they are asked questions $x$ and $y$ respectively.  We view such a strategy $p$ as an element of $[0,1]^{k^2n^2}$.  Above, we only considered \emph{deterministic strategies}, namely those $p$ for which there are functions $A,B:[k]\to [n]$ such that $p(A(x),B(y)|x,y)=1$ for all $x,y\in [k]$.  We let $C_{det}(k,n)\subseteq [0,1]^{k^2n^2}$ denote the set of such deterministic strategies.  Later, we will consider several other sets of strategies.

Given a strategy $p$, the \emph{value of the game $\frak G$ with respect to the stratey $p$} is the expected value the players win $\frak G$ if they play according to $p$, that is, 
$$\val(\frak G,p):=\sum_{(x,y)\in [k]\times [k]}\pi(x,y)\sum_{(a,b)\in [n]\times [n]}D(x,y,a,b)p(a,b|x,y).$$

We set $\val(\frak G):=\sup_{p\in C_{det}(k,n)}\val (\frak G,p)$ and refer to this as the \emph{classical value} of the game $\frak G$.



We can now rephrase the definition of $\MIP$ in terms of nonlocal games:  a language $\mathbf{L}$ belongs to $\MIP$ if and only if there is an ``efficient mapping'' $z\mapsto \frak G_z$ (in the precise sense described earlier in this subsection) so that:
\begin{itemize}
    \item If $z\in \mathbf{L}$, then $\val(\frak G_z)\geq \frac{2}{3}$.
    \item If $z\notin \mathbf{L}$, then $\val(\frak G_z)\leq \frac{1}{3}$.
\end{itemize}

The class $\MIP^*$ appearing in the result $\MIP^*=\operatorname{RE}$ is defined in the analogous way except that we replace classical strategies by quantum strategies.  But first, an interlude to explain all things quantum.

\section{A quantum detour}\label{sec5}

In this section, we introduce the quantum prerequisites necessary to understand the definition of the complexity class $\mip^*$.  Our presentation of quantum mechanics is fairly standard and can be found in any good textbook on quantum mechanics.  As mentioned above, we also found Paulsen's lecture notes \cite{paulsen} very helpful as well.

\subsection{Quantum measurements}\label{sec5:quantum}

In quantum mechanics, one associates to each physical system a corresponding Hilbert space $\cal H$.  The \emph{state} of the system at any given time is given by a unit vector $\xi\in \cal H$.  The state of the system evolves linearly according to a certain partial differential equation (the \emph{Schr\"odinger equation}) \textbf{until it is measured}.  A measurement should be thought of as an experiment on the system which has a finite number, say $n$, possible outcomes.  (There are also experiments that can have a countably infinite set of outcomes, say the infinite discrete set of energies of some particle, or even a continuum of outcomes, say when measuring the position or momentum of a particle; for the purposes of this paper, it suffices to focus on the case of finitely many outcomes.)  Formally, a measurement with $n$ outcomes consists of $n$ bounded operators $M_1,\ldots,M_n\in \cal B(\cal H)$.  The \emph{Born rule} states that, if the state of the system is $\xi$ upon measurement, then the probability that the $i^{\text{th}}$ outcome happens is given by $\|M_i\xi\|^2$.  Furthermore, in case the $i^{\text{th}}$ outcome is measured, the \emph{collapse dynamics} tells us that the state of the system instantaneously (and discontinuously) changes to $M_i(\xi)/\|M_i(\xi)\|$.  Since the sum of the outcome probabilities must be $1$, we see that
$$1=\sum_{i=1}^n \|M_i\xi\|^2=\sum_{i=1}^n \langle M_i^*M_i\xi,\xi\rangle.$$  Since this equality must hold true for all unit vectors $\xi\in \cal H$, it follows that $\sum_{i=1}^n M_i^*M_i=I_{\cal H}$.  Consequently, any sequence $M_1,\ldots,M_n\in \cal B(\cal H)$ satisfying this latter property constitutes a measurement of the system.

If one is only interested in the probabilities of the outcomes rather than the outcomes themselves (as we will be when we return to our discussion of nonlocal games), then it simplifies matters by replacing a measurement as above by a sequence $P_1,\ldots,P_n$ consisting of positive operators which sum up to $I_{\cal H}$ and interpret the probability that the $i^{\text{th}}$ outcome is obtained when the system is in state $\xi$ to be given by $\langle P_i\xi,\xi\rangle$.  Such a collection of positive operators is called a \textit{positive operator-valued measure} or \emph{POVM} on $\cal H$ (the terminology comes from spectral theory).  If one specializes even further to the case that each $P_i$ is not only a positive operator but in fact a projection, then one speaks of \emph{projection-valued measures} (or \emph{PVMs}) on $\cal H$.  Note then that the projections are automatically pairwise orthogonal, so a PVM on $\cal H$ with $n$ outcomes corresponds to a decomposition of $\cal H$ into $n$ orthogonal subspaces.

Many introductions to quantum mechanics discuss the measurements of \emph{observables}.  An observable for the physical system is a self-adjoint operator $\cal O$ on $\cal H$.  Supposing for simplicity that $\cal H$ is finite-dimensional, the Spectral Theorem implies that we can find a PVM $P_1,\ldots,P_n$ on $\cal H$ such that the $P_i$'s correspond to the projections onto the various eigenspaces of $\cal H$ corresponding to $\cal O$.  The self-adjointness assumption on $\cal O$ further implies that the corresponding eigenvalues are real numbers, whence we can interpret them as corresponding to actual possible physical measurements.  Conversely, given any PVM $P_1,\ldots,P_n$ on $\cal H$ and real numbers $\lambda_1,\ldots,\lambda_n$, one has an observable $\cal O:=\sum_{i=1}^n \lambda_i P_i$.

A simple example of the content of the previous paragraph is given by the \emph{spin} of an electron.  The spin of an electron along any choice of axis comes in one of two flavors:  ``up'' or ``down.''  (By the way, this is what is ``quantum'' about quantum mechanics:  many attributes of a physical system come in a discrete set of possibilities.)  For the sake of completeness, let us say that we are measuring spin along the vertical axis.  The state of the electron is given by a unit vector $\psi$ in the Hilbert space $\bb C^2$.  We view the usual orthonormal basis $\{e_1,e_2\}$ for $\bb C^2$ as representing the two possible spin values:  so $e_1$ corresponds to ``up'' while $e_2$ corresponds to ``down.''  Now a general unit vector $\psi$ in $\bb C^2$ can be written in the form $\psi=\alpha_1e_1+\alpha_2e_2$ for unique complex numbers $\alpha_1,\alpha_2\in \bb C$ for which $|\alpha_1|^2+|\alpha_2|^2=1$.  What is strange and new about quantum mechanics is that a given electron can be in a state that is neither up nor down.  More specifically, when neither $\alpha_1$ nor $\alpha_2$ are $0$, the electron is considered in a \emph{superposition} of the two states and will only reveal one of these two states upon a measurement of the spin, that is, using the PVM $P_1,P_2$ on $\bb C^2$ consisting of the projections onto the coordinate axes.  The state of the electron merely gives us probabilistic information as to which of the two outcomes will happen upon such a measurement.  Moreover, once the measurement has been made, the new state of the electron instantaneously and discontinuously jumps to the unit vector $e_1$ or $e_2$ corresponding to the outcome of the measurement just made.  This reflects the fact that if another measurement is made directly following the first measurement, the same outcome will occur.  It is important to make the distinction between superposition and definite measurement outcome with probabilities measuring ignorance of the actual value.

The above description of quantum mechanics we have given is the standard or \emph{Copenhagen} interpretation and it is a mighty big pill to swallow upon a first reading.  (Technically speaking, this is really the \emph{von Neumann-Dirac} formulation of the theory; however, it has become common parlance to refer to this interpretation as the Copenhagen interpretation, even though Niels Bohr himself explicitly disagreed with this formulation.)  Perhaps the biggest point of contention is the question ``What constitutes a measurement?'' together with the follow-up question ``Why did the state of the electron \emph{collapse} to one of the two basis states?''  This is the so-called \emph{measurement problem} and is a very popular topic of debate amongst philosophers and theoretical physicists.  It has lead to a plethora of alternate interpretations of quantum mechanics (often yielding mathematically equivalent predictions); a good introduction to these foundational issues is Barrett's recent book \cite{barrett}.

To keep the strangeness coming, suppose that we want to measure spin in the horizontal direction instead of the vertical direction.  It turns out that the appropriate basis to consider now is now $\{v_1,v_2\}$, where $v_1=\frac{1}{\sqrt{2}}e_1+\frac{1}{\sqrt{2}}e_2$ and $v_2=\frac{1}{\sqrt{2}}e_1-\frac{1}{\sqrt{2}}e_2$.  In other words, the PVM $Q_1,Q_2$ consisting of the orthogonal projections onto the lines spanned by $v_1$ and $v_2$ respectively constitutes a measurement of the spin of the electron in the horizontal direction.  Suppose that an electron has a definite spin, say up, in the vertical direction, whence its state is $e_1$.  In the eigenbasis for the observable of spin in the horizontal direction, the state becomes $e_1=\frac{1}{\sqrt{2}}v_1+\frac{1}{\sqrt{2}}v_2$.  Consequently, a measurement of an electron with an up spin in the vertical direction will yield a spin of either left or right in the horizontal direction with equal probability.  Even more strangely, suppose that the electron that had a definite vertical spin that was up was then measured in the horizontal direction and the outcome was spin left, that is, the measurement led to an outcome state of $v_1$.  Suppose further that a subsequent measurement of the electron in the vertical direction was performed.  Since $v_1=\frac{1}{\sqrt{2}}e_1+\frac{1}{\sqrt{2}}e_2$, we see that the outcome of the measurement now yields up or down with equal probability.  Thus, the measurement in the horizontal direction destroyed the definite spin the electron had in the vertical direction!



\subsection{The spookiness of entanglement}\label{sec5:spookiness}

The postulates of quantum mechanics tell us that if $\cal H_A$ and $\cal H_B$ are the Hilbert spaces representing two physical systems, then the appropriate Hilbert space for studying the composite system is the tensor product space $\cal H_A\otimes \cal H_B$.  The fact that elements of the tensor product need not be merely simple tensors leads to the fascinating concept of \emph{entanglement}, which, in some sense, is the essence of this entire story!

In order to get an idea of the utility of entanglement as a resource in, say, quantum information theory, we present the example of \emph{superdense coding}.  We set $\psi_{EPR}:=\frac{1}{\sqrt{2}}(e_1\otimes e_1+e_2\otimes e_2)\in \bb C^2\otimes \bb C^2\cong \bb C^4$.  This quantum state is known as the \textit{EPR state}, named after Einstein, Podolsky, and Rosen.  We will have more to say about this state and why EPR were considering it shortly.  Let us imagine that Alice and Bob each possess an electron and the joint state of the vertical spins of the two electrons is $\psi_{EPR}$, that is, the electrons are in an equal superposition of both spins being up or both spins being down.  Furthermore, imagine that Alice and Bob are really (really) far away from each other.  We show how Alice and Bob can utilize the fact that their electrons are in this entangled state in order for Alice to send two classical bits of information to Bob by just sending one \emph{qubit} of information, that is, by Alice sending Bob her electron (after she has done some work on it first).

Depending on what two bits of information Alice wishes to send to Bob, she performs one of the following actions to her electron:
\begin{itemize}
    \item $\psi_{11}:=(I\otimes I)\psi_{EPR}=\frac{1}{\sqrt 2}(e_1\otimes e_1+e_2\otimes e_2)$
    \item $\psi_{12}:=(X\otimes I)\psi_{EPR}=\frac{1}{\sqrt 2}(e_2\otimes e_1+e_1\otimes e_2)$
    \item $\psi_{21}:=(Z\otimes I)\psi_{EPR}=\frac{1}{\sqrt 2}(e_1\otimes e_1-e_2\otimes e_2)$
    \item $\psi_{22}:=(ZX\otimes I)\psi_{EPR}=\frac{1}{\sqrt 2}(e_1\otimes e_2-e_2\otimes e_1)$
\end{itemize}

Here, $X=\left(\begin{matrix}0 & 1 \\ 1 & 0\end{matrix}\right)$, the so-called \emph{bit-flip operator}, and $Z=\left(\begin{matrix} 1 & 0 \\ 0 & -1\end{matrix}\right)$, the so-called \emph{phase-flip operator}.

One can check that the four vectors $\psi_{11},\psi_{12},\psi_{21},\psi_{22}$ form an orthonormal basis for $\bb C^2\otimes \bb C^2\cong \bb C^4$ known as the \emph{Bell basis}.  Consequently, any observable $\cal O$ on $\bb C^4$ with distinct eigenvalues and with the Bell basis vectors as eigenvectors can be used to distinguish these vectors, that is, when the state of the system is $\psi_{ij}$, a measurement of $\cal O$ will yield $\psi_{ij}$ with probability one, whence Bob knows which of the four actions above Alice took and thus knows which pair of bits she wished to send to Bob.  (One can be explicit about the observable $\cal O$, namely $\cal O=(H\otimes I_{\bb C^2})C$, where $H:=\frac{1}{\sqrt{2}}(X+Z)$ is the so-called \emph{Hadamard operator} and $C$ is the so-called \emph{controlled not operator}.) 


Notice something peculiar about the EPR state:  if the state of two electrons is given by $\psi_{EPR}$, then they are in a superposition of either both electrons having spin up or both electrons having spin down (with equal probability).  However, if Alice performs a measurement of the spin of her electron and sees a result of spin up, she knows, with absolute certainty, that a subsequent measurement of the spin of Bob's electron will also be spin up.  Thus, while Bob's electron did not have a determinate spin before Alice's measurement, the result of Alice's measurement \emph{instantaneously} gave a determinate value to the spin of Bob's electron. 

Einstein was worried by this phenomenon, which he called ``spooky action at a distance.''  Together with Podolsky and Rosen \cite{EPR}, they used the EPR state to present an argument for the \emph{incompleteness} of quantum mechanics.  The gist of the argument is as follows:  suppose that Alice and Bob share a pair of electrons in the EPR state $\psi_{EPR}$ and that Alice and Bob are again really (really) far apart.  Suppose that Alice measures her electron and sees the result spin up.  Then Alice knows with 100\% certainty that if Bob were to measure his electron, then it must also have a determinately up spin.  Ditto for a measurement result of spin down.  Since Alice can predict with certainty the outcome of Bob's measurement and since her measurement could not possibly have altered the spin of Bob's electron, Bob's spin must have a definite value, independent of whether or not Alice were to measure it.  This definite value must represent some element of physical reality and if quantum mechanics were to be complete, there must be some counterpart of this physical reality in the theory.  Since there is nothing in the description of the EPR state which specifies a determinate value for Bob's spin, quantum mechanics must be incomplete.  

It gets even worse, for if Alice were to decide to measure her spin along a different axis, say the horizontal axis, then once again the result of her measurement would allow her to definitively conclude the value of Bob's electron's spin in the horizontal axis.  In this case, \emph{both} the  vertical and horizontal spins would have definite, predetermined values, which is a contradiction to the fact that knowing, say, the vertical spin of an electron forces us to be maximally uncertain about the horizontal spin of the electron.  So in some sense, the EPR argument even posits that quantum mechanics is inconsistent!

The underlying philosophy that EPR have in their argument is usually dubbed \emph{local realism}:  the term ``local'' refers to the assumption that Alice's measurement could not have affected Bob's electron since they are so far away and communication can not travel faster than the speed of light, while the term ``realism'' refers to the statement that the fact that one can determine the spin of Bob's electron with certainty implies that there must be some real, predetermined value to the spin.  EPR believed that there should be some ``hidden variable'' explaining this predetermined spin, allowing them to preserve their classical, locally real intuitions.

John Bell \cite{Bell} set up a thought experiment to determine whether there could indeed be a formulation of quantum mechanics that was complete and adhered to the local realist philosophy.  He showed that this is in fact impossible by showing that a small set of local realist assumptions leads to an inequality on the expected outcome of a certain experiment and that a particular quantum measurement could violate that inequality.  Moreover, it is actually experimentally testable whether or not this inequality holds in nature.  Spoiler alert:  the inequality is violated by nature, whence quantum mechanics comes out victorious!  Thus, while seemingly strange, quantum mechanics lies in contradistinction to the local realist assumptions.

Besides being an intellectually fascinating story, there turns out to be a direct link between these Bell inequalities and the phenomena of having quantum strategies for nonlocal games that exceed all possible classical values, which we now explain.  (The idea of treating the violation of Bell-type inequalities as quantum strategies for nonlocal games that exceed the classical value of the game seems to have first been seriously studied by Cleve, Hoyer, Toner, and Watrous \cite{CHTW}).

We have already discussed deterministic strategies for nonlocal games.  One may imagine incorporating a probabilistic component to these strategies by considering a probability space $(\Omega,\mu)$ and determinstic strategies $A_\omega:[k]\to [n]$ and $B_\omega:[k]\to [n]$, one for each $\omega\in \Omega$.  Consequently, the players can randomly (according to $(\Omega,\mu)$) select an $\omega$ and then play deterministically according to $A_\omega$ and $B_\omega$.  In terms of the EPR experiment, one may think of $\omega$ as the ``hidden variable'' for which we do not have perfect knowledge but that if we were to know it, then things would behave deterministically.  The probability space $(\Omega,\mu)$ represents our epistemic (lack of) knowledge of the hidden variable.  Consequently, we now have probabilistic strategies $$p(a,b|x,y):=\mu(\{\omega \in \Omega \ : \ A_\omega(x)=a \text{ and }B_\omega(y)=b\}),$$ which are called \emph{local strategies}, the term ``local'' referring to the fact that each player's output still only depends on their local environment.  The set of such local strategies is denoted $C_{loc}(k,n)$.  It is straightforward to see that $C_{loc}(k,n)$ is a compact, convex subset of $[0,1]^{k^2n^2}$ whose extreme points are the elements in $C_{det}(k,n)$.  Moreover, it is clear that every element of $C_{loc}(k,n)$ is a convex combination of elements of $C_{det}(k,n)$, whence $\val(\frak G)=\sup_{p\in C_{loc}(k,n)}\val (\frak G,p)$ for any nonlocal game $\frak G$ with $k$ questions and $n$ answers.

On the other hand, we can consider quantum strategies for nonlocal games as follows.  We let $C_{q}(k,n)$ consist of those strategies $p$ for which $$p(a,b|x,y)=\langle (A^x_a\otimes B^y_b)\psi,\psi\rangle,$$ where, for each $x,y\in [k]$, $A^x=(A^x_a)_{a\in [n]}$ and $B^y=(B^y_b)_{b\in [n]}$ are POVMS with $n$ outcomes on \emph{finite-dimensional} Hilbert spaces $\cal H_A$ and $\cal H_B$ respectively.  We call such a strategy $p$ a \emph{quantum strategy}.  These stratgies correspond to Alice and Bob sharing a (possibly entangled) state $\psi$ of their composite system $\cal H_A\otimes \cal H_B$ and performing measurements $A^x$ and $B^y$ on their portion of the state upon receiving questions $x$ and $y$ respectively.  Using a technique known as \emph{Naimark dilation} (a special case of the Stinespring Dilation theorem from above), one can replace POVMs with the more convenient to use PVMs without altering the definition of $C_q(k,n)$.  It is a straightforward argument to show that $C_q(k,n)$ is a convex subset of $[0,1]^{k^2n^2}$.

We have that $C_{loc}(k,n)\subseteq C_{q}(k,n)$.  Indeed, since every element of $C_{loc}(k,n)$ is a convex combination of deterministic strategies and $C_q(k,n)$ is convex, it suffices to show that every determinstic strategy $p$ is contained in $C_q(k,n)$.  However, this is quite easy:  if $A:[k]\to [n]$ is the function determining Alice's strategy, let $A^x$ be the POVM on $\bb C$ for which $A^x_{A(x)}=I$ and $A^x_{a}=0$ for all $a\not=A(x)$.  Bob's POVM $B^b_y$ is defined in the analogous way.  It follows that,for any state $\xi\in \bb C\otimes \bb C$, we have that $p(a,b|x,y)=\langle (A^x_a\otimes B^y_b)\xi,\xi\rangle$. 

Given a non-local game $\frak G$, we define its \emph{entangled value} to be 

$$\val^*(\frak G):=\sup_{p\in C_{q}(k,n)}\val(\frak G,p).$$  By the previous paragraph, we have that $\val(\frak G)\leq \val^*(\frak G)$ for any nonlocal game $\frak G$.  The idea behind Bell's theorem, recast in the setting of nonlocal games, is that there are nonlocal games $\frak G$ for which $\val(\frak G)<\val^*(\frak G)$.

For example, we consider the following game, known as the \emph{CHSH game}.  (The acronym CHSH stands for Clauser, Horne, Shimony, and Holt,the researchers responsible for the CHSH inequality, a Bell-type inequality that was one of the first to be experimentally testable.)  The CHSH game $\frak G_{CHSH}$ is a game with $k=n=2$.  The question distribution is the uniform distribution on $[2]\times [2]$ and with decision predicate $D$ given by the following conditions:
\begin{itemize}
    \item If $x=1$ or $y=1$, then Alice and Bob win if and only if their answers agree.
    \item If $x=y=2$, then Alice and Bob win if and only if their answers disagree.
\end{itemize}

By inspecting all determinstic strategies, one finds that $\val(\frak G_{CHSH})=\frac{3}{4}$.  However, the entangled value of the game satisfies $\val^*(\frak G)=\cos^2(\frac{\pi}{8})\approx 0.85>\val(\frak G_{CHSH})$.  The interested reader can find the details for this calculation in \cite[Section 3.1]{CHTW}.  We merely point out that a quantum strategy for achieving $\val^*(\frak G)$ uses the EPR state $\psi_{EPR}$.






\subsection{MIP*}\label{sec5:mip*}

Based on the nonlocal game definition of the complexity class $\MIP$ and our recent discussion of quantum strategies for nonlocal games, it should be clear how to define the complexity class $\MIP^*$:  the language $\bf L$ belongs to $\MIP^*$ if there is an efficient mapping (in the precise sense from Subsection \ref{sec4:nonlocal}) $z\mapsto \frak G_z$ from strings to non-local games such that:
\begin{itemize}
    \item If $z\in \bf L$, then $\val^*(\frak G_z)\geq \frac{2}{3}$.
    \item If $z\notin \bf L$, then $\val^*(\frak G_z)\leq \frac{1}{3}$.
\end{itemize}

We remark that the definition of $\mip^*$ first appeared in the aforementioned paper \cite{CHTW}.

To be fair, we are really defining the complexity class $\MIP^*(2,1)$, which only has two provers and one round of interactions.  There are ways to define similar classes that allow more verifiers and rounds, but the eventual result $\MIP^*=\operatorname{RE}$ will show they yield the same class anyways, so we will not bother.

So how do the classes $\MIP$ and $\MIP^*$ relate?  The lesson from the previous section was that provers that share entanglement can win some nonlocal games more often than they ``rightfully should.''  In other words, it seems that it might be the case that for a language $\bf L$ that belongs to $\MIP$ and for a string $z$ that does \emph{not} belong to $\bf L$, the provers might have a strategy for the corresponding game $\frak G_z$ whose value exceeds $\frac{1}{3}$.

Nevertheless (and perhaps somewhat surprisingly), Ito and Vidick \cite{IV} showed that $\MIP\subseteq \MIP^*$.  The rough idea behind this inclusion is that it suffices to show that $\NEXP\subseteq \MIP^*$ and the games involved in the proof that $\MIP=\NEXP$ are such that their classical and quantum values are approximately the same.

Later, Natarajan and Wright \cite{NW} showed that $\NEEXP\subseteq \MIP^*$.  Recalling that $\MIP=\NEXP\subsetneq \NEEXP$, this shows that $\MIP\subsetneq \MIP^*$, whence adding entanglement does indeed strictly increase the computational power of the verifier.

So exactly how much extra power does entanglement give us?  Besides the result mentioned in the last paragraph, there was only an a priori seemingly silly upper bound on $\mip^*$, namely $\MIP^*\subseteq \operatorname{RE}$, where $\operatorname{RE}$ (which is short for \emph{recursively enumerable}) is the complexity class which consists of those languages $\bf L$ for which there is a Turing machine $\bf M$ (with absolutely no efficiency requirements whatsoever) whose domain is $\bf L$, that is, $\bf L$ consists of the set of inputs for which $\bf M$ halts.  An alternative formulation of $\operatorname{RE}$ is helpful:  $\bf L$ belongs to $\operatorname{RE}$ if there is a total computable function whose range is $\bf L$ (this is why modern computability theorists refer to this as being \emph{computably enumerable} or \emph{CE}).  To see the inclusion $\mip^*\subseteq \operatorname{RE}$, note first that, given any dimension $d$, one can effectively enumerate a countable set of quantum strategies of dimension $d$ that is dense in the set of such strategies and for which one can effectively compute $\val(\frak G,p)$ for any such quantum strategy $p$.  By letting $d$ tend to $\infty$, if one ever finds such a strategy $p$ for which $\val(\frak G_z,p)>\frac{1}{2}$, one knows that $z\in \bf L$ (and one is guaranteed that this will happen for some such $p$ if $z\in \bf L$).  Note that if $z\notin \bf L$, this procedure will never convince us that $z\notin \bf L$ because maybe we did not wait long enough and a higher dimensional strategy would indeed have convinced us if we were just a bit more patient.

The amazing fact proven in \cite{MIP*} is that this upper bound is actually tight!  That is, $\MIP^*=\operatorname{RE}$ holds!  More specifically, the authors prove that there is an effective mapping $\mathbf{M}\mapsto \frak G_{\mathbf{M}}$ from (codes for) Turing machines to nonlocal games such that:
\begin{itemize}
    \item If $\mathbf{M}$ halts on the empty tape, then $\val^*(\frak G_{\mathbf{M}})=1$.
    \item If $\mathbf{M}$ does not halt on the empty tape, then $\val^*(\frak G_{\mathbf{M}})\leq \frac{1}{2}$.
\end{itemize}

The language consisting of codes for Turing machines that halt on the empty tape is known as the \emph{halting problem} $\mathbf{HALT}$.  Since the halting problem is complete for the class $\operatorname{RE}$, the inclusion $\operatorname{RE}\subseteq \MIP^*$ holds.

Irregardless of your interest in CEP, the equality $\mip^*=\operatorname{RE}$ is an amazing fact.  The halting problem is an undecidable problem (this follows from a simple diagonalization argument together with the fact that there is a so-called \emph{universal Turing machine}).  Nevertheless, if two cooperating but non-communicating provers share some quantum entanglement, they can reliably convince a verifier whether or not a given Turing machine halts!  This is a landmark intellectual achievement.

The proof of $\MIP^*=\operatorname{RE}$ is very complicated and we will not discuss it here.  The introduction to \cite{MIP*} does a great job outlining the essence of the proof.

The story of $\mip^*$ is about allowing quantum resources but keeping the computational model classical.  It is also interesting to ask what happens if we also replace the computational model we are using (i.e. the Turing machine) with a quantum computational model (e.g. quantum circuits).  It turns out that there is nothing to be gained here: by prefixing the corresponding classical complexity class with a ``Q'' to denote its counterpart defined using a quantum computational model, we have $\operatorname{QIP}=\IP$, $\operatorname{QMIP}=\MIP$, and $\operatorname{QMIP}^*=\MIP^*=\operatorname{RE}$; see \cite{quantumproofs} for the details.

\section{From MIP*=RE to the failure of CEP:  the traditional route}\label{sec6}

The derivation of the negative solution to the CEP from $\mip^*=\operatorname{RE}$ now proceeds in two steps:  we first show how $\mip^*=\operatorname{RE}$ leads to a negative solution to \emph{Tsirelson's problem} in quantum information theory; we show this in the first subsection.  In the second subsection, we then show how a negative solution to Tsirelson's problem naturally leads to a negative solution to Kirchberg's QWEP problem.  As we already observed in Subsection \ref{sec3:from}, this leads to a negative solution to the CEP.

\subsection{A negative solution to Tsirelson's problem}\label{sec6:tsirelson}

In order to explain Tsirelson's problem, we need to introduce some more collections of strategies.  First, we define $C_{qs}(k,n)$ exactly as in the definition of $C_q(k,n)$ except that we remove the finite-dimensionality assumptions on Alice's and Bob's state spaces $\cal H_A$ and $\cal H_B$; a strategy in this larger class is called a \emph{quantum spatial strategy}.  Quantum spatial strategies still correspond to the idea that Alice and Bob each have their own physical system and the state of their composite system is given by the tensor product.  It can be checked that there is no loss of generality in restricting attention to separable Hilbert spaces in the definition of $C_{qs}(k,n)$.  Moreover, by considering projections onto larger and larger finite-dimensional subspaces, we see that $C_{qs}(k,n)\subseteq \overline{C_q(k,n)}$, the closure of $C_q(k,n)$ in the usual topology it inherits from being a subset of $[0,1]^{k^2n^2}$.  Like $C_q(k,n)$, one can check that $C_{qs}(k,n)$ is convex. 

There is another model that is natural to consider which arises in quantum field theory.  In quantum field theory, one usually considers a large quantum system (maybe the system describing the whole universe!) and then it may be difficult to separate Alice and Bob's systems as isolated subsystems of the larger system.  The state of the large system is now given by some unit vector $\xi$ in a single Hilbert space $\cal H$ and Alice's and Bob's measurements are now given by families of POVMs $(A^x)_{x\in [k]}$ and $(B^y)_{y\in[k]}$ acting on this single Hilbert space $\cal H$.  Since we are still assuming that Alice and Bob are far away and so they cannot interact with each other, it is natural to assume that either of them can measure first without affecting the value of the other's measurements (or even that they can perform their measurements simultaneously).  According to von Neumann, the mathematical way of modeling this situation is to assume that Alice's and Bob's measurements commute with one another, that is, $A^x_aB^y_b=B^y_bA^x_a$ for all $x,y\in [k]$ and all $a,b\in [n]$.  The corresponding strategy is given by $p(a,b|x,y)=\langle A^x_aB^y_b\xi,\xi\rangle$ and is called a \emph{quantum commuting strategy}.  (Commutativity ensures that this a priori complex value lies in $[0,1]$.)  The set of quantum commuting strategies is denoted $C_{qc}(k,n)$.  Note that there is no requirement that $\cal H$ be finite-dimensional (although one can take it to be separable) and, in fact, requiring $\cal H$ to be finite-dimensional yields another description of the set $C_q(k,n)$ (see \cite{DLTW}).  Later, we will see that $C_{qc}(k,n)$ is a closed convex subset of $[0,1]^{k^2n^2}$ and that, like $C_q(k,n)$, one can use PVMs instead of POVMs without changing the definition.  

It is clear that $C_{qs}(k,n)\subseteq C_{qc}(k,n)$.  In \cite{tsirelson}, Boris Tsirelson claimed (without proof) that equality holds for all $(k,n)$.  After he was questioned about this, he realized that he could not prove this claim.  In fact, upon further reflection, he could not even establish whether or not $C_{qs}(k,n)$ was closed nor whether or not $C_{qa}(k,n):=\overline{C_{qs}(k,n)}=\overline{C_{q}(k,n)}$ coincided with $C_{qc}(k,n)$ (see his note \cite{tsirelson2}).  The question of whether or not $C_{qa}(k,n)=C_{qc}(k,n)$ for all $(k,n)$ is known as \emph{Tsirelson's problem}.  Incidentally, in \cite{slofstra} Slofstra showed that Tsirelson's original claim was false, that is, there is a pair $(k,n)$ such that $C_{qs}(k,n)\not=C_{qc}(k,n)$ and even strengthened this result to show that $C_{qs}(k,n)$ need not be closed, that is, there is $(k,n)$ for which $C_{qs}(k,n)\subsetneq C_{qa}(k,n)$.  

Fix a nonlocal game $\frak G$ with $k$ questions and $n$ answers.  It is clear that $$\sup_{p\in C_{qa}(k,n)}\val (\frak G,p)=\sup_{p\in C_{qs}(k,n)}\val (\frak G,p)=\val^*(\frak G).$$  However, we can also use elements of $C_{qc}(k,n)$ to define values of games, namely we define the \emph{commuting value} of $\frak G$ to be $\val^{co}(\frak G):=\sup_{p\in C_{qc}(k,n)}\val(\frak G,p)$.  It is clear that $\val^*(\frak G)\leq \val^{co}(\frak G)$ and that equality holds for all nonlocal games if Tsirelson's problem has an affirmative answer.  In fact, it can be shown that an affirmative answer to Tsirelson's problem is equivalent to the statement $\val^*(\frak G)=\val^{co}(\frak G)$ for all nonlocal games $\frak G$.

Recall that in our discussion of the inclusion $\mip^*\subseteq \operatorname{RE}$, we discussed how $\val^*(\frak G)$ can be effectively approximated from below.  On the other hand, it turns out that $\val^{co}(\frak G)$ can be effectively approximated from above.  This result follows from two facts:
\begin{itemize}
    \item There is a finitely presented group $G_{\frak G}$ (which in fact only depends on the number of questions and answers in $\frak G$) and an element $\eta_{\frak G}\in C^*(G_{\frak G})$ such that $\val^{co}(\frak G)=\|\eta_{\frak G}\|$ (see Corollary \ref{mainfritzcor} below), and
    \item For any finitely presented group $G$, one can always find effective upper bounds on the operator norm of $C^*(G)$ (a result due to Fritz, Netzer and Thom \cite[Corollary 2.2]{FNT}).
\end{itemize}

In Subsection \ref{sec7:negative} below, we offer a simple model-theoretic proof of the fact that $\val^{co}(\frak G)$ can be approximated from above, although, to be fair, we really establish a slightly different version of this fact sufficient to derive the failure of Tsirelson's problem from $\MIP^*=\operatorname{RE}$.  In any event, if $\val^*(\frak G)=\val^{co}(\frak G)$, that is, if Tsirelson's problem has an affirmative answer, then we can effectively approximate $\val^*(\frak G)=\val^{co}(\frak G)$ both from below and above, which would then imply that all languages in $\MIP^*$ are decidable, contradicting $\MIP^*=\operatorname{RE}$!

By the way, the argument in the preceding paragraph shows that $\MIP^{co}\subseteq \operatorname{coRE}$, where $\MIP^{co}$ is defined exactly like $\MIP^*$ but using the commuting value $\val^{co}$ of games instead of the entangled value $\val^*$ and $\operatorname{coRE}$ denotes the class of languages whose complement lies in $\operatorname{RE}$. It is currently unknown if this upper bound is sharp.

\subsection{A negative solution to Kirchberg's QWEP problem}\label{sec6:QWEP}

In this subsection, we show how a negative solution to Tsirelson's problem yields a negative solution to Kirchberg's QWEP problem.  We follow Fritz' presentation \cite{Fr} closely.

We begin by considering the abelian \cstar-algebra $\bb C^n$.  For each $a=1,\ldots,n$, we let $e_a$ denote the $a^{\text{th}}$ standard basis element of $\bb C^n$.  (We are using $a$ as the index since we are using the notation from nonlocal games.)  For any $k\geq 1$, we also consider the $k$-fold free product $\Asterisk_{x=1}^k\bb C^n$ and denote by $e^x_a$ the version of $e_a$ in the $x^{\text{th}}$-copy of $\bb C^n$.

\begin{prop}\label{Fritz1}

\

\begin{enumerate}
    \item There is a 1-1 correspondence between $n$-outcome POVMS $\{A_1,\ldots,A_n\}$ in $\cal B(\cal H)$ and ucp maps $\Phi:\bb C^n\to \cal B(\cal H)$ given by $\Phi(e_a):=A_a$.
    \item There is a 1-1 correspondence between $k$-tuples $\{A_1^x,\ldots,A^x_n\}_{x=1}^k$ of $n$-outcome POVMS in $\cal B(\cal H)$ and ucp maps $\Phi:\Asterisk_{x=1}^k\bb C^n\to \cal B(\cal H)$ given by $\Phi(e^x_a):=A^x_a$.
\end{enumerate}
\end{prop}

\begin{proof}
The proof of (1) is easy to check, using that a positive map with commutative domain is automatically completely positive.  Part (2) follows from (1) and a theorem of Florin Boca \cite{boca}, which implies that the individual ucp maps $\Phi^x:\bb C^n\to \cal B(\cal H)$ given by $\Phi^x(e^x_a):=A^x_a$ can be jointly extended to a single ucp map $\Phi:\Asterisk_{x=1}^k \bb C^n\to \cal B(\cal H)$.
\end{proof}

We now bring group \cstar-algebras into the picture, getting us closer to the QWEP problem.  We first note that $\bb C^n\cong C^*(\bb Z_n)$, where $\bb Z_n$ denotes the additive group of integers modulo $n$.  Indeed, let $u$ be a generator of $\bb Z_n$ and consider the element $z:=\sum_{a=1}^n \exp(\frac{2\pi ia}{n})e_a\in \bb C^n$.  It is readily verified that $z$ is an element of $\cal U(\bb C^n)$ of order $n$, whence the assignment $u\mapsto z$ yields a unitary representation $\bb Z_n\to \cal U(\bb C^n)$, extending to a $*$-homomorphism $C^*(\bb Z_n)\to \bb C^n$ that can be checked to be an isomorphism.  (This identification usually goes under the name \emph{discrete Fourier transform}.)

Let $\bb F(k,n):=\Asterisk_{x=1}^k \bb Z_n$ denote the group freely generated by $k$ elements of order $n$.  We then have
$$C^*(\bb F(k,n))=C^*(\Asterisk_{x=1}^k \bb Z_n)\cong\Asterisk_{x=1}^k C^*(\bb Z_n)\cong\Asterisk_{x=1}^k \bb C^n.$$  We abuse notation slightly and let $e^x_a$ denote the element of $C^*(\bb F(k,n))$ corresponding to $e^x_a\in \Asterisk_{x=1}^k \bb C^n$.  (Another viewpoint is that $(e^x_a)_{a=1}^n$ denote the spectral projections corresponding to the $x^{\text{th}}$-unitary element of $C^*(\bb F(k,n))$.)

Here is the main result connecting the QWEP problem and Tsirelson's problem:
\begin{thm}\label{strategychar}
Fix $k,n\geq 2$ and a strategy $p\in [0,1]^{k^2n^2}$.  We then have:
\begin{enumerate}
    \item $p\in C_{qa}(k,n)$ if and only if there is a state $\phi$ on $C^*(\bb F(n,k))\otimes_{\min} C^*(\bb F(n,k))$ for which $p(a,b|x,y)=\phi(e^x_a\otimes e^y_b)$.
    \item $p\in C_{qc}(k,n)$ if and only if there is a state $\phi$ on $C^*(\bb F(n,k))\otimes_{\max} C^*(\bb F(n,k))$ for which $p(a,b|x,y)=\phi(e^x_a\otimes e^y_b)$.
\end{enumerate}
\end{thm}

\begin{proof}
For the forward direction of (1), we may assume, without loss of generality, that $p\in C_{qs}(k,n)$, say $p(a,b|x,y)=\langle (A^x_a\otimes B^y_b)\xi,\xi\rangle$, where the POVMs $A^x$ and $B^y$ act on the Hilbert spaces $\cal H_A$ and $\cal H_B$ respectively.  By Proposition \ref{Fritz1} and the above identification $C^*(\bb F(k,n))\cong \Asterisk_{x=1}^k \bb C^n$, we have ucp maps $\Phi_A:C^*(\bb F(k,n))\to \cal B(\cal H_A)$ and $\Phi_B:C^*(\bb F(k,n))\to \cal B(\cal H_B)$ corresponding to these POVMs.  These two ucp maps combine to yield a ucp map $\Phi=\Phi_A\otimes \Phi_B:C^*(\bb F(k,n))\otimes_{\min}C^*(\bb F(k,n))\to \cal B(\cal H_A\otimes \cal H_B)$.  Consequently, we can define a state $\phi$ on $C^*(\bb F(k,n))\otimes_{\min}C^*(\bb F(k,n))$ by setting $\phi(w\otimes z):=\langle (\Phi(w\otimes z)\xi,\xi\rangle$.  It is clear that this state $\phi$ ``implements'' $p$ as in the statement of (1).

Conversely, suppose that $\phi$ is a state on $C^*(\bb F(n,k))\otimes_{\min}C^*(\bb F(n,k))$ for which $p(a,b|x,y)=\phi(e^x_a\otimes e^y_b)$.  Concretely represent $C^*(\bb F(k,n))\subseteq \cal B(\cal H)$ so that $C^*(\bb F(k,n))\otimes_{\min} C^*(\bb F(k,n))\subseteq \cal B(\cal H\otimes \cal H)$.  Extend $\phi$ to a state on $\cal B(\cal H\otimes \cal H)$, which we will continue to denote $\phi$.  Since convex combination of vector states are dense in the state space of $\cal B(\cal H\otimes \cal H)$, given $\epsilon>0$ there are vectors $\xi_1,\ldots,\xi_n\in \cal H\otimes \cal H$ for which $|\phi(e^x_a\otimes e^y_b)-\sum_{i=1}^n \langle (e^x_a\otimes e^y_b)\xi_i,\xi_i\rangle|<\epsilon$ for all $x,y\in [k]$ and $a,b\in [n]$.  This shows that $p$ can be approximated by convex combinations of elements of $C_{qs}(k,n)$.  Since $C_{qa}(k,n)$ is closed and convex, we have that $p\in C_{qa}(k,n)$.

The proof of the forward direction of (2) is identical to the proof of the forward direction of (1), using the fact that one can combine ucp maps with commuting ranges into a ucp map on the maximal tensor product.  For the converse, suppose that $\phi$ is a state on $C^*(\bb F(n,k))\otimes_{\max}C^*(\bb F(n,k))$ for which $p(a,b|x,y)=\phi(e^x_a\otimes e^y_b)$.  Let $\pi_\phi:C^*(\bb F(n,k))\otimes_{\max}C^*(\bb F(n,k))\to \cal B(\cal H)$ be the GNS representation corresponding to the state $\phi$ with cyclic vector $\xi$.  Set $A_a^x:=\pi_\phi(e^x_a\otimes I)$ and $B^y_b:=\pi_\phi(I\otimes e^y_b)$.  It is clear that $A^x_a$ and $B^y_b$ commute for all $x,y,a,b$ and that $p(a,b|x,y)=\langle A^x_aB^y_b \xi,\xi\rangle$, whence $p\in C_{qc}(k,n)$. 
\end{proof}

We note that the proof above fulfills a few promises made earlier, namely that elements of $C_{qc}(k,n)$ can always be taken to arise from PVMs (instead of just POVMs) and that $C_{qc}(k,n)$ is closed and convex (being the continuous image of the compact convex set of states on $C^*(\bb F(k,n))\otimes_{\max}C^*(\bb F(k,n))$.

Given a nonlocal game $\frak G=(\pi,D)$ with $k$ questions and $n$ answers, set $$\eta_\frak G:=\sum_{x,y\in [k]}\pi(x,y)\sum_{a,b\in [n]}D(x,y,a,b)(e^x_a\otimes e^y_b)\in C^*(\bb F(k,n))\odot C^*(\bb F(k,n)).$$

\begin{cor}\label{mainfritzcor}
For any nonlocal game $\frak G$, we have $\val^*(\frak G)=\|\eta_{\frak G}\|_{\min}$ and $\val^{co}(\frak G)=\|\eta_{\frak G}\|_{\max}$.
\end{cor}

We remind the reader that the previous corollary is responsible for the negative solution to Tsirelson's problem from $\mip^*=\operatorname{RE}$.  Indeed, $\|\eta_{\frak G}\|_{\max}$ corresponds to the operator norm of $\eta_{\frak G}$ when viewed as an element of $C^*(\bb F(k,n)\times \bb F(k,n))$; by \cite[Corollary 2.2]{FNT}, one can effectively compute upper bounds of the operator norm of elements of $C^*(\bb F(k,n)\times \bb F(k,n))$, whence one can effectively compute upper bounds for $\val^{co}(\frak G)$.

\begin{cor}
For any $k,n\geq 2$, if $(C^*(\bb F(k,n)),C^*(\bb F(k,n))$ is a nuclear pair, then Tsirelson's problem has a positive solution for scenarios of dimension $(k,n)$, that is, $\val^*(\frak G)=\val^{co}(\frak G)$ for all nonlocal games $\frak G$ with $k$ questions and $n$ answers.
\end{cor}

The quotient map $\bb Z\to \bb Z_n$ yields a quotient map $\bb F_k\to \bb F(k,n)$, leading to a surjective $*$-homomorphism $C^*(\bb F_k)\to C^*(\bb F(k,n))$.  One can show that this map has a ucp lift $C^*(\bb F(k,n))\to C^*(\bb F_k)$ (see, for example, \cite[Lemma D.3]{Fr}), whence $(C^*(\bb F(k,n)),C^*(\bb F(k,n))$ is a nuclear pair if $(C^*(\bb F_k),C^*(\bb F_k))$ is a nuclear pair.  (It can be shown that $\bb F(k,n)$ contains a copy of $\bb F_2$ if $(k,n)\not=(2,2)$, whence the converse to the previous sentence also holds in this case.)  Consequently, we have:

\begin{cor}
If Kirchberg's QWEP problem has a positive answer, then Tsirelson's problem has a positive answer.
\end{cor}

In the last subsection, we saw that $\mip^*=\operatorname{RE}$ implies that Tsirelson's problem has a negative solution; combined with the previous corollary, we now have that the QWEP problem has a negative solution, and thus, coupled with the discussion in Subsection \ref{sec3:from} above, we finally have the desired negative solution to the CEP!

\section{From MIP*=RE to the failure of CEP:  a model-theoretic shortcut}\label{sec7}

In this section, we show how ideas from logic yield a (in this author's opinion) more elementary derivation of a negative solution of CEP from $\mip^*=\operatorname{RE}$.  Much of the material presented in this section represents joint work of the author and Bradd Hart \cite{GH} and \cite{GH2}.

\subsection{A continuous logic for studying tracial von Neumann algebras}\label{sec7:continuous}

The negative solution to CEP from $\mip^*=\operatorname{RE}$ presented in this section uses techniques from logic.  Consequently, we need to describe an appropriate first-order language in a certain \emph{continuous logic} for studying tracial von Neumann algebras.  (We apologize for the double use of the word ``language'' in this paper.  The complexity-theoretic languages have been denoted using bold letters $\bf L$; we will use Roman letters $L$ for languages in the sense of logic.)

For a von Neumann algebra $\cal M$, we let $\cal M_1$ denote the operator norm unit ball.  Recall that by a ${}^*$-polynomial $p(x_1,\ldots,x_n)$ in the indeterminates $x_1,\ldots,x_n$ we mean an expression built from the indeterminates using the ${}^*$-algebra operations.  Let $\mathcal F$ denote the set of all ${}^*$-polynomials $p(x_1,\ldots,x_n)$ ($n\geq 0$) such that, for \emph{any} von Neumann algebra $\cal M$, we have $p(\cal M_1^n)\subseteq \cal M_1$.  For example, the following functions belong to $\mathcal F$:
\begin{itemize}
\item the ``constant symbols'' $0$ and $1$ (thought of as $0$-ary functions);
\item $x\mapsto x^*$;
\item $x\mapsto \lambda x$ ($|\lambda|\leq 1$)
\item $(x,y)\mapsto xy$
\item $(x,y)\mapsto \frac{x+y}{2}$.
\end{itemize}

We then work in the formal \emph{language} $L_{vNa}:=\mathcal F\cup \{\tr_\R,\tr_\Im,d\}$, where $\tr_\Re$ (resp. $\tr_\Im$) denote the real (resp. imaginary) parts of the trace and $d$ denotes the metric on the operator norm unit ball given by $d(x,y):=\|x-y\|_{\tau}$.  We can then formulate certain properties of tracial von Neumann algebras using the language $L_{vNa}$ as follows.

\emph{Basic $L_{vNa}$-formulae} will be formulae of the form $\tr_\Re(p(\vec x))$ or $\tr_\Im(p(\vec x))$ for $p\in \mathcal F$.  \emph{Quantifier-free $L_{vNa}$-formulae} are formulae of the form $f(\varphi_1,\ldots,\varphi_m)$, where $f:\r^m\to \r$ is a continuous function and $\varphi_1,\ldots,\varphi_m$ are basic $L_{vNa}$-formulae.  Finally, an arbitrary $L_{vNa}$-formula is of the form 
$$Q^1_{x_{i_1}}\cdots Q^k_{x_{i_k}}\varphi(x_1,\ldots,x_n),$$ where each $i_j\in \{1,\ldots,n\}$, $\varphi(x_1,\ldots,x_n)$ is a quantifier-free $L_{vNa}$-formula, and each $Q^i$ is either $\sup$ or $\inf$; we think of these $Q_i$'s as \emph{quantifiers over the unit ball of the algebra}.  

For those keeping score at home, our setup here is a bit more specialized than the general treatment of continuous logic in \cite{mtfms} (or even the version \cite{mtoa2} presented for operator algebraists), but a dense set of the formulae in \cite{mtfms} are logically equivalent to formulae in the above form, so there is no loss of generality in our treatment here.

Also, in order to keep the set of formulae ``separable'' and ``computable'', when forming the set of quantifier-free formulae, we really should restrict ourselves to a computable dense subset of the set of all continuous functions $\r^m\to \r$ as $m$ ranges over $\n$. (See \cite[Section 2]{GH}.)

Suppose that $\varphi(\vec x)$ is a formula, $\cal M$ is a tracial von Neumann algebra, and $\vec a\in \cal M_1^n$, where $n$ is the length of the tuple $\vec x$.  We let $\varphi(\vec a)^{\cal M}$ denote the real number obtained by replacing the variables $\vec x$ with the tuple $\vec a$; we may think of $\varphi(\vec a)^{\cal M}$ as the truth value of $\varphi(\vec x)$ in $\cal M$ when $\vec x$ is replaced by $\vec a$.  For example, if $\varphi(x_1)$ is the formula $\sup_{x_2}d(x_1x_2,x_2x_1)$, then $\varphi(a)^{\cal M}=0$ if and only if $a$ is in the center of $\cal M$.

If $\varphi$ has no \emph{free variables} (that is, all variables occurring in $\varphi$ are bounded by some quantifier), then we say that $\varphi$ is a \emph{sentence} and we observe that $\varphi^{\cal M}$ is a real number.  Given a tracial von Neumann algebra, the \emph{theory of $\cal M$} is the function $\Th(\cal M)$ which maps the sentence $\varphi$ to the real number $\varphi^{\cal M}$.  Sometimes authors define $\Th(\cal M)$ to consist of the set of sentences $\varphi$ for which $\varphi^{\cal M}=0$; since $\Th(\cal M)$, as we have defined it, is determined by its zeroset, these two formulations are equivalent.

If $\varphi(\vec x)$ is a formula, then there is a bounded interval $[m_\varphi,M_\varphi]\subseteq \r$  called the range of $\varphi$ such that, for any tracial von Neumann algebra $\cal M$ and any $\vec a\in \cal M_1$, we have $\varphi(\vec a)^{\cal M}\in [m_\varphi,M_\varphi]$. 

At this point we need to mention an important if not seemingly pedantic point (to a nonlogician).  We have been focusing our attention on those structures in the language $L_{vNa}$ that actually correspond to (unit balls of) tracial von Neumann algebras.  This is a perfectly legitimate thing to do because the class of tracial von Neumann algebras form an \emph{elementary class}.  Perhaps a simpler example from classical logic will help illustrate the point.  Let $L_{grp}=\{\cdot,e\}$ consist of a single binary function symbol $\cdot$ and constant symbol $e$.  Of course, the intended $L_{grp}$-structures are the ones that interpret these symbols as the multiplication and identity of a group.  However, there are perfectly reasonable, if not silly, $L_{grp}$-structures, such as, for example, one that interprets $\cdot$ as a constant function.  The key point is that we can write down a collection of axioms, that is, a set of $L_{grp}$-sentences $T_{grp}$, that single out the class of groups in the sense that an $L_{grp}$-structure $G$ is a group if and only if every sentence in $T_{grp}$ is true in $G$.  This is the definition of what it means for the class of groups to be an elementary class in the language $L_{grp}$.

A similar situation is true in our context, namely, there is a collection $T_{vNa}$ of $L_{vNa}$-sentences such that an $L_{vNa}$-structure is the unit ball of a von Neumann algebra if and only if each sentence in $T_{vNa}$ evaluates to $0$ in the structure.  In fact, one can add to these axioms a couple of extra sentences in order to obtain the theory $T_{II_1}$ whose models are all (unit balls of) II$_1$ factors.


\subsection{The model-theoretic reformulation of CEP}\label{sec7:model}

An $L_{vNa}$-sentence $\sigma$ of the form $$\sup_{x_1}\cdots \sup_{x_n} \varphi(x_1,\ldots,x_n)$$ is called \emph{universal} if $\varphi$ is quantifier-free and the range of $\varphi$ is non-negative and similarly \emph{existential} if all the quantifiers are $\inf$.  This terminology is justified if one thinks of the value 0 as ``true'' for then $\sigma^{\cal M}=0$ if and only if $\varphi(a_1,\ldots,a_n)=0$ for all $a_1,\ldots,a_n\in \cal M_1$. If we restrict the function $\Th(\cal M)$ to the set of all universal (resp. existential) sentences, the resulting function is defined to be the \emph{universal} (resp. \emph{existential}) theory of $\cal M$, denoted $\Th_\forall(\cal M)$ (resp. $\Th_\exists(\cal M)$).

It is fairly easy to see that $\Th_\forall(\cal M)=\Th_\forall(\cal M^\u)$ for any tracial von Neumann algebra $\cal M$ and any ultrafilter $\u$.  (The \emph{\L os theorem} \cite[Proposition 4.3]{mtoa2} shows that $\Th(\cal M)=\Th(\cal M^\u)$, but we will not need this more general fact.)  Consequently, if $\cal N$ embeds into $\cal M^\u$, then we have that $\Th_\forall(\cal N)\leq \Th_\forall(\cal M)$ (as a function).  It turns out that the converse is also true.  To see this, assume that $\Th_\forall(\cal N)\leq \Th_\forall(\cal M)$.  To show that $\cal N$ embeds into an ultrapower of $\cal M$, it suffices (by standard ultrapower arguments) to show, given any finitely many $a_1,\ldots,a_n\in \cal N_1$, any atomic formula $\varphi(x_1,\ldots,x_n)$, and any $\epsilon>0$, that there are $b_1,\ldots,b_n\in \cal M_1$ such that $|\varphi(\vec a)^{\cal N}-\varphi(\vec b)^{\cal M}|<\epsilon$.  Set $r:=\varphi(\vec a)^{\cal N}$ and $\sigma:=\inf_{\vec x}|r-\varphi(\vec x)|$.  It is clear that $\sigma^{\cal N}=0$.  The assumption that $\Th_\forall(\cal N)\leq \Th_\forall(\cal M)$ implies that $\Th_\exists(\cal M)\leq \Th_\exists(\cal N)$, whence $\sigma^{\cal M}=0$, which easily implies the existence of the desired tuple $\vec b\in \cal M_1$.  (None of this is particular to the case of tracial von Neumann algebras and holds for any pair of structures in the same language.)  


We can thus reformulate the CEP as follows:  for every tracial von Neumann algebra $\cal M$, we have that $\Th_\forall(\cal M)\leq \Th_\forall(\R)$.  Recalling that $\R$ embeds into every II$_1$ factor, we can further reformulate the CEP:  there is a unique universal theory of II$_1$ factors, namely $\Th_\forall(\R)$.

\subsection{The Completeness Theorem for (continuous) first-order logic}\label{sec7:completeness}

Before discussing the Completeness theorem for II$_1$ factors in the context of continuous logic, we consider the simpler example of groups in classical logic.

Consider the following theorem in group theory:  every group has a unique identity element.  This theorem can be written as an $L_{grp}$-sentence $\sigma_e$ defined by $\forall x(\forall y(x\cdot y=y\cdot x=x)\to x=e)$.  What exactly does it mean for $\sigma_e$ to be a theorem of group theory? 

Syntactically, what this means is that in some formal proof system for first-order logic, there is a formal proof of $\sigma_e$ from $T_{grp}$, denoted $T_{grp}\vdash \sigma_e$.  This formal proof is simply a finite list of sentences, each of which is either an element of $T_{grp}$ or can be obtained from earlier elements of the list using the rules of the proof system, with $\sigma_e$ being the last element of the list.

Semantically, we might say that in every model of $T_{grp}$, that is, in every group, the sentence $\sigma_e$ is true; we denote this relationship by $T_{grp}\models \sigma_e$.

For any reasonable proof system, it is fairly easy to prove that $\vdash$ implies $\models$; this is called the \emph{Soundness theorem for first order logic}.  A much less obvious result is that the converse also holds, namely that any time $T_{grp}\models \sigma$, then in fact $T_{grp}\vdash \sigma$.  (There is obviously nothing special here about $T_{grp}$ and this works for any classical first-order theory.)  This is called the \emph{Completeness theorem for first-order logic} and is due to Einstein's pal Kurt G\"odel.  (See \cite[Section 2.5]{enderton} for a nice treatment.)

The relevance of the Completeness theorem is that if we use some effective coding of the symbols of $L_{grp}$ and the logical symbols, then we can start a computer program  running all proofs from $T_{grp}$ and outputting all theorems of $T_{grp}$.  In other words, the language (in the sense of complexity theory) consisting of codes for theorems of group theory belongs to RE.  Note that all that was used about $T_{grp}$ is that the set of codes for axioms in $T_{grp}$ itself belongs to RE.

There are corresponding Soundness and Completeness theorems for continuous logic due to Ben-Yaacov and Pedersen \cite{BYP}.  Due to the approximate nature of continuous logic, the Completeness theorem takes a slightly different form.  Restricted to our case of interest, namely $T_{II_1}$, it reads:  for every $L_{vNa}$-sentence $\sigma$, we have
$$\sup\{\sigma^{\cal M}\ : \ \cal M \text{ a II}_1\text{ factor}\}=\inf\{r\in \bb Q^{>0} \ : \ T_{II_1}\vdash \sigma\dminus r\}.$$  Here, $\dminus$ is the function given by $r\dminus s:=\max(r-s,0)$.  Consequently, the Completeness theorem tells us that the largest truth value that $\sigma$ could take in a II$_1$ factor is the smallest upper bound for $\sigma$ that you could prove from the axioms $T_{II_1}$.  

The proof system for continuous logic is still of the form that you can effectively enumerate theorems from an effectively enumerated set of axioms.  In particular, since the set of axioms for $T_{II_1}$ given in \cite{mtoa2} is easily checked to be effectively enumerated, we see that the set of theorems of $T_{II_1}$ belongs to RE.

\subsection{CEP and the computability of the universal theory of $\R$}\label{sec7:CEP}

Given a tracial von Neumann algebra $\cal M$, we say that the universal theory of $\cal M$ is \emph{computable} if there is an algorithm such that, upon input an $L_{vNa}$-sentence $\sigma$ and a rational $\epsilon>0$, returns $a,b\in \bb Q^{>0}$ with $a<b$ and $b-a<\epsilon$ and for which $\sigma^{\cal M}\in (a,b)$.

The ideas in the previous subsection allow us to prove the following:

\begin{thm}[G. and Hart \cite{GH}]
If CEP has a positive answer, then the universal theory of $\cal R$ is computable.
\end{thm}

\begin{proof}
If CEP holds, then, recalling that $\R$ embeds into any II$_1$ factor, we have, for any universal sentence $\sigma$, that $\sup\{\sigma^M \ : \ \mathcal{M} \text{ a II}_1 \text{ factor}\}=\sigma^\R$.  Consequently, if we start enumerating all proofs from $T_{II_1}$ and record all instances of theorems of the form $\sigma\dminus r$, then we know that $\sigma^\R\leq r$ and this allows us to effectively enumerate better and better upper bounds for $\sigma^\R$.

On the other hand, we can also effectively enumerate better lower bounds for $\sigma^\R$.  There are two ways that one can go about this.  One way is to write $\sigma=\sup_x \varphi(x)$, where $\varphi(x)=f(\tau(p_1(x)),\ldots,\tau(p_n(x)))$, and each $p_i$ is a $*$-polynomial and $f$ is a ``computable'' continuous function, that is, generated from a computable set of connectives.  One can then approximately calculate $\varphi$ on matrices of larger dimensions with rational coordinates.  (Technically, $x$ is restricted to range over matrices of operator norm at most one, but one can efficiently verify this too.)  Another option is to consider the existential sentence $\sigma_0:=M_\sigma\dminus \sigma$, where $M_\sigma$ is an upper bound for $\sigma$ (uniform over all II$_1$ factors) that is effectively computable from $\sigma$ itself.  Since $\R$ embeds in every II$_1$ factor, we once again have $\sup\{\sigma_0^{\cal M} \ : \ {\cal M} \text{ a II}_1 \text{ factor}\}=\sigma_0^\cal R$ (this does not use CEP), and thus we can enumerate all proofs from $T_{II_1}$ and every time we see that $\sigma_0\dminus r$, we know that $\sigma^\cal R\geq M_\sigma-r$.

We run both the upper and lower bound algorithms simultaneously and wait until they output numbers within $\epsilon$ of each other.
\end{proof}

\subsection{Synchronous strategies, definable sets, and finishing the proof}\label{sec7:synchronous}

Based on the theorem in the previous section, in order to refute CEP, it suffices to prove the following theorem:

\begin{thm}[G. and Hart \cite{GH2}]\label{GoldHart2}
The universal theory of $\R$ is not computable.
\end{thm}

We will use $\operatorname{MIP}^*=\operatorname{RE}$ to prove this result.  But how?  Given a nonlocal game $\frak G$, the definition of $\val^*(\frak G)$ resembles a universal sentence in the language $L_{vNa}$ except that it is not a priori clear how to view the set of correlations $C_{qa}$ as something that we can quantify over in a II$_1$ factor.  

Thankfully, a specific subset of $C_{qa}$ can be characterized by a formula that our logic can handle. A strategy $p$ is called \emph{synchronous} if $p(a,b|x,x)=0$ for all $x\in [k]$ and distinct $a,b\in [n]$.  That is, $p$ is synchronous if, whenever both players are asked the same question, they always answer with the same answer.  We let $C_{qa}^s(k,n)$ (resp. $C_{qc}^s(k,n)$) denote the synchronous elements of $C_{qa}(k,n)$ (resp. $C_{qc}(k,n)$).  Given a nonlocal game $\frak G$ with $k$ questions and $n$ answers, we set its \emph{synchronous entangled value} to be $$\sval^*(\frak G):=\sup_{p\in C_{qa}^s(k,n)}\val(\frak G,p).$$  One defines the \emph{synchronous commuting value} $\sval^{co}(\frak G)$ in the obvious way.  In general, we have that $\sval^*(\frak G)\leq \val^*(\frak G)$ and $\sval^{co}(\frak G)\leq \val^{co}(\frak G)$.  .  

Paulsen et. al. \cite[Corollary 5.6]{quantum} showed that $p\in C_{qc}^s(k,n)$ if and only if there is a tracial state $\tau$ on $C^*(\bb F(k,n))$ such that $p(a,b|x,y)=\tau(e^x_ae^y_b)$.  Contrast this result with Theorem \ref{strategychar} above:  gone is the maximal tensor product of $C^*(\bb F(k,n))$ with itself, but instead the state is required to be a tracial state.  Later, Kim, Paulsen, and Schaufhauser \cite[Theorem 3.6]{KPS} showed that $p\in C_{qa}^s(k,n)$ if and only if there is an \emph{amenable} tracial state $\tau$ on $C^*(\bb F(k,n))$ such that $p(a,b|x,y)=\tau(e^x_ae^y_b)$.  One definition of an amenable tracial state $\tau$ on a \cstar-algebra $\A$ is that there is a $*$-homomorphism $\theta:\A\to \cal R^\u$ with a ucp lift $\A\to \ell^\infty(R)$ for which $\tau=\tau_{\R}^\u\circ \theta$, where $\tau_{\R}$ is the unique trace on $\R$.  There are many alternate characterizations of being an amenable trace showing that this is indeed a robust condition.  In fact, Kirchberg used the notion of amenable trace in his proof that, for finite von Neumann algebras, being QWEP is equivalent to being isomorphic to a $*$-subalgebra of $\cal R^\u$. 

In any event, the above characterization of $C_{qa}^s(k,n)$ can be used to show that $p\in C_{qa}^s(k,n)$ if and only if there are $n$-outcome PVMs $(f^x)_{x\in [k]}$ in $\R^\u$ such that $p(a,b|x,y)=\tau_\R^\u(f^x_af^y_b)$.  In the sequel, we let $X_n$ denote the set of PVMs in $\R^\u$ of length $n$.  By the previous paragraph, we have
$$\sval^*(\frak G)=\sup_{f^1,\ldots,f^k\in X_n}\left(\sum_{x,y\in [k]}\pi(x,y)\sum_{a,b\in [n]}D(x,y,a,b)\tau(f^x_af^y_b)\right)^{\R^\u}.$$  We now note two very important facts:
\begin{enumerate}
    \item The value contained in the parentheses is a legitimate first order formula evaluated in the ultrapower $\cal R^\u$ of $\cal R$.
    \item The \emph{proof} of $\MIP^*=\operatorname{RE}$ actually shows that the reduction $\bf M\mapsto \frak G_{\bf M}$ from Turing machines to nonlocal games is such that if $\bf M$ halts, then $\sval^*(\frak G_{\bf M})=1$.  
\end{enumerate}

From these two facts, it seems like the proof of Theorem \ref{GoldHart2} is complete, for if we could approximately compute the value of any universal sentence in $\cal R$, then we could approximate $\sval^*(\frak G_{\bf M})$ for any Turing machine $\bf M$ and thus be able to decide the halting problem!

There is one (not so minor) issue:  the supremum in the above display is still not technically allowable in our logic!  Indeed, we are taking the supremum over elements from a certain set $X_n$ rather than just tuples from the unit ball.  However, it turns out that the founders of continuous logic thought long and hard about such suprema and their efforts will pay off tremendously.

To explain this, we return to classical logic for one moment and the case of groups.  Given any group $G$, its center $Z(G)$ can be defined by the formula $\gamma(x):=\forall y(xy=yx)$, that is, $Z(G)=\{g\in G \ : \ \gamma(g) \text{ is true in }G\}$.  Consequently, given any formula $\theta(x)$, the formula $(\forall x\in Z(G))\theta(x)$ represents an actual sentence in classical logic, being shorthand for the more cumbersome $\forall x(\gamma(x)\rightarrow \theta(x))$. 

We are faced with a similar situation in the above paragraph.  The elements of $X_n$ are those that make the formula $$\max\left(\max_{i=1,\ldots,n}d(p_i,p_i^*),\max_{i=1,\ldots,n}d(p_i,p_i^2),d\left(\sum_{i=1}^n p_i,1\right)\right)$$ equal to $0$.  One would hope that we could thus take the supremum over this set of elements as a shorthand for a more complicated ``legitimate'' formula.  Unfortunately, such a move is not always possible.

More generally, given $L_{vNa}$-formulae $\theta(x)$ and $\psi(x)$, the expression $$\sup\{\psi(x)^{\cal R} \ : \ \theta(x)^{\cal R}=0\}$$ is only equivalent to $\sigma^{\cal R}$ for an actual $L_{vNa}$-sentence $\sigma$ if $\theta(x)$ satisfies a certain ``almost-near'' property, that is, for each $\epsilon>0$, there is a $\delta>0$ such that, for all $a\in \cal R_1$, if $\theta(a)^\R<\delta$, then there is $b\in \cal R_1$ with $\theta(b)^{\R}=0$ and $d(a,b)<\epsilon$.  In the operator algebraic literature, this is usually referred to as a \emph{weak stability} phenomena.  In the model theory literature, this is called being a \emph{definable set}.  (See the author's paper \cite{spectral} for more on definability in continuous logic and its connection to operator-algebraic matters.)

Now here is the fantastic (and fortuitious part):  the formula defining $X_n$ above does have this property!  And here's the kicker:  Kim, Paulsen, and Schaufhauser themselves proved it \cite[Lemma 3.5]{KPS} while establishing their above characterization of $C_{qa}^s(k,n)$.  Thus, we are entitled to write the above ``formula'' for $\sval^*(\frak G)$ as a shorthand for a legitimate sentence in the language of continuous logic.  There is a little bit of fine print to check, namely that the resulting sentence is in fact universal and that this transformation can be done effectively, but the details can indeed be carried out.  This completes the proof of Theorem \ref{GoldHart2} above and thus the model-theoretic proof of the negative solution to CEP.

\subsection{A G\"odelian refutation of the CEP}\label{sec7:godelian}

G\"odel's Incompleteness Theorem is one of the landmark intellectual achievements of the 20th century.  It addresses a seemingly simple question:  is there an algorithm such that, upon input a sentence in the language of number theory, returns the truth value of the sentence in the natural numbers?  Surprisingly, G\"odel proved that the answer is no \cite{godel}!  (See also \cite[Section 3.5]{enderton} for a more modern treatment.)

G\"odel actually proved something much stronger, namely he proved that any attempt to answer the previous question by giving an effective axiomatization of number theory is doomed to fail.  More specifically, there is a natural (and effective) collection of axioms known as \emph{Peano arithmetic} such that any effective extension $T$ of Peano arithmetic is destined to be incomplete, meaning there will be sentence $\sigma$ that is true in $\bb N$ but not provable from $T$.  Since simply listing all true sentences of $\bb N$ as axioms is obviously complete, it follows that the set of true sentences is not effectively enumerable.

We can use the ideas in the previous subsection to give a G\"odelian-style refutation of CEP.  Indeed, one can view CEP as the question of asking whether or not the effective list of axioms for being a II$_1$ factor is enough to axiomatize the universal theory of $\cal R$.  The failure of CEP shows that the answer to this is no.  But perhaps there is a stronger, but still effective, list of axioms extending the axioms for being a II$_1$ factor so that any model of these axioms would then satisfy the conclusion of CEP.  Our proof from the previous subsection shows that the answer is still no:

\begin{thm}[G. and Hart \cite{GH2}]
There does not exist an effectively enumerable list $T$ of axioms extending $T_{II_1}$ such that all models of $T$ are embeddable in $\R^\u$.
\end{thm}

In particular, this shows that the collection of axioms $\sigma\dminus r$ for which $\sigma^{\cal R}\leq r$ is not effectively enumerable.

The previous theorem allows us to provide ``many'' counterexamples to CEP:

\begin{cor}
There is a sequence $\cal M_1, \cal M_2,\ldots,$ of separable II$_1$ factors, none of which embed into an ultrapower of $\cal R$, and such that, for all $i<j$, $\cal M_i$ does not embed into an ultrapower of $\cal M_j$.
\end{cor}

\begin{proof}
We construct the sequence inductively.  Set $\cal M_1$ to be any separable II$_1$ factor that does not embed into an ultrapower of $\cal R$.  Suppose now that $\cal M_1,\ldots,\cal M_n$ have been constructed satisfying the conclusion of the Corollary.  For each $i=1,\ldots,n$, let $\sigma_i$ be a nonnegative sentence such that $\sigma_i^{\cal R}=0$ but $\sigma_i^{\cal M_i}>0$.  For each $i=1,\ldots,n$, fix a rational number $\delta_i\in (0,\sigma_i^{\cal M_i})$.  Let $T$ be the theory of II$_1$ factors together with the single condition $\max_{i=1,\ldots,n}(\sigma_i\dminus \delta_i)=0$.  It is clear that $T$ is an effectively enumerable subset of the theory of $\cal R$.  Thus, by the previous theorem, there is a separable model $\cal M_{n+1}$ of $T$ such that $\cal M_{n+1}$ does not embed into an ultrapower of $\cal R$.  Given $i=1,\ldots,n$, since $\sigma_i^{\cal M_i}>\delta_i$ while $\sigma_i^{\cal M_{n+1}}\leq \delta_i$, it follows that $\cal M_i$ does not embed into an ultrapower of $\cal M_{n+1}$.  This indicates how to continue the recursive construction, completing the proof.
\end{proof}

\subsection{The universal theory of $\cal R$ and the moment approximation problem}\label{sec7:universal}

In this subsection, we offer a purely operator-algebraic reformulation of the statement that $\Th_\forall(\R)$ is not computable, first proven in \cite{GH2}.

Given positive integers $n$ and $d$, we fix variables $x_1,\ldots,x_n$ and enumerate all $*$-monomials in the variables $x_1,\ldots,x_n$ of total degree at most $d$ as $m_1,\ldots,m_L$. (Of course, $L=L(n,d)$ depends on both $n$ and $d$.)  We consider the map $\mu_{n,d}:\R_1^n \rightarrow \bb D^L$ given by $\mu_{n,d}(\vec a)=(\tau(m_i(\vec a)) \ : \ i=1,\ldots,L)$.  (Here, $\bb D$ is the complex unit disk.)


We let $X(n,d)$ denote the range of $\mu_{n,d}$ and $X(n,d,p)$ be the image of the unit ball of $M_p(\mathbb C)$ under $\mu_{n,d}$.  Notice that $\bigcup_{p\in \bb N} X(n,d,p)$ is dense in $X(n,d)$.

\begin{thm}
The following statements are equivalent:
\begin{enumerate}
    \item The universal theory of $\R$ is computable.
    \item There is a computable function $F:\bb N^3\to \bb N$ such that, for every $n,d,k\in \bb N$, $X(n,d,F(n,d,k))$ is $\frac{1}{k}$-dense in $X(n,d)$.
\end{enumerate}
\end{thm}

\begin{proof}
First suppose that the universal theory of $\R$ is computable.  We produce a computable function $F$ as in (2).  Fix $n$, $d$, and $k$, and set $\epsilon:=\frac{1}{3k}$.  Computably find $s_1,\ldots,s_t$, an $\epsilon$-net in $\bb D^L$.  For each $i=1,\ldots,t$, ask the universal theory of $\R$ to compute intervals $(a_i,b_i)$ with $b_i-a_i<\epsilon$ and with $\left(\inf_{\vec x}|\mu_{n,d}(\vec x)-s_i|\right)^\R\in (a_i,b_i)$.  For each $i=1,\ldots,t$ such that $b_i<2\epsilon$, let $p_i\in \bb N$ be the minimal $p$ such that when you ask the universal theory of $ M_p(\bb C)$ to compute intervals of shrinking radius containing $\left(\inf_{\vec x}|\mu_{n,d}(\vec x)-s_i|\right)^{M_p(\bb C)}$, there is a computation that returns an interval $(c_i,d_i)$ with $d_i<2\epsilon$.  Let $p$ be the maximum of these $p_i$'s.  We claim that setting $F(n,d,k):=p$ is as desired.  Indeed, suppose that $s\in X(n,d)$ and take $i=1,\ldots,t$ such that $|s-s_i|<\epsilon$.  Then $\left(\inf_{\vec x}|\mu_{n,d} (\vec x)-s_i|\right)^\R<\epsilon$, whence $b_i<2\epsilon$.  It follows that there is an interval $(c_i,d_i)$ as above with $\left(\inf_{\vec x}|\mu_{n,d}(\vec x)-s_i|\right)^{M_p(\bb C)}<d_i<2\epsilon$.  Let $a\in M_p(\bb C)$ realize the infimum.  Then $|\mu_{n,d}(\vec a)-s|<3\epsilon=\frac{1}{k}$, as desired.

Now suppose that $F$ is as in (2).  We show that that the universal theory of $\R$ is computable.  Towards this end, fix a universal sentence 
\[
\sigma = \sup_{\vec x} f(\tau(m_1),\ldots,\tau(m_\ell))
\]
where $\vec x = x_1,\ldots,x_n$, $m_1,\ldots,m_\ell$ are *-monomials in $\vec x$ of total degree at most $d$, and $f$ is a ``computable'' connective.  Fix also rational $\epsilon>0$.  We show how to compute the value of $\sigma^\R$ to within $\epsilon$.  Since $f$ is computable, it has a ``computable modulus of continuity'' $\delta$, meaning that we can find $k\in \bb N$ computably so that $\frac{1}{k} \leq \delta(\epsilon)$.  Set $p = F(n,d,2k)$.  Computably construct a sequence $\vec a_1,\ldots,\vec a_t\in (M_p(\mathbb{C})_1)^n$ that is a $\frac{1}{2k}$ cover of $(M_p(\mathbb C)_1)^n$ (with respect to the $\ell^1$ metric corresponding to the 2-norm).  Consequently, $\mu_{n,d}(\vec a_1),\ldots,\mu_{n,d}(\vec a_t)$ is a $\frac{1}{2k}$-cover of $X(n,d,p)$.  Set 
\[
r:=\max_{i=1,\ldots,t}f(\tau(m_1(\vec a_i)),\ldots,\tau(m_l(\vec a_i))).
\]
By assumption, $X(n,d,p)$ is $\frac{1}{2k}$-dense in $X(n,d)$.  It follows that $r\leq \sigma^\R\leq r+\epsilon$, as desired.
\end{proof}

\subsection{A negative solution to Tsirelson's problem from $\MIP^*=\operatorname{RE}$}\label{sec7:negative}

In this subsection, we offer an alternative proof of the negative solution to Tsirelson's problem using $\mip^*=\operatorname{RE}$ and the Completeness Theorem; we follow closely the treatment given in \cite{GH2}.  As in the definition of the synchronous entangled value $\sval^*(\frak G)$ of a nonlocal game $\frak G$ with $k$ questions and $n$ answers, we define its \emph{synchronous commuting value} to be $\sval^{co}(\frak G):=\sup_{p\in C_{qc}^s(k,n)}\val(\frak G,p)$.

\begin{defn}
Fix $0<r\leq 1$.  We define $\mip^{co,s}_{0,r}$ to be the set of those languages $\bf L$ for which there is an efficient mapping $z\mapsto \mathfrak G_z$ from strings to nonlocal games such that $z\in \bf L$ if and only if $\sval^{co}(\mathfrak G_z)\geq r$.
\end{defn}

\begin{thm}\label{mipcos}[G. and Hart \cite{GH2}]
For any $0<r\leq 1$, every language in $\mip^{co,s}_{0,r}$ belongs to the complexity class coRE.
\end{thm}

In other words, if $\mathbf{L}\in \mip^{co,s}_{0,r}$, then there is an algorithm which enumerates the complement of $\bf L$.


For the remainder of this subsection, we work in the first-order language $L_{\tau C^*}$ for tracial \cstar-algebras, that is, \cstar-algebras equipped with a distinguished tracial state, which is defined in a manner analogous to the language $L_{vNa}$ used to study tracial von Neumann algebras.  Fix a nonlocal game $\frak G$ with $k$ questions and $n$ answers.  Let $w=(w_{x,a})_{x\in [k],a\in [n]}$ denote a tuple of variables and $\psi_\mathfrak G(w)$ be the $L_{\tau C^*}$-formula $$\sum_{(x,y)\in[k]\times [k]}\pi(x,y)\sum_{(a,b)\in[n]\times[n]}D(x,y,a,b)\tau(w_{x,a}w_{y,b}),$$ which informally calcualates the expected value of winning when playing according to a strategy from $C^s_{qc}$ (recalling the Paulsen et al. characterization of elements of $C^s_{qc}$).  We then let $\theta_{\mathfrak G,r}$ be the $L_{\tau C^*}$-sentence 
\[
\inf_w\max\left(\max_{x,a}(\|w_{x,a}^2-w_{x,a}\|,\max_{x,a}\|w_{x_a}^*-w_{x,a}\|,\max_x\|\sum_i w_{x,a}-1\|,r\dotminus \psi_\mathfrak G(w)\right),
\]

which informally calculates $\sval^{co}(\frak G)$.

Let $T_{\tau C^*}$ be the $L_{\tau C^*}$-theory of tracial C*-algebras.  The following is immediate:

\begin{prop}\label{vern2}
For any nonlocal game $\mathfrak G$, we have $\sval^{co}(\mathfrak G)\geq r$ if and only if the theory $T_{\tau C^*}\cup\{\theta_{\mathfrak G,r}=0\}$ is \emph{satisfiable}, that is, has a model.
\end{prop}

We will also need the following immediate consequence of the 
Completeness Theorem:

\begin{lem}\label{con}
Let $U$ be a continuous theory.  Then $U$ is satisfiable if and only if $U\not\vdash \bot$\footnote{$\bot$ represents a contradiction i.e. any continuous sentence which cannot evaluate to 0.  For instance, the constant function 1.}.
\end{lem}

We can now prove Theorem \ref{mipcos}.  Let $\bf L$ belong to $\mip^{co,s}_{0,r}$ and consider a string $z\notin \bf L$, with corresponding game $\mathfrak G_z$.  By Proposition \ref{vern2} and Lemma \ref{con}, we have that $T_{\tau C^*}\cup\{\theta_{\mathfrak G_z,r}=0\} \vdash \bot.$
 Since this latter condition is recursively enumerable, the proof of Theorem \ref{mipcos} is complete.  


One can now deduce the failure of Tsirelson's problem from MIP*=RE as follows.  Suppose, towards a contradiction, that $C_{qa}^s(k,n)=C_{qc}^s(k,n)$ for every $k$ and $n$.  Let $\bf M\mapsto \mathfrak G_{\bf M}$ be the efficient mapping from Turing machines to nonlocal games provided by $\mip^*=\operatorname{RE}$.  Given a Turing machine $\bf M$, one simultaneously starts computing lower bounds on $\val^*(\mathfrak G_{\bf M})$ while running proofs from $T_{\tau C^*}\cup\{\theta_{\mathfrak G_{\bf M},1}=0\}$.  Since we are assuming that $C_{qa}^s(k,n)=C_{qc}^s(k,n)$ (where $k$ and $n$ are the number of questions and answers of $\frak G_z$), we have that either the first computation eventually yields the fact that $\val^*(\mathfrak G_{\cal M})>\frac{1}{2}$, in which case $\cal M$ halts, or else the second computation eventually yields the fact that $T_{\tau\text{\cstar}}\cup\{\theta_{\mathfrak G_{\bf M},1}=0\}\vdash \bot$,
in which case $\sval^*(\mathfrak G_{\bf M})<1$, and $\bf M$ does not halt.  In this way, we can decide the halting problem, a contradiction.  Note that we derived the a priori stronger statement that $\cqa^s(k,n)\not=\cqc^s(k,n)$ for some $k$ and $n$.

\section{The enforceable II$_1$ factor (should it exist)}\label{sec8}

In this section, we describe a model-theoretic weakening of CEP, namely the statement that the enforceable II$_1$ factor exists.  In order to explain this statement and its connection to CEP, we first need to introduce a certain two-player game.

\subsection{A different kind of game}\label{sec8:different}

We introduce a method for building tracial von Neumann algebras first introduced in \cite{GoldGames} for an arbitrary structure in continuous logic (based on the discrete case presented in Hodges' book \cite{hodges}).  This method goes under many names, such as \textit{Henkin constructions}, \textit{model-theoretic forcing}, or \textit{building models by games}.

We fix a countably infinite set $C$ of distinct symbols that are to represent generators of a separable tracial vNa that two players (traditionally named $\forall$ and $\exists$) are going to build together (albeit adversarially).
The two players take turns playing finite sets $\Sigma$ of expressions of the form $\left|\|p(c)\|_{\tau}-r\right|<\epsilon$, where $c$ is a tuple of variables from $C$, $p(x)$ is a $*$-polynomial, and each player's move is required to extend (that is, contain) the previous player's move.  These sets are called (open) \emph{conditions}.  The game begins with $\forall$'s move.  Moreover, these conditions are required to be \emph{satisfiable}, meaning that there should be some tracial von Neumann algebra $\cal M$ and some tuple $a$ from $\cal M_1$ such that $\left|\|p(a)\|_{\tau}-r\right|<\epsilon$ for each such expression in the condition.  We play this game for countably many rounds.
At the end of this game, we have enumerated some countable, satisfiable set of expressions. Provided that the players address a ``dense'' set of moments infinitely often, they can ensure that the play is \emph{definitive}, meaning that the final set of expressions yields complete information about all $*$-polynomials over the variables $C$ (that is, for each $*$-polynomial $p(x)$ and each tuple $c$ from $C$, there should be a unique $r$ such that the play of the game implies that $\|p(c)\|_{\tau}=r$) and that this data describes a countable, dense $*$-subalgebra of a unique tracial von Neumann algebra, which is called the \textit{compiled structure}.  In what follows, we assume all plays of the game are definitive.
\subsection{Enforceable properties of tracial von Neumann algebras}\label{sec8:enforceable}

Crucial to the connection between the above games and the CEP is the notion of an enforceable property:

\begin{defn}
Given a property $P$ of tracial von Neumann algebras, we say that $P$ is an \textit{enforceable} property if there a strategy for $\exists$ so that, regardless of player $\forall$'s moves, if $\exists$ follows the strategy, then the compiled structure will have property $P$.
\end{defn}

Perhaps being an enforceable property seems so severe that there are in fact no enforceable properties.  We will soon see that many interesting properties are in fact enforceable.  First, we mention the \emph{Conjunction lemma} \cite[Lemma 2.4]{GoldGames}: If $P_n$ is an enforceable property for each $n\in \mathbb N$, then so is the conjunction $\bigwedge_n P_n$.

As a first example of an enforceable property of tracial von Neumann algebras, we show that being a factor is enforceable.  To see this, let $\theta(x)$ be the $L_{vNa}$-formula $\sqrt{\|x\|_{\tau}^2-\tau(x)^2}$ and let $\eta(x)$ be the $L_{vNa}$-formula $\sup_y\|xy-yx\|_{\tau}$.  Finally, let $\sigma$ be the $L_{vNa}$-sentence $\sup_x(\theta(x)\dminus \eta(x))$.  It was shown in \cite{mtoa2} that a von Neumann algebra $\cal M$ is a factor if and only if $\sigma^{\cal M}=0$.  To see that being a factor is enforceable, by the Conjunction Lemma, it suffices to show that, given any $n\in \bb N$ and rational $\epsilon>0$, the expression $(\theta(c_n)\dminus \eta(c_n))<\epsilon$ is enforceable.  To see that this is the case, suppose that player $\forall$ opened the game with the open condition $\Sigma$.  Without loss of generality, we may suppose that $c_n$ appears in $\Sigma$.  Since $\Sigma$ is satisfiable in some tracial von Neumann algebra, it is also satisfiable in some II$_1$ factor $\cal M$ (as every tracial von Neumann algebra embeds in a II$_1$ factor).  Consequently, since $\sigma^{\cal M}=0$, we see that $\cal M$ also witnesses that $\Sigma\cup\{\theta(c_n)\dminus \eta(c_n)<\epsilon\}$ is a condition, whence $\exists$ can respond with this condition, as desired.  

We next show that being a II$_1$ factor is enforceable.  Since being a factor is enforceable, it suffices to show that it is enforceable that, in the compiled structure, there is a projection of irrational trace (say $\frac{1}{\pi}$).  By the Conjunction Lemma again, it suffices to show that, for any rational $\epsilon>0$, there is some $n\in \bb N$ for which $\max(d(c_n,c_n^*),d(c_n,c_n^2),|\tau(c_n)-\frac{1}{\pi}|)<\epsilon$ is enforceable.  Indeed, if this is enforceable, then since ``almost'' projections are near actual projections, there will be actual projections in the compiled structure whose trace approaches $\frac{1}{\pi}$, whence there will be an actual projection of trace $\frac{1}{\pi}$ as desired.  However, this condition is clearly enforceable by the exact same argument used in the previous paragraph, this time, using a ``fresh'' constant, that is, some $c_n$ which did not appear in player $\forall$'s opening play $\Sigma$.


One can go even further and show that being a \emph{McDuff} II$_1$ factor is enforceable.  A II$_1$ $\cal M$ factor is McDuff if $\cal M\bar\otimes \cal R\cong \cal M$.  For example, $\cal R$ is McDuff.  An alternate formulation for being McDuff will prove useful:  $\cal M$ is McDuff if and only if there is a copy of $M_2(\bb C)$ inside of $\cal M'\cap \cal M^\u$.  This amounts to showing that:  for any finite $\cal F\subseteq \cal M$ and any rational $\epsilon>0$, there are matrix units $(e_{ij})_{i,j=1,2}$ for $M_2(\bb C)$ for which $\|[x,e_{ij}]\|_{\tau}<\epsilon$ for all $x\in \cal F$ and all $i,j=1,2$.  Hopefully by now the strategy is apparent:  given any open play $\Sigma$ for player $\forall$, we realize $\Sigma$ in some tracial von Neumann algebra $\cal M$.  We then note that $\Sigma$ is also realized in $\cal M\bar \otimes M_2(\bb C)$ and then choose fresh constants $c_{n_1},\ldots,c_{n_4}$ and say that they are ``almost'' matrix units for $M_2(\bb C)$ which almost commute with $c_1,\ldots,c_n$.  As we let $n$ increase and $\epsilon$ decrease and using the fact that ``almost'' matrix units are near actual matrix units, the result follows using the Conjunction Lemma.

\subsection{Existentially closed tracial von Neumann algebras (and yet another reformulation of CEP)}\label{sec8:existentially}

One can push the line of reasoning in the previous subsection much further.  First, it is helpful to introduce the notion of an \emph{existentially closed (e.c.) tracial von Neumann algebra}.  A tracial von Neumann algebra $\cal M$ is e.c.\ if:  whenever $\cal M\subseteq \cal N$, there is an embedding $\cal N\hookrightarrow \cal M^\u$ into some ultrapower of $\cal M$ that restricts to the diagonal embedding of $\cal M$ into its ultrapower.  If $\cal N$ is separable (whence so is $\cal M$), then this is equivalent to the above definition where we can use any nonprincipal ultrapower of $\cal M$.  This version of the definition is the semantic version.  Syntactically, $\cal M$ is e.c.\ if:  for any existential formula $\varphi(x)$ (where $x$ is a finite tuple of variables), any $a\in \cal M_1$, and any tracial von Neumann algebra $\cal N$ containing $\cal M$, we have $\varphi(a)^{\cal M}=\varphi(a)^{\cal N}$.  In other words, any phenomena that ``could happen'' in an extension of $\cal M$ approximately also happens in $\cal M$.  It is for this reason that one should think of an e.c.\ tracial von Neumann algebra as being an analog of an algebraically closed field.

E.c.\ tracial von Neumann algebras appear in abundance.  Indeed, any tracial von Neumann algebra embeds into an e.c.\ one of the same density character.  Moreover, we know many properties of an e.c.\ tracial von Neumann algebra:  they must be McDuff II$_1$ factors, all of their automorphisms must be approximately inner, etc...  The reader interested in learning more about e.c.\ tracial von Neumann algebras can consult \cite{FGHS}, \cite{spectral}, and \cite{GHS}.

But can we name a concrete e.c.\  tracial von Neumann algebra?  Well:

\begin{thm}[Farah, G., Hart, and Sherman \cite{FGHS}]
$\cal R$ is an e.c.\ tracial von Neumann algebra if and only if CEP has a positive solution.
\end{thm}

\begin{proof}
We first note that if $\cal R$ is e.c., then CEP holds:  given a II$_1$ factor $\cal M$, we have that $\cal R\subseteq \cal M$, whence, since $\cal R$ is e.c., we have that $\cal M$ embeds into $\cal R^\u$.  Conversely, suppose that CEP holds; we show that $\cal R$ is e.c.  To see this, suppose that $\cal R\subseteq \cal M$ with $\cal M$ separable.  By CEP, $\cal M$ embeds into $\cal R^\u$.  At the moment, this does not imply that $\cal R$ is e.c.\ as the composed embedding $\cal R\hookrightarrow \cal R^\u$ need not be the diagonal embedding.  However, a nontrivial result of Kenley Jung \cite{jung} implies that every embedding $\pi:\cal R\hookrightarrow \cal R^\u$ is unitarily conjugate to the diagonal embedding, meaning that there is a unitary element $u\in \cal R^\u$ such that $\pi(a)=uau^*$ for all $a\in \cal R$ (viewing $\R$ as literally a subalgebra of $\R^\u$ via the diagonal embedding).  It is straightforward to check that this finishes the job.
\end{proof}

Following \cite{mtoa3}, we call a tracial von Neumann algebra $\cal M$ \emph{locally universal} if every tracial von Neumann algebra embeds into an ultrapower of $\cal M$.  In this terminology, CEP asks if $\cal R$ is locally universal.  The proof of the previous theorem shows the following:  

\begin{thm}
Every e.c.\ tracial von Neumann algebra is locally universal.  In particular, locally universal tracial von Neumann algebras exist.
\end{thm}

The latter conclusion was first reached (using a different argument) in \cite[Example 6.4]{mtoa3} and was referred to as a resolution to the \emph{Poor Man's Connes Embedding Problem}.  Note also that CEP holds if and only if any locally universal tracial von Neumann algebra embeds in $\R^\u$.

Another important fact for us is the following; see \cite[Proposition 2.10]{GoldGames} for a proof:

\begin{thm}
Being an e.c.\ tracial von Neumann algebra is an enforceable property.
\end{thm}

The proof of the previous theorem is a more elaborate version of the arguments given in the last section.

One might ask:  is there some first-order way of axiomatizing the e.c.\ tracial von Neumann algebras?  The answer is no, a result first proven by Hart, Sinclair, and the author in \cite{GHS} although we now know of some more elementary proofs (see \cite[Corollary 5.19]{spectral} for example).

\subsection{CEP and enforceability}\label{sec8:CEP}

As we have seen in the previous subsections, while enforceability of a property seemed like it shoud rarely happen, we actually know of many interesting properties that are in fact enforceable.  We now consider a real extreme version of this:

\begin{defn}
A tracial von Neumann algebra $\cal M$ is said to be \emph{enforceable} if the property of being isomorphic to $\cal M$ is an enforceable property.
\end{defn}

Clearly, if an enforceable tracial von Neumann algebra exists, then it is unique.  Enforceable structures do exist in many other contexts.  For example, the enforceable graph is the \emph{random} or \emph{Rado} graph and the enforceable field of a particular characteristic is the algebraic closure of the prime field.  (Note, however, that the enforceable group does not exist; while somewhat implicit in \cite{hodges}, this is made explicit in \cite{GKL}.)  On the analytic side, we have that the enforceable metric space is the \emph{Urysohn space} \cite{usvyatsov}, the unique Hilbert space of dimension $\aleph_0$ is the enforceable Hilbert space \cite[Section 15]{mtfms}, and the enforceable Banach space is the \emph{Gurarij Banach space} \cite{BYH}.

So what about the enforceable tracial von Neumann algebra?  Here is the connection to CEP:

\begin{thm}[G. \cite{GoldGames}]  The following statements are equivalent:
\begin{enumerate}
    \item CEP has a positive solution.
    \item The property of being hyperfinite is enforceable.
    \item $\cal R$ is the enforceable tracial von Neumann algebra.
    \item The property of being embeddable in $\R^\u$ is enforceable.
\end{enumerate}
\end{thm}

\begin{proof}
(1) implies (2):  By CEP, every open condition is satisfied in $\R^\u$ and hence in $\R$.  Thus, given any $n$ and rational $\epsilon>0$, if player $\forall$ opens with $\Sigma$, then $\Sigma$ is satisfied in $\R$ and thus $c_1,\ldots,c_n$ are all within $\epsilon$ in $\|\cdot\|_\tau$ of some finite linear combination of approximate matrix units for some sufficiently large matrix algebra.  Thus, player $\exists$ can respond with this extension of $\Sigma$.  Now apply the Conjunction Lemma.

(2) implies (3):  If being hyperfinite is enforceable, then since being a II$_1$ factor is also enforceable, we see by the Conjunction Lemma that being a hyperfinite II$_1$ factor is enforceable, whence $\R$ itself is enforceable.

(3) implies (4) is trivial.  For (4) implies (1), if being embeddable in $\R^\u$ is enforceable, then since being e.c.\ is also enforceable, we see that there is an e.c.\ tracial von Neumann algebra that embeds in $\R^\u$.  Since this e.c.\ tracial von Neumann algebra is necessarily locally universal, by the observation made in the previous subsection, we have that CEP has a positive solution.
\end{proof}

Now that we know that CEP has a negative solution, we see that no e.c.\ tracial von Neumann algebra embeds in $\R^\u$.  Since being e.c.\ is enforceable, we see that the situation is pretty dire:  it is enforceable that the compiled structure does not embed in $\R^\u$, which should be seen as a ``generic'' negative solution to the CEP.

\subsection{Properties of the enforceable II$_1$ factor (again, should it exist)}\label{sec8:properties}

Now that we know that CEP has a negative solution, we know that $\cal R$ is not enforceable.  But there is still the possibility that the enforceable tracial von Neumann algebra $\cal E$ (which must necessarily be a II$_1$ factor) exists.  It is this author's humble opinion that the existence of the enforceable II$_1$ factor is one of the most interesting open problems in the model theory of operator algebras.  Indeed, if $\cal E$ exists, then it rivals $\cal R$ for being the most ``canonical'' II$_1$ factor.  On the other hand, if $\cal E$ does not exist, then this can be seen as a strong negative solution to the CEP.

We first mention a theorem that might help us figure out whether or not it exists; see \cite{GoldGames} for a proof:

\begin{thm}[Dichotomy theorem]
Exactly one of the following two conditions holds:
\begin{enumerate}
    \item For every enforceable property $P$ of tracial von Neumann algebras, there exist continuum many nonisomorphic separable tracial von Neumann algebras with property $P$.
    \item The enforceable II$_1$ factor $\cal E$ exists.
\end{enumerate}
\end{thm}

Consequently, one strategy for showing that $\cal E$ does exist is to find some enforceable property $P$ such that fewer than continuum many tracial von Neumann algebras have property P.

On the other hand, in order to prove that $\cal E$ does not exist, it might prove useful to analyze some of its properties (should it exist).  As mentioned above, being e.c.\  is an enforceable property and thus $\cal E$, if it exists, has all of the properties common to e.c. factors, such as being McDuff and having only approximate inner automorphisms.  Moreover, as shown in \cite[Section 6]{GoldGames}, $\cal E$ would embed into every e.c.\ factor, which is reminiscent of the situation that $\cal R$ embeds into every II$_1$ factor.  

Recalling that $\R$ has the McDuff property, we see that $\cal R\bar \otimes \cal R\cong \cal R$.  However, one can show that if $\cal E$ exists, then $\cal E\bar\otimes \cal E\not\cong \cal E$.  Indeed, it is possible to show that if the property of being isomorphic to $\cal M\bar\otimes \cal M$ for some II$_1$ factor $\cal M$ is enforceable, then CEP holds (see \cite[Remark 5.8]{GoldGames}).  Thus, $\cal E\not\cong \cal M\bar\otimes \cal M$ for any tracial von Neumann algebra $\cal M$.

The theorem of Jung mentioned above states that every embedding of $\cal R$ into $\cal R^\u$ is unitarily conjugate to the diagonal embedding.  We say that a II$_1$ factor $\cal M$ has the \emph{Jung property} if every embedding of $\cal M$ into its ultrapower $\cal M^\u$ is unitarily conjugate to the diagonal embedding.  Atkinson and Kunnawalkam Elayavalli \cite{AK} showed that $\cal R$ is the only $\R^\u$-embeddable factor with the Jung property.  However, in \cite{GoldResembles}, we showed that $\cal E$, if it exists, also has the Jung property.  One can use this fact to show that $\cal E$, should it exist, cannot even be \emph{elementarily equivalent} to $\cal E\bar\otimes \cal E$, meaning that there must be some $L_{vNa}$-sentence $\sigma$ such that $\sigma^{\cal E}\not=\sigma^{\cal E\bar\otimes \cal E}$! 

\end{document}